\documentclass[10pt, reqno]{amsart}
\usepackage{textcmds}  
\usepackage{amsmath, amssymb, amsfonts, amstext, verbatim, amsthm, mathrsfs}
\usepackage{microtype}
\usepackage[all]{xy}
\usepackage[modulo]{lineno}
\usepackage[usenames]{color}
\usepackage{aliascnt}
\usepackage{enumitem}
\usepackage{xspace}
\usepackage{amsfonts}
\usepackage{amssymb}
\usepackage[centertags]{amsmath}
\usepackage{amsthm}
\usepackage[margin=3cm]{geometry}
\usepackage{dsfont}
\usepackage{bm}
\usepackage{color}
\usepackage{subfigure}
\usepackage{amsmath}
\usepackage{array}
\usepackage[all]{xy}
\usepackage{euscript}

\usepackage{graphics,graphicx}  

\let\oldtocsection=\tocsection

\let\oldtocsubsection=\tocsubsection

\renewcommand{\tocsection}[2]{\hspace{0em}\oldtocsection{#1}{#2}}
\renewcommand{\tocsubsection}[2]{\hspace{1em}\oldtocsubsection{#1}{#2}}

\usepackage[backref,colorlinks=true,linkcolor=black,citecolor=blue,urlcolor=blue,citebordercolor={0 0 1},urlbordercolor={0 0 1},linkbordercolor={0 0 1}]{hyperref} 

\makeatletter
\def\@secnumfont{\bfseries}
\def\section{\@startsection{section}{1}%
  \z@{.7\linespacing\@plus\linespacing}{.5\linespacing}%
  {\normalfont\Large\bfseries}}
\def\subsection{\@startsection{subsection}{2}%
  \z@{.5\linespacing\@plus.7\linespacing}{-.5em}%
  {\normalfont\large\bfseries}}
  \def\subsubsection{\@startsection{subsubsection}{3}%
  \z@{.5\linespacing\@plus.7\linespacing}{-.5em}%
  {\normalfont\bfseries}}
\makeatother

\theoremstyle{plain}
\newtheorem{thm}{Theorem}[section]
\newtheorem{thm*}[thm]{Theorem*}
\newtheorem{lemma}[thm]{Lemma}
\newtheorem{prop}[thm]{Proposition}

\theoremstyle{definition}

\newtheorem{definition}[thm]{Definition}
\newtheorem{remark}[thm]{Remark}

\theoremstyle{remark}

\numberwithin{equation}{section} 	

\newcommand{\R}{\mathbb{R}}
\newcommand{\Z}{\mathbb{Z}}
\newcommand{\C}{\mathbb{C}}

\newcommand{\PP}{\mathbb{P}}
\newcommand{\K}{\mathbb{K}}

\newcommand{\La}{\Lambda}
\newcommand{\la}{\lambda}
\newcommand{\lam}{\lambda}

\newcommand{\e}{\varepsilon}
\newcommand{\eps}{\e}

\newcommand{\op}{\operatorname}

\newcommand{\wt}{\widetilde}
\newcommand{\wh}{\widehat}

\newcommand{\std}{\text{std}}

\newcommand{\sss}{\vspace{2.5 mm}}
\newcommand{\bcs}{\natural}

\newcommand{\Int}{\text{Int}\,}

\newcommand{\bdy}{\partial}
\newcommand{\D}{\mathbb{D}}
\newcommand{\calL}{\mathcal{L}}
\newcommand{\calD}{\mathcal{D}}
\newcommand{\univ}{\text{univ}}

\newcommand{\nil}{\emptyset}
\newcommand{\Op}{\mathcal{O}p\,}
\newcommand{\calR}{\mathcal{R}}
\newcommand{\calM}{\mathcal{M}}
\newcommand{\calMovl}{\ovl{\calM}}

\newcommand{\calN}{\mathcal{N}}
\newcommand{\calP}{\mathcal{P}}
\newcommand{\calA}{\mathcal{A}}
\newcommand{\calS}{\mathcal{S}}
\newcommand{\hf}{HF}
\newcommand{\cf}{CF}
\newcommand{\sh}{SH}
\newcommand{\hw}{HW}
\newcommand{\twsh}{{SH}_{\Omega}}

\newcommand{\twhwblank}{{HW}}

\newcommand{\ovl}{\overline}

\newcommand{\calJ}{\mathcal{J}}
\newcommand{\fuk}{\mathfrak{Fuk}}
\newcommand{\wfuk}{\mathfrak{WFuk}}
\newcommand{\fukdir}{\fuk^{\rightarrow}}

\newcommand{\twfuk}{\mathfrak{Fuk}_{\Omega}}
\newcommand{\twfukblank}{\mathfrak{Fuk}}
\newcommand{\conf}{\text{Conf}}
\newcommand{\pslr}{\text{PSL}(2,\mathbb{R})}
\newcommand{\pslc}{\text{PSL}(2,\mathbb{C})}
\newcommand{\hol}{\text{Hol}}
\newcommand{\reg}{\text{reg}}
\newcommand{\calF}{\mathcal{F}}
\newcommand{\ev}{\text{ev}}
\newcommand{\ec}{\e_c}
\newcommand{\calH}{\mathcal{H}}
\newcommand{\calC}{\mathcal{C}}
\newcommand{\tot}{\text{tot}}
\newcommand{\twmu}{\mu_{\Omega}}

\newcommand{\calI}{\mathcal{I}}
\newcommand{\calB}{\mathcal{B}}
\newcommand{\calE}{\mathcal{E}}

\newcommand{\calT}{\mathcal{T}}

\newcommand{\gen}{\op{Gen}}
\newcommand{\genHLL}[2]{\gen(#1,#2)}

\newcommand{\calJctct}{\mathcal{J}_\text{ctct}}

\newcommand{\Hhom}{\text{Hhom}}
\newcommand{\hhom}{H\!\hom}
\newcommand{\fun}{\text{fun}}

\newcommand{\bdfuk}{\mathfrak{Fuk}_{i_{\mho}}}

\newcommand{\oo}{\mathfrak{0}}
\newcommand{\nn}{\mathfrak{1}}
\newcommand{\bL}{\mathbf{L}}
\newcommand{\bx}{\mathbf{x}}

\newcommand{\bimod}{bimod}
\newcommand{\euA}{\EuScript{A}}
\newcommand{\euP}{\EuScript{P}}
\newcommand{\euT}{\euA_{\op{tot}}}
\newcommand{\euB}{\EuScript{B}}
\newcommand{\euU}{\EuScript{U}}
\newcommand{\euD}{\EuScript{D}}

\newcommand{\fs}{\calF}
\newcommand{\identity}{\mathds{1}}
\newcommand{\iv}{IV}
\newcommand{\tv}{TV}
\newcommand{\LL}{\mathfrak{L}}
\newcommand{\Jo}{J_{\op{ref}}}

\newcommand{\KL}{\mathbb{L}}
\newcommand{\field}{\mathbb{F}}

\newcommand{\calHtau}{\calH_{\tau}}
\newcommand{\calcyl}{\calC}
\newcommand{\calcylovl}{\ovl{\calC}}
\newcommand{\calcyluniv}{\calC^\univ}
\newcommand{\calcont}{\calD}
\newcommand{\calcontovl}{\ovl{\calcont}}
\newcommand{\calcontuniv}{\calD^\univ}
\newcommand{\signu}{s(u)}

\newcommand{\calRovl}{\ovl{\calR}}
\newcommand{\calRuniv}{\calR^{\univ}}
\newcommand{\calRunivpre}{\wt{\calR}^\univ}
\newcommand{\calRovluniv}{\ovl{\calR}^{\univ}}
\newcommand{\calTR}[1]{\calT_{\calR_{#1}}}
\newcommand{\calTRsemi}[1]{\calT_{\calR_{#1}}^{\op{semi}}}

\newcommand{\calJfix}{\calJ}
\newcommand{\calsphere}{\mathcal{N}}
\newcommand{\calsphereovl}{\ovl{\mathcal{N}}}
\newcommand{\calsphereuniv}{\calsphere^\univ}

\newcommand{\CP}{\mathbb{CP}}

\newcommand{\calTsphere}[1]{\calT_{\calsphere_{#1}}}

\newcommand{\calTcyl}[1]{\calT_{\calcyl_{#1}}}
\newcommand{\calTcylsemi}[1]{\calT_{\calcyl_{#1}}^{\op{semi}}}
\newcommand{\calTcont}[1]{\calT_{\calcont_{#1}}}
\newcommand{\calTcontsemi}[1]{\calT_{\calcont_{#1}}^{\op{semi}}}

\newcommand{\ii}{{int}}
\newcommand{\rd}{{rd}}
\newcommand{\pl}{{pl}}
\newcommand{\ipl}{{\ii,\pl}}
\newcommand{\ird}{{\ii,\rd}}
\newcommand{\spr}{{spr}}

\newcommand{\visphere}{V_\ii(\calT_\calN)}
\newcommand{\vipR}{V_{\ii,\pl}(\calT_\calR)}
\newcommand{\virR}{V_{\ii,\rd}(\calT_\calR)}

\newcommand{\vicyl}{V_\ii(\calT_{\calcyl})}
\newcommand{\vipcyl}{V_\ipl(\calT_{\calcyl})}
\newcommand{\vircyl}{V_\ird(\calT_{\calcyl})}
\newcommand{\vicont}{V_\ii(\calT_{\calcont})}
\newcommand{\vipcont}{V_\ipl(\calT_{\calcont})}
\newcommand{\vircont}{V_\ird(\calT_{\calcont})}

\newcommand{\orien}{\op{or}}

\renewcommand{\ker}{\op{Ker}}

\newcommand{\End}{\op{End}}

\title{Squared Dehn twists and deformed symplectic invariants}
\author{Kyler Siegel}
\thanks{The author was partially supported by NSF grant DGE-114747.}

\begin{document}

\sloppy

\begin{abstract}
We establish an infinitesimal version of fragility for squared Dehn twists around even dimensional Lagrangian spheres. The precise formulation involves twisting the Fukaya category by a closed two-form or bulk deforming it by a half-dimensional cycle.
As our main application, we compute the twisted and bulk deformed symplectic cohomology of the subflexible Weinstein manifolds constructed in \cite{murphysiegel}.
\end{abstract}

\maketitle

\tableofcontents

\section{Introduction}

The classical Dehn twist is a certain self-diffeomorphism of the annulus which is the identity near the boundary and the antipodal map on the core circle.
It is well-appreciated that Dehn twists around curves play a central role in the study of surfaces. Among other things they generate the mapping class group of any closed orientable surface and therefore form a basic building block for their automorphisms.
The Dehn twist also has a natural generalization to higher dimensions, sometimes called the ``generalized Dehn twist" or ``Dehn--Seidel twist" or ``Picard--Lefschetz transformation".
It is a self-diffeomorphism of the unit disk cotangent bundle of the sphere, $D^*S^n$, which is the identity near the boundary and the antipodal map on the zero section.
In fact, it was first observed by Arnold \cite{arnold1995some} that generalized Dehn twists are symplectomorphisms which respect to the canonical symplectic structure on $D^*S^n$.
Moreover, by the Weinstein neighborhood they can be implanted into a neighborhood of any Lagrangian sphere $S$ in a symplectic manifold $(M^{2n},\omega)$. 
The resulting symplectomorphism $\tau_S: M \rightarrow M$ is well-defined up to symplectic isotopies fixing the boundary.
Following in their two-dimensional fraternal footsteps,
higher dimensional Dehn twists have recently become a major element of symplectic geometry.
Besides being interesting automorphisms in their own right, they provide powerful tools for computations in Floer theory and Fukaya categories.
Dehn twists also arise as monodromy transformations around critical values of Lefschetz fibrations, allowing for deep connections with singularity theory and algebraic geometry.

The classical Picard--Lefschetz formula describes the action of a generalized Dehn twist on the level of singular homology. Using it, one can easily check that the iterates of a Dehn twist around an odd-dimensional sphere are typically distinct, even on the level of homotopy theory.
On the other hand, Dehn twisting twice around an even dimensional sphere acts trivially on homology.
If fact, at least if the sphere is two-dimensional, the squared Dehn twist is known to be smoothy isotopic to the 
identity rel boundary \cite{seidel1999}.
A similar proof seems to work for six-dimensional spheres (see the discussion in \cite[\S 5.3]{maydanskiy2009}), and for general even-dimensional spheres it is known that some finite iterate of the Dehn twist is smoothly isotopic to the identity \cite{krylov2007relative}, although the precise order is unknown.
At any rate, Seidel realized that $\tau_S^2$ is typically {\em not} symplectically isotopic to the identity rel boundary, and this can be detected using Floer theory \cite{seidel1997floer}.
This is an example of a rigidity phenomenon in symplectic geometry which goes far beyond smooth topology.

\sss

In \cite{seidel1997floer}, Seidel also made the intriguing observation that, for $S$ two-dimensional, $\tau_S^2$ is typically {\em fragile}. That is, although $\tau_S^2: M \rightarrow M$ is not symplectically isotopic to the identity rel boundary, there exist symplectic forms $\omega_t$, $t \in [0,1]$, and symplectomorphisms $\Phi_t$ of $(M,\omega_t)$ fixing the boundary such that $(\omega_0,\Phi_0) = (\omega,\tau_S^2)$ and $\Phi_t$ is symplectically isotopic to the identity rel boundary for any $t > 0$.
This observation seems to suggest that symplectic rigidity is more delicate than one could reasonably guess.

As a step in interpreting this phenomenon, Ritter observed in \cite{ritter2010deformations} that deforming a symplectic form is at least heuristically related to twisting symplectic invariants by a closed two-form. Roughly, for $\Omega$ a sufficiently small closed two-form, the symplectic geometry of $(M,\omega + \Omega)$ ought to be reflected in the $\Omega$-twisted symplectic invariants of $(M,\omega)$.
As a manifestation of this, in $\S\ref{sec:fukaya categories}$ we construct the Fukaya category of a Liouville domain twisted by a closed two-form $\Omega$, denoted by $\twfuk(M,\theta)$,
and in \S\ref{sec:squared dehn twists} we prove:
\begin{thm}\label{thm:main thm twisted}
Let $L$ and $S$ be Lagrangian spheres in a four-dimensional Liouville domain $(M^4,\theta)$.
Assume that $L$ and $S$ intersect once transversely, and let $\Omega$ be a real closed two-form on $M$ 
such that $\int_L \Omega = 0$ and $\int_S \Omega \neq 0$.
Then $L$ and $\tau_S^2 L$ are not quasi-isomorphic in $\fuk(M,\theta)$, but are quasi-isomorphic in $\twfuk(M,\theta)$.
\end{thm}
\noindent This result can be interpreted as an infinitesimal analogue of the fragility of $\tau_S^2$.
It roughly states that, although $\tau_S^2$ is not symplectically isotopic to the identity, it behaves like the identity in the presence of the twisting two-form $\Omega$. 

\sss

At first glance, fragility for squared Dehn twists around two-dimensional Lagrangian spheres seems to have no analogue in higher dimensions. For example, since $T^*S^n$ has trivial second cohomology for $n > 2$, by Moser's theorem there are no nontrivial deformations of the symplectic form of $T^*S^n$.
However, there is actually a higher analogue of twisting symplectic invariants, namely the notion of ``bulk deformations" as introduced by Fukaya--Oh--Ohta--Ono.
The Fukaya category of $(M,\theta)$ bulk deformed by a smooth cycle $i_\mho: (\mho,\bdy \mho) \rightarrow (M,\bdy M)$, denoted by $\fuk_\mho(M,\theta)$, is defined by counting pseudoholomorphic polygons with interior point constraints in the cycle.
We give a construction of $\fuk_\mho(M,\theta)$ in $\S\ref{sec:bulk deformations}$, and in \S\ref{sec:squared dehn twists} we prove the following higher analogue of Theorem \ref{thm:main thm twisted}:
\begin{thm}\label{thm:main thm bulk deformed}
Let $L$ and $S$ be Lagrangian spheres in a $4m$-dimensional Liouville domain $(M^{4m},\theta)$, for $l \geq 2$. 
Assume that $L$ and $S$ intersect once transversely, and let $i_\mho: (\mho,\bdy \mho) \rightarrow (M,\bdy M)$ be a smooth half-dimensional cycle in $M$ which is disjoint from $L$ and intersects $S$ once transversely.
Then $L$ and $\tau_S^2L$ are not quasi-isomorphic in $\fuk(M,\theta)$, but are quasi-isomorphic in $\fuk_\mho(M,\theta)$.
\end{thm}
\noindent Theorem \ref{thm:main thm bulk deformed} seems to have no direct interpretation in terms of deformations of symplectic forms, but can perhaps be viewed as an abstract or noncommutative generalization of fragility.
It also seems plausible that both Theorem \ref{thm:main thm twisted} and Theorem \ref{thm:main thm bulk deformed} could be extended to triviality statements for $\tau_S^2$ as an automorphism of the full twisted or bulk deformed Fukaya category.

\sss

Our main application of Theorems \ref{thm:main thm twisted} and \ref{thm:main thm bulk deformed} is to the study of subflexible Weinstein domains as defined in \cite{murphysiegel}.
These are examples of exotic symplectic manifolds whose ordinary symplectic cohomology vanishes. 
It was shown in \cite{murphysiegel} that many of these examples can be seen to be nonflexible, and hence exotic, using twisted symplectic cohomology.
In \S\ref{subsec:bulk deformed sh} we also define the bulk deformed version of symplectic cohomology,
which can be used to distinguish further subflexible examples for which no two-dimensional cohomology class suffices.

In this paper we provide the main computational tool for these examples.
Recall that subflexibilization is defined in \cite{murphysiegel} in terms of Lefschetz fibrations.
Let $(X^{2n+2},\la)$ be a Liouville domain admitting a Liouville Lefschetz fibration over the disk with fiber $(M^{2n},\theta)$ and vanishing cycles $V_1,...,V_k \subset M$. 
The subflexibilization $(X',\la')$ of $(X,\la)$ is defined as follows:
\begin{itemize}
\item
For $i = 1,..,k$, assume there is a Lagrangian disk $T_i \subset M$ with Legendrian boundary $\bdy T_i \subset \bdy M$
such that $T_i$ intersects $V_i$ once transversely.
\item
For $i = 1,...,k$, attach a Weinstein handle $H_i$ to $(M,\theta)$ along $\bdy T_i$.
Let $S_i$ be the Lagrangian sphere given by the union of $T_i$ and the core of $H_i$.
\item
Take $(X',\la')$ to be the total space of the Liouville Lefschetz fibration with fiber $M \cup H_1 \cup ... \cup H_k$ and vanishing cycles $\tau_{S_1}^2V_1,..., \tau_{S_k}^2V_k$.
\end{itemize}

The connection to Theorems \ref{thm:main thm twisted} and \ref{thm:main thm bulk deformed} is hopefully now apparent. 
At least in the case $\dim M = 4$, $\tau_{S_i}^2V_i$ is smoothly isotopic to $V_i$, and consequently as a smooth manifold $X'$ is simply given by attaching $k$ subcritical handles to $X$.\footnote{This is also true for $\dim M = 4l \geq 8$ on the level of homology, although smoothly there is a subtlety regarding the framings of handles; see \cite[Rmk. 4.2]{murphysiegel} and specifically \cite{maydanskiyseidelcorrigendum} for more details.}
Moreover, by \cite[Cor. 3.9]{murphysiegel} the ordinary symplectic cohomology of $(X',\la')$ vanishes.
On the other hand, the symplectic geometry of $(X,\la)$ should be reflected in the 
twisted or bulk deformed symplectic geometry of $(X',\la')$.
We make this precise in \S\ref{sec:squared dehn twists} and \S\ref{sec:from fiber to total space}, culminating in the following theorem. 
Take $\mho$ to be union of the cocores of $H_1,...,H_k$ crossed with the base disk, viewed as 
a smooth cycle (of dimension $n+2$, or codimension $n$) in $X'$ disjoint from the critical handles, and $\Omega$ to be the Poincar\'e dual thereof.
\begin{thm}\label{thm:compute twsh and bdfuk weak}
Suppose $(X,\la)$ has a Lefschetz thimble with nontrivial wrapped Floer cohomology. Then:
\begin{itemize}
\item
If $\dim X =  6$, $\sh_\Omega(X',\la')$ is nontrivial.
\item
If $\dim X = 4m+2 \geq 10$, $\sh_\mho(X',\la')$ is nontrivial.
\end{itemize}
\end{thm}
\noindent 
We explain in \S\ref{sec:from fiber to total space} how this result can be deduced from general tools in symplectic Picard--Lefschetz theory.
Although the above result suffices to establish nonflexibility,
it can be strengthened to the following more elegant statement:
\begin{thm*}\label{thm:compute twsh and bdfuk strong}
In general, we have:
\begin{itemize}
\item
If $\dim X = 6$, there is an isomorphism $\twsh(X',\la') \cong \sh(X,\la)$.
\item
If $\dim X = 4m+2 \geq 10$, there is an isomorphism $\sh_\mho(X',\la') \cong \sh(X,\la)$.
\end{itemize}
\end{thm*}
\noindent The asterisk indicates that the proof relies on either transversality in symplectic field theory or expected but not yet available results in symplectic Picard--Lefschetz theory.

\sss

The rest of this paper is organized as follows. In \S\ref{sec:fukaya categories} we review the main features of Seidel's construction of the Fukaya category and establish some basic properties and notation.
In \S\ref{sec:turning on a B-field} we introduce B-fields and explain how to incorporate them into the frameworks of Fukaya categories and symplectic cohomology.
In \S\ref{sec:bulk deformations} we introduce bulk deformations, construct the bulk deformed Fukaya category and bulk deformation symplectic cohomology, and provide some general context.
In \S\ref{sec:Lefschetz fibrations} we discuss Lefschetz fibrations and some important computational tools.
In \S\ref{sec:squared dehn twists} we then combine these various ingredients to prove Theorem \ref{thm:main thm twisted} and Theorem \ref{thm:main thm bulk deformed}.
Finally, in \S\ref{sec:from fiber to total space} we discuss how to deduce invariants of the total space of a Lefschetz fibration from those of the fiber and use this to prove Theorem \ref{thm:compute twsh and bdfuk weak}
and Theorem* \ref{thm:compute twsh and bdfuk strong}.

\section*{Acknowledgements}
{
I wish to thank my thesis advisor Yasha Eliashberg for numerous stimulating discussions and for providing the inspiration behind this project.
I thank Sheel Ganatra for helping to alleviate a number of my confusions, and Oleg Lazarev for stimulating discussions.
I also thank Mohammed Abouzaid and Paul Seidel for explaining some of their work related to \S\ref{sec:from fiber to total space}. Lastly, I thank the anonymous referee for many attentive comments which helped improve the exposition of this paper.
}

\section*{Conventions}
{
\begin{itemize}
\item By default we assume Lagrangians are embedded, with interior disjoint from the boundary of the ambient symplectic manifold.
\item
Unless stated otherwise, we will assume our symplectic manifolds have trivial first Chern class.
\item In the general context of undeformed symplectic invariants, we work over an arbitrary field $\field$.
\item In the context of twisted invariants as in \S\ref{sec:turning on a B-field}, we 
work over a field $\K$ equipped with an injective group homomorphism $\R \rightarrow \K^*$.
When applying the Picard--Lefschetz techniques in \S\ref{subsec:a picard--lefschetz}, we further assume that $\K$ is not of characteristic two (this is inherited from \cite{seidelbook}).
\item In context of bulk deformed invariants by a cycle of even codimension $l$ as in \S\ref{sec:bulk deformations},
 we work over a graded ring $\KL$ of the form 
$\KL = \KL_0[\hbar,\hbar^{-1}]$, where $\KL_0$ is a field of characteristic zero and $\hbar$ is a formal variable of degree $2-l$.

\end{itemize}

\section{Fukaya categories}\label{sec:fukaya categories}

\subsection{The Fukaya category of a Liouville domain}\label{subsec:fukaya category of a Liouville domain}\label{subsec:the fukaya category of a Liouville domain}

In this subsection we review the various ingredients that go into constructing the Fukaya category of a symplectic manifold. 
We will focus on Liouville domains, a particularly nice class of open symplectic manifolds.
Recall that a Liouville domain is a pair $(M^{2n},\theta)$, where:
\begin{itemize}
\item
$M$ is a smooth compact manifold with boundary.
\item
$\theta$ is a $1$-form on $M$ such that $d\theta$ is symplectic.
\item
The Liouville vector field $Z_\theta$, defined by $(d\theta)(Z_\theta,\cdot) = \theta$, is outwardly transverse to $\bdy M$.
\end{itemize}
Roughly, the Fukaya category $\fuk(M,\theta)$ is an $\calA_\infty$ category with objects given by closed exact Lagrangians in $(M,\theta)$, morphisms given by Floer cochain complexes between Lagrangians, and higher $\calA_\infty$ products given by counting pseudoholomorphic polygons with Lagrangian boundary conditions. Working in the exact setting alleviates many of the analytic difficulties that plague the field. 
Still, there are various technical issues to deal with, related to the fact that Lagrangians might not intersect transversely and moduli spaces of pseudoholomorphic curves might not be cut out transversely. In \cite{seidelbook}, Seidel gives a careful outline for overcoming these issues using coherent perturbations of the Cauchy--Riemann equations. 
This approach requires choosing the following auxiliary data in a coherent manner:
\begin{enumerate}
\item
{\em Floer data} for every pair of Lagrangians, which is the data needed to define Floer complexes
\item
{\em strip-like ends} for boundary-punctured Riemann disks, which are needed to formulate asymptotic conditions for pseudoholomorphic curves near boundary punctures
\item 
{\em perturbation data}, which are used to perturb the Cauchy--Riemann equations and ensure that all relevant moduli spaces of pseudoholomorphic curves are cut out transversely.
\end{enumerate}
Since we consider pseudoholomorphic maps whose domains have arbitrary complex structures, these choices must vary smoothly over the moduli space of boundary-punctured Riemann disks. Moreover, in order for various curve counts to fit together to satisfy the $\calA_\infty$ equations, our choices should be appropriately compatible with the structure of the Deligne--Mumford--Stasheff compactification.

We now explain each of these elements in more detail.
This general framework for handling pseudoholomorphic curve invariants will reappear in several flavors throughout the paper.

\subsubsection{The moduli space $\calR_{k+1}$ and its compactification.}
\label{subsubsec:stable disks}
For $k+1 \geq 3$, let $\calR_{k+1}$ denote the moduli space of Riemann disks with $k+1$ ordered boundary marked points, modulo biholomorphisms. 
Here and for the rest of the paper we require marked points to be pairwise disjoint.
We further require the ordering of the marked points to respect the boundary orientation.
We declare the first marked point to be ``negative" and the rest to be ``positive"\footnote{In the sequel we will endow marked points and punctures of Riemann surfaces with signs.
When discussing strip-like ends and cylindrical ends we assume their signs match those of the corresponding marked points or punctures.}.
Note that $\calR_{k+1}$ is a smooth manifold of dimension $k-2$.
There is also a universal family $\calRunivpre_{k+1} \rightarrow \calR_{k+1}$, where the fiber $\calRunivpre_r$ over $r \in \calR_{k+1}$ is a boundary-marked Riemann disk which represents $r$ itself.
Concretely, we can take 
\begin{align*}
\calR_{k+1} = \conf_{k+1}(\bdy\D^2)/\pslr, \;\;\;\;\; \calRunivpre_{k+1} = \conf_{k+1}(\bdy \D^2) \times_{\pslr} \D^2.
\end{align*}
Finally, let $\calRuniv_{k+1} \rightarrow \calR_{k+1}$ be given by puncturing the marked points in each fiber of $\calRunivpre_{k+1} \rightarrow \calR_{k+1}$.

Let $\calTR{k+1}$ denote the set of stable planted\footnote{By a {\em planted tree} we mean a tree with one external vertex distinguished as the root.} 
ribbon trees with $k$ leaves.
That is, an element of $\calTR{k+1}$ is a tree $T$ with:
\begin{itemize}
\item
one distinguished external vertex, called the ``root"
 \item 
 $k$ remaining external vertices, called the ``leaves"
 \item 
{\em ribbon structure:} a cyclic ordering of the edges incident to each vertex
 \item 
{\em stability:} each internal vertex $v$ has valency $|v| \geq 3$.
\end{itemize}
We will always endow the edges of $T$ with the orientation pointing away from the root. In particular, this induces an absolute ordering of the edges incident to each vertex, starting with the incoming edge.
Let $V_i(T)$ and $E_i(T)$ denote the internal vertices and internal edges of $T$ respectively.
As we will now recall, $\calTR{k+1}$ models the stratification structure of the Deligne--Mumford--Stasheff compactification of $\calR_{k+1}$.

As a preliminary step, endow the family $\calRuniv_{k+1} \rightarrow \calR_{k+1}$
with a choice of {\em universal strip-like ends}, for all $k+1 \geq 3$.
This means, for each $r \in \calR_{k+1}$, 
pairwise disjoint holomorphic embeddings
\begin{align*}
\epsilon_0: \left(\R_- \times [0,1], \R_- \times \{0,1\}\right) \hookrightarrow (\calRuniv_r,\bdy \calRuniv_r)\\
\epsilon_1,...,\epsilon_k: \left(\R_+ \times [0,1], \R_+ \times \{0,1\}\right) \hookrightarrow (\calRuniv_r,\bdy \calRuniv_r)
\end{align*}
such that $\lim_{s \rightarrow \pm \infty}\epsilon_i(s,\cdot)$ is the $i$th puncture of $\calRuniv_r$ for $i = 0,...,k$.
These should combine to give smooth fiberwise embeddings
\begin{align*}
\epsilon_0,...,\epsilon_k&: \calR_{k+1} \times \R_{\pm} \times [0,1] \hookrightarrow \calRuniv_{k+1}.
\end{align*}
Note that so far we have not mentioned any compatibility between these choices for different $k$.

For $T \in \calTR{k+1}$ and $\e > 0$ small, set 
\begin{align*}
\ovl{\calR}_T := \prod_{v\in V_i(T)} \calR_{|v|},\;\;\;\;\;
\ovl{\calR}_T^{\e} := \ovl{\calR}_T \times (-\e,0]^{E_i(T)}.
\end{align*}
We will identify $\ovl{\calR}_T$ with the subset $\ovl{\calR}_T \times \{0\}^{E_i(T)} \subset \ovl{\calR}_T^{\e}$. 
We use $\rho_e$ to denote the coordinate on $(-\e,0]$ corresponding to $e \in E_i(T)$.
This will be a gluing parameter corresponding to a gluing region having neck length $\ell_e := -\log(-\rho_e)$ \cite[9e]{seidelbook}.
That is, given Riemann surfaces $S_+$ and $S_-$ with strip-like ends 
$\epsilon_+: \R_+ \times [0,1] \hookrightarrow S_+$
and $\epsilon_-: \R_- \times [0,1] \hookrightarrow S_-$, we glue with parameter $\ell = -\log(-\rho) \in (0,\infty)$ by starting with the disjoint union $S_+ \coprod S_-$, throwing away $\epsilon_+([\ell,\infty) \times [0,1])$ 
and $\epsilon_-((-\infty,-\ell] \times [0,1])$, and then identifying what remains of the strip-like ends via $\epsilon_+(s+\ell,t) \sim \epsilon_-(s,t)$.

Observe that, using the absolute ordering of the edges at the internal vertices,
each $e \in E_i(T)$ corresponds to two boundary marked points of $\ovl{\calR}_T$.
By gluing at these two marked points using our universal strip-like ends,
we get a map
\begin{align*}
\phi_{T,e}: \{r \in \ovl{\calR}^\e_T\;:\; \rho_e  \neq 0\} \rightarrow \ovl{\calR}_{T/e}^\e.
\end{align*}
Here $T/e$ is the tree obtained by contracting the internal edge $e$.
More specifically, we endow $T/e$ with the {\em planarly induced} ribbon structure, defined as follows.
Let $\iv(e)$ and $\tv(e)$ denote the initial and terminal vertices of $e$ in $T$.
Let $v \in T/e$ denote the resulting vertex after contracting $e$, and let $E(v)$ denote the edges incident to it.
We order $E(v)$ by enumerating $E(\iv(e))$ up to $e$, then enumerating $E(\tv(e))$, and lastly enumerating the remaining elements of $E(\iv(e))$.

We can now define the Deligne--Mumford--Stasheff compactification of $\calR_{k+1}$ as a topological space by
\begin{align*}
\ovl{\calR}_{k+1} := \left(\coprod_{T \in \calTR{k+1}} \ovl{\calR}_T^\e \right) / \sim, 
\end{align*}
where $r \sim \phi_{T,e}(r)$ for any $r$ in the domain of $\phi_{T,e}$.
Note that as a set we have $\ovl{\calR}_{k+1} = \coprod_{T \in \calTR{k+1}}\ovl{\calR}_T$, and the topology is such that
$\calRovl^\e_T$ is a collar neighborhood of $\calRovl_T$.
Actually, it is well known that the space $\calRovl_{k+1}$ has much more structure than just a topological space.
Among other things it is naturally a $(k-2)$-dimensional smooth manifold with corners, in fact a convex polytope.
Elements of $\ovl{\calR}_{k+1}$ are {\em stable broken disks}.
Here we think of the limiting case of gluing for $\rho_e = 0$ as simply identifying the two marked points, i.e. producing a boundary node.
The unique tree $T_0$ with no internal edges corresponds to the open stratum $\ovl{\calR}_{T_0} = \calR_{k+1}$,
and more generally the codimension of the stratum $\calRovl_T$ is given by the number of internal edges of $T$.
There is also a partially compactified universal family $\calRovluniv_{k+1} \rightarrow \ovl{\calR}_{k+1}$, 
where the fiber $\calRovluniv_r$ over $r \in \ovl{\calR}_{k+1}$ represents the corresponding broken disk.

\subsubsection{Lagrangian labels.}
Let $\mathbf{L} := (L_0,...,L_k)$ be a list of closed exact Lagrangians (not necessarily pairwise distinct) in $(M,\theta)$.
Let $\calR(\textbf{L})$ denote the moduli space of Riemann disks with $k+1$ ordered boundary marked points, modulo biholomorphisms, such that the segments of the boundary between the marked points are labeled in order by $L_0,...,L_k$.
This space is of course equivalent to $\calR_{k+1}$, but it will be convenient to keep track of the Lagrangian labels.

For $T \in \calTR{k+1}$, let $\gamma_0(T)$ denote the minimal path from the root to the first leaf, let $\gamma_i(T)$ denote the minimal path from the $i$th leaf to the $(i+1)$st leaf for $i = 1,...,k-1$, and let $\gamma_k(T)$ denote the minimal path from the $k$th leaf to the root.
We say that $T$ is {\em labeled by $\mathbf{L}$} if the paths $\gamma_0(T),...,\gamma_k(T)$ are labeled in order by $L_0,...,L_k$.
More visually, if we embed $T$ as a ribbon graph in $\R^2$ with the external vertices at infinity, this data is equivalent to labeling the connected components of $\R^2 \setminus T$ in order by $L_0,...,L_k$.
Observe that each edge has two associated labels, corresponding to its two sides.
Also, each $v \in V_i(T)$ has an associated list $\mathbf{L}_v$ of Lagrangian labels, namely those encountered (in order) in a small neighborhood of $v$.
We set
\begin{align*}
\ovl{\calR}_T(\mathbf{L}) :=  \prod_{v \in V_i(T)} \calR(\mathbf{L}_v),\;\;\;\;\;  \ovl{\calR}_T^\e(\mathbf{L}) := \ovl{\calR}_T(\mathbf{L}) \times (-\e,0]^{E_i(T)}.
\end{align*}
For any internal edge $e$, the contracted tree $T/e$ naturally inherits an $\mathbf{L}$ labeling from $T$, and therefore as above we can define gluing maps
\begin{align*}
\phi_{T,e}: \{r \in \ovl{\calR}^\e_T(\mathbf{L})\;:\; \rho_e \neq 0 \} \rightarrow \ovl{\calR}_{T/e}^\e(\mathbf{L}),
\end{align*}
along with the compactification $\ovl{\calR}(\mathbf{L})$, the universal family
$\calRuniv(\mathbf{L}) \rightarrow \calR(\mathbf{L})$, and its partial compactification
$\calRovluniv(\mathbf{L}) \rightarrow \ovl{\calR}(\mathbf{L})$.

\subsubsection{Floer data}\label{subsubsec:Floer data}

Let $\calH$ denote the space of smooth real-valued functions on $M$ (i.e. ``Hamiltonians") which vanish on $\Op(\bdy M)$.
Let $\Jo$ be a fixed reference almost complex structure on $M$ which is compatible with $d\theta$ and 
makes the boundary of $M$ {\em weakly $\Jo$-convex} (see \cite[\S 2.3]{cieliebak2012stein}).
For such an almost complex structure, there is a maximum principle which prevents pseudoholomorphic curves in $M$ from touching the boundary unless they are entirely contained in it.
Let $\calJfix$ denote the space of almost complex structures on $M$ which are $d\theta$-compatible and coincide with $\Jo$ on $\Op(\bdy M)$.
For a pair $(L_0,L_1)$ of Lagrangians in $(M,\theta)$, a 
{\em Floer datum} $(H,J)$ consists of:
\begin{itemize}
\item
a time-dependent family of Hamiltonians $H \in C^{\infty}([0,1],\calH)$
such that the image of $L_0$ under its time-$1$ flow $\phi^1_{H}$ 
is transverse to $L_1$
\item
a time-dependent family of almost complex structures $J \in C^{\infty}([0,1],\calJfix)$.
\end{itemize}
Let $\genHLL{L_0}{L_1}$ denote the finite set of time-$1$ Hamiltonian flow trajectories of $H$ which start on $L_0$ and end on $L_1$.
Note that $\genHLL{L_0}{L_1}$ is in bijection with the set of intersection points $\phi_{H}^1(L_0) \cap L_1$.

For Lagrangian labels $L_0,L_1$ and $x_0,x_1 \in \genHLL{L_0}{L_1}$, let $\wh{\calM}(x_0,x_1)$ denote the corresponding space of Floer strips, i.e. maps $u: \R \times [0,1] \rightarrow M$ such that:
\begin{itemize}
\item
$u$ satisfies Floer's equation: $\bdy_s u + J_{L_0,L_1}(\bdy_t u - X_{H}) = 0$
\item
$u(\R \times \{0\}) \subset L_0$ and $u(\R \times \{1\}) \subset L_1$
\item
$\lim_{s \rightarrow -\infty}u(s,\cdot) = x_0$ and $\lim_{s \rightarrow +\infty} u(s,\cdot) = x_1$,
\end{itemize}
and let $\calM(x_0,x_1)$ denote the quotient of $\wh{\calM}(x_0,x_1)$ by the free $\R$-action which translates in the $s$ coordinate.
We say that the Floer data $(H,J)$ is {\em regular for the pair $L_0,L_1$} if the moduli spaces $\calM(x_0,x_1)$ are regular (i.e. their corresponding linearized Cauchy--Riemann type operators are surjective), and hence smooth manifolds, for all $x_0,x_1 \in \genHLL{L_0}{L_1}$.
It is a standard fact in Floer theory that a generic\footnote{Here we say that a property holds ``generically'' if the subspace on which it holds is comeager, i.e. it contains a countable intersection of open dense subsets.} choice of Floer data is regular.
More precisely, after choosing $H$ such that $\phi^1_{H}(L_0)$ intersects $L_1$ transversely, the moduli spaces $\calM(x_0,x_1)$ for all $x_0,x_1 \in \genHLL{L_0}{L_1}$ are regular for generic $J \in C^{\infty}([0,1],\calJfix)$.  
In particular, if $L_0$ and $L_1$ already intersect transversely, we can take $H$ to be trivial, and furthermore we can take $J = \Jo$ if all of the associated moduli spaces of Floer strips happen to already be regular.

Assume now that we have chosen a regular Floer datum$(H_{L_0,L_1},J_{L_0,L_1})$ for every pair of closed exact Lagrangians $(L_0,L_1)$ in $(M,\theta)$. 

\subsubsection{Consistent universal strip-like ends}\label{subsubsec:consistent strip-like ends}
In order for the moduli spaces defined below to have the desired compactification structure, we need to pick strip-like ends somewhat more carefully.
At this point, for any $T \in \calTR{k+1}$, gluing along all of the internal edges of $T$ induces a map
\begin{align*}
\phi_T: \ovl{\calR}_T \times (-\e,0)^{E_i(T)} \rightarrow \calR_{k+1}.
\end{align*}
{The image of $\phi_T$ is equipped with two a priori different families of strip-like ends, 
one induced by our universal choice for $\calR_{k+1}$, and one 
induced by gluing the universal choices for each $\calR_{|v|}$.
We say our universal choices are {\em consistent} if these two families agree, at least sufficiently close to the stratum $\ovl{\calR}_T$, for all $T \in \calTR{k+1}$.
A basic fact is that by making choices inductively we can find consistent universal strip-like ends (see \cite[\S 9g]{seidelbook} for details).
From now on we assume such a choice has been made,
and we use this to induce consistent universal strip-like ends on the families $\calRuniv(\bL) \rightarrow \calR(\bL)$ for all $\bL$.

\subsubsection{Consistent universal perturbation data}\label{subsubsec:Consistent universal perturbation data}

Let $S$ be a fixed Riemann disk with $k+1 \geq 3$ boundary punctures, equipped with Lagrangian labels $\bL$ and strip-like ends induced from the universal family.
A {\em perturbation datum for $S$} is a pair $(K,J)$ consisting of:
\begin{itemize}
\item
$K \in \Omega^1(S,\calH)$
\item
$J \in C^\infty(S,\calJfix)$
\end{itemize}
subject to the conditions:
\begin{itemize}
\item
For any $p \in \bdy S$ with corresponding label $L$, we have $K(X)|_L \equiv 0$ for any $X \in T_p\bdy S$.
\item 
For each boundary puncture of $S$ with adjacent labels $(L,L')$ and corresponding strip-like end $\epsilon$, we have
\begin{align*}
(\epsilon^*K,\epsilon^*J) \equiv (H_{L,L'}dt,J_{L,L'}).
\end{align*}
\end{itemize}
In other words, $K$ is a Hamiltonian-valued one-form on $S$, $J$ is an $S$-dependent family of almost complex structures on $M$, and these reduce to the already chosen Floer data along each of the strip-like ends.

A {\em universal choice of perturbation data} consists of a smoothly varying choice of fiberwise perturbation data for the universal family $\calRuniv(\bL) \rightarrow \calR(\bL)$, for all Lagrangian labels $\bL$.
By design, the perturbation data are standard on the strip-like ends, and hence can be glued together.
This means that the image of $\phi_T: \ovl{\calR}_T(\bL) \times (-\e,0)^{E_i(T)} \rightarrow \calR_{k+1}(\bL)$ is equipped with 
two a priori different families of perturbation data, one induced by the universal choice for $\calR_{k+1}(\bL)$ and one induced by gluing.
Naively we would like to require these to coincide, in parallel to the situation for strip-like ends. However, this turns out to be too stringent, since it imposes restrictions which make it difficult to perturb in order to achieve regularity (see \cite[Remark 9.6]{seidelbook}). Instead, we say our perturbation data are {\em consistent} if, for all $T$ and $\bL$, we have:
\begin{itemize}
\item
the two families agree on the {\em thin parts} of $\calRuniv_r(\bL)$ for all $r \in \calR(\bL)$ sufficiently close to the boundary stratum $\ovl{\calR}_T(\bL)$
 (see \cite[Remark 9.1]{seidelbook} for details)
\item
the perturbation data on $\calRuniv(\bL) \rightarrow \calR(\bL)$ extends smoothly to $\calRovluniv(\bL) \rightarrow \ovl{\calR}(\bL)$ in such a way that it agrees with the induced perturbation data on $\ovl{\calR}_T(\bL)$.
\end{itemize}
As explained in  \cite[\S 9i]{seidelbook}, consistent universal choices of perturbation data can be found by an inductive procedure.
From now on we assume such a choice has been made.

\subsubsection{The moduli space $\calM(\bx)$}\label{subsubsec:pseudoholomorphic polygons}

As before, let $S$ be a fixed Riemann disk with $k+1 \geq 3$ boundary punctures and Lagrangian labels $\bL$.
We assume $S$ is equipped with the strip-like ends $\epsilon_0,...,\epsilon_k$ and perturbation datum $(K,J)$ induced from our universal choices.
Note that the symplectic form $d\theta$ can be used to dualize any Hamiltonian $H \in \calH$ to a Hamiltonian vector field $X_H$ on $M$ via the prescription $d\theta(X_H,-) = dH$. 
Similarly, by using $d\theta$ to dualize the outputs of $K$, we obtain a one-form $Y$ on $S$ with values in Hamiltonian vector fields on $M$.
In this context, a {\em pseudoholomorphic polygon with domain $S$} is a map $u: S \rightarrow M$ 
which sends each boundary component of $S$ to its corresponding Lagrangian label and satisfies the 
{\em inhomogeneous pseudoholomorphic curve equation}
\begin{align*}
\left(Du - Y\right)^{0,1} = 0.
\end{align*}
Here the superscript denotes the complex anti-linear part with respect to the complex structure on $S$ and $J$ on $M$.
Let $\gen(\bL)$ denote the set of tuples $\bx = (x_0,...,x_k)$,
where $x_0 \in \genHLL{L_0}{L_k}$ 
and $x_i \in \genHLL{L_{i-1}}{L_i}$ for $i = 1,...,k$.
Given $\bx \in \gen(\bL)$,
$u$ is said to have {\em asymptotics $\bx$} if 
\begin{align*}
\lim_{s \rightarrow \pm \infty} (u \circ \epsilon_i)(s,\cdot) = x_i\;\;\;\;\; \text{for}\;\; i= 0,...,k.
\end{align*} 
Let $\calM_S(\bx)$ denote the space of pseudoholomorphic polygons $u$ with domain $S$ and asymptotics $\bx$.
We define the moduli space of pseudoholomorphic polygons with asymptotics $\bx$ and {\em arbitrary domain} by
\begin{align*}
\calM(\bx) := \{(r,u)\;:\; r \in \calR(\bL),\; u \in \calM_{\calRuniv_r}(\bx)\}.
\end{align*}
Here the Lagrangian labels $\bL$ are implicit in the notation.
\begin{prop}{\cite[\S 9k]{seidelbook}}
\label{prop:moduli spaces are topological manifolds}
For generic perturbation data,
the moduli space $\calM(\bx)$ is regular and hence has the structure of a smooth manifold.
\end{prop}

\noindent In the special case $\bx = (x_0,x_1)$, we similarly define $\calM(\bx) := \calM(x_0,x_1) = \wh{\calM}(x_0,x_1)/\R$ to be the corresponding moduli space of Floer strips.

\subsubsection{The compactification $\ovl{\calM}(\bx)$}\label{def:stable pseudoholomorphic polygons}

There is a natural compactification of $\calM(\bx)$ by allowing pseudoholomorphic maps to acquire certain boundary nodes.
The domain of such a map is a tree of boundary-marked Riemann disks as in $\ovl{\calR}_{k+1}$, but now the stability condition only applies to constant components.
This is in accordance with the general procedure of compactifying moduli spaces of pseudoholomorphic curves by stable maps.

More precisely, for $k+1 \geq 2$, let $\calTRsemi{k+1}$ be defined in the same way as $\calTR{k+1}$, except that we replace the stability condition with {\em semistability}, i.e. each interval vertex must have valency at least two.
Suppose we have Lagrangian labels $\bL = (L_0,...,L_k)$ and $\bx = (x_0,...,x_k) \in \gen(\bL)$.
A {\em stable broken pseudoholomorphic polygon with asymptotics $\bx$} consists of:
\begin{itemize}
\item
$T \in \calTRsemi{k+1}$ labeled by $\bL$
\item 
$x_e \in \genHLL{L_e}{L_{e}'}$ for each edge $e$ of $T$, where $(L_e,L_e')$ denotes the Lagrangian labels on either side of $e$,
and such that $x_e = x_i$ if $e$ is the $i$th external edge of $T$
\item 
$u_v \in \calM(\bx_v)$ for each $v \in V_i(T)$, where $\bx_v \in \gen(\bL_v)$ denotes the ordered list of elements $x_e$ encountered at the edges incident to $v$.
\end{itemize}

Let $\ovl{\calM}_T(\bx)$ denote the moduli space of equivalence classes of stable broken pseudoholomorphic polygons with asymptotics $\bx$ and fixed tree structure $T \in \calTRsemi{k+1}$.
We set
\begin{align*}
\ovl{\calM}(\bx) := \coprod_{T \in \calTRsemi{k+1}}  \calM_T(\bx),
\end{align*}
equipped with the Gromov topology.
Gromov's compactness theorem implies that $\calMovl(\bx)$ is compact.
Moreover, for $\e > 0$ sufficiently small, standard gluing techniques produce a map 
\begin{align*}
\phi_{T}: \ovl{\calM}_T(\bx) \times (-\e,0)^{E_i(T)} \rightarrow \calM(\bx).
\end{align*}
By analyzing these maps, one can show:
\begin{prop}{\cite[\S 9l]{seidelbook}}\label{prop:compact top mfd with corners}
The space $\ovl{\calM}(\bx)$ has the structure of a compact topological manifold with corners with open stratum $\calM(\bx)$.
\end{prop}
\noindent For our purposes we will not need to know the precise global structure of $\ovl{\calM}(\bx)$ in general, but rather just an understanding of the zero and one dimensional pieces.

\subsubsection{Brane structures}\label{subsubsec:brane structures}

In its most rudimentary form, the Fukaya category of $(M,\theta)$ is an ungraded $\calA_\infty$ category defined over a coefficient field of characteristic two.
However, it is sometimes desirable to upgrade this to a $\Z$-graded $\calA_\infty$ category over a coefficient field of arbitrary characteristic, and for this we need to equip Lagrangians with {\em brane structures}.
This extra structure is probably not central to the main results of this paper but we will nevertheless assume it in order to streamline the discussion.\footnote{In particular, we expect that Theorems ~\ref{thm:compute twsh and bdfuk weak} and ~\ref{thm:compute twsh and bdfuk strong} hold without the assumption $c_1(X,d\la) = 0$, provided we work with ungraded versions of (deformed) symplectic cohomology.}

From now on we assume $c_1(M,d\theta) = 0$ and fix a
nonvanishing section $\eta \in \La^{\text{top}}_\C(T^*M)$.
For any Lagrangian $L \subset M$, we get an associated {\em squared phase map} $\alpha: L \rightarrow S^1$ of the form 
\begin{align*}
\alpha(p) = \dfrac{\eta^2(v_1\wedge...\wedge v_n)}{| \eta^2(v_1\wedge...\wedge v_n)|},
\end{align*}
where $v_1,...,v_n$ is any basis for $T_pL$.
A {\em Lagrangian brane} consists of a triple $(L,P^\#,\alpha^\#)$, where:
\begin{itemize}
\item $L$ is a Lagrangian
\item $P^\#$ is a $Pin$ structure on $L$ 
\item $\alpha^\#$ is a {\em grading} on $L$, i.e. a function $L \rightarrow \R$ such that $\exp(2\pi i \alpha^\#(p)) = \alpha(p)$.
\end{itemize}
In general, the obstruction to putting a $Pin$ structure on $L$ is $\omega_2(TL) \in H^2(L;\Z/2)$,
and the set of $Pin$ structures on $L$ (if nonempty) is an affine space over $H^1(L;\Z/2)$.
The obstruction to putting a grading on $L$ is the {\em Maslov class} $\mu_L \in H^1(L;\Z)$, and if it exists it is unique up to an overall integer shift.
As an important special case, note that any sphere of dimension at least two automatically admits a unique $Pin$ structure and a countably infinite set of gradings.

Suppose we have $\bL = (L_0,...,L_k)$ and $\bx = (x_0,...,x_k) \in \gen(\bL)$, and each $L_i$ is equipped with a brane structure.
Then the moduli space $\calM(\bx)$ is oriented (see \cite[\S 11]{seidelbook}).
More precisely, for each pair of Lagrangians $L,L'$ and each $x \in \genHLL{L}{L'}$, there is an associated {\em orientation space} $\orien(x)$, which is an associated real one-dimensional vector space. 
We choose an arbitrary\footnote{Alternatively, there is an equivalent but somewhat more canonical definition of the Fukaya category which does not require choosing trivializations of orientation spaces, at the cost of slightly more notation.} trivialization of all orientation spaces, after which the orientation on $\calM(\bx)$ is uniquely determined.
In particular, each $u \in \calM(\bx)^{\oo}$ has an associated sign $\signu \in \{-1,1\}$,
where $\calM(\bx)^{\oo}$ denotes the isolated (i.e. Fredholm index zero) points of $\calM(\bx)$.
Moreover,
the Maslov index endows each $x_i$ with an integer grading $|x_i| \in \Z$,
and we have 
\begin{align*}
\dim \calM(\bx) = |x_0| - |x_1| - ... - |x_k| + k-2.
\end{align*}
(see  \cite{seidel2000graded} and \cite[\S 11]{seidelbook}).
From now on we will mostly suppress $Pin$ structures and gradings from the notation and (by slight abuse) denote a Lagrangian brane $(L,P^\#,\alpha^\#)$ simply by $L$.
}

\subsubsection{The Fukaya category}

Let $\field$ be an arbitrary coefficient field.
The Fukaya category $\fuk(M,\theta)$ is defined to be the $\Z$-graded $\calA_\infty$ category over $\field$ with:
\begin{itemize}
\item
objects given by closed exact Lagrangian branes in $(M,\theta)$
\item 
for objects $L_0,L_1$, $\hom(L_0,L_1)$ is the $\field$-module generated by $\genHLL{L_0}{L_1}$
\item
for $k \geq 1$, objects $L_0,...,L_k \subset M$, and
$x_i \in \genHLL{L_{i-1}}{L_i}$ for $i = 1,...,k$, 
we set
\begin{align*}
\mu^k(x_k,...,x_1) := \sum_{\substack{x_0 \in \genHLL{L_0}{L_k}\\ u \in \calM(x_0,...,x_k)^{\oo}}} \signu x_0.
\end{align*}
\end{itemize}
Note that the above sums are finite by \S\ref{def:stable pseudoholomorphic polygons}. By the dimension formula from \S\ref{subsubsec:brane structures}, $\mu^k$ has degree $2-k$.
Finally, following a well-known outline, an analysis of the boundaries of the one-dimensional components of $\ovl{\calM}(\bx)$
confirms that the terms $\mu^1,\mu^2,...$ satisfy a signed version of the $\calA_\infty$ structure equations (see \cite[1a]{seidelbook}).

\subsection{Basic invariance properties}\label{subsec:basic_inv}

In this subsection we explain in what sense the Fukaya category is independent of the various ingredients involved in its construction. 
Our primary motivation is to show that certain notions of quasi-isomorphism do not depend on any of these choices.
Most of the results in the section also have straightforward extensions to the settings of twisted and bulk deformed Fukaya categories described in \S\ref{subsec:the twisted Fukaya category} and \S\ref{subsec:bd fuk} respectively.
As such, we sprinkle in a few relevant remarks and leave the precise formulations to the reader. 

\subsubsection{Notions of equivalence}

For an $\calA_\infty$ category $\calC$ over a field $\field$, let $H\calC$ denote the cohomology category.
This is an ordinary graded linear category, possibly without identity morphisms,
with the same objects as $\calC$.
The morphism space $\hhom(X_0,X_1)$ between objects $X_0$ and $X_1$ is by definition the cohomology of $\hom(X_0,X_1)$ with respect to $\mu^1$,
with composition given by $[x'] \cdot [x] := (-1)^{|x|}[\mu^2(x',x)]$.
 
In the case of $\fuk(M,\theta)$,
standard techniques (e.g. the Piunikhin--Salamon--Schwarz isomorphism) show that the endomorphism space $\hhom(L,L)$ of any 
object $L$ is isomorphic as a graded $\field$-algebra to the singular cohomology $H(L;\field)$ of $L$.
Under this isomorphism, the multiplicative unit of $H(L;\field)$ plays the role of an identity morphism in $H\fuk(M,\theta)$.
In particular, since the cohomology category has an identity morphism for every object, we say that 
$\fuk(M,\theta)$ is {\em cohomologically unital}.
 
In general, an $\calA_\infty$ functor $\calC_0 \rightarrow \calC_1$ between two cohomologically unital $\calA_\infty$ categories $\calC_0$ and $\calC_1$ is said to be an {\em $\calA_\infty$ quasi-isomorphism} if the induced cohomology level functor $H\calC_0 \rightarrow H\calC_1$ is an isomorphism of categories.
In this case, a useful consequence of homological perturbation theory is that one can also find an $\calA_\infty$ quasi-isomorphism $\calC_1 \rightarrow \calC_0$ in the reverse direction (see \cite[Corollary 1.14]{seidelbook}, which in fact produces a two-sided inverse up to homotopy).
Similarly, an $\calA_\infty$-functor $\calC_0 \rightarrow \calC_1$ is said to be an {\em $\calA_\infty$ quasi-equivalence} if the induced cohomology level functor is an equivalence of categories, and in this case one can find a reverse $\calA_\infty$ quasi-equivalence $\calC_1 \rightarrow \calC_0$ (see \cite[Theorem 2.9]{seidelbook}).
For later use, we state the following case with slightly more care. In general, 
we say that two objects of a cohomologically unital $\calA_\infty$ category are {\em quasi-isomorphic} if they are isomorphic in the cohomology level category.
\begin{lemma}\label{lem:quasi-iso trick}
Let $\calC$ be a cohomologically unital $\calA_\infty$ category over $\field$ with objects of the form $A_i$ and $B_i$ for $i$ in some indexing set $\calI$, and let $\calA$ and $\calB$ denote the full $\calA_\infty$ subcategories with objects $\{A_i \;:\; i \in \calI\}$ and $\{B_i \;:\; i \in \calI\}$ respectively. Assume that $A_i$ and $B_i$ are quasi-isomorphic in $\calC$ for each $i \in \calI$. Then there is an $\calA_\infty$ quasi-isomorphism $\calA \rightarrow \calB$ sending $A_i$ to $B_i$ for each $i \in \calI$.
\end{lemma}

In the context of bulk deformations, we will also need the above lemma to hold over $\KL$. Here a bit of care is needed, since the relevant homological algebra over an arbitrary commutative ring is rather subtle. 
The proof of \cite[Corollary 1.14]{seidelbook} is based on \cite[Remark 1.13]{seidelbook}, which explains how the transfer principle for $\calA_\infty$ categories (over a field) follows from the homological perturbation lemma.
As recalled in \cite[Prop. 2.2]{lekili2012arithmetic}, the homological perturbation lemma still holds over an arbitrary commutative ring, although in general it implies only a weak form of the transfer principle.
Upon closer inspection of  \cite[Remark 1.13]{seidelbook}, we see that the necessary ingredient is that each morphism space $(\hom_\calC(X_0,X_1),\mu_\calC^1)$ can be split into a direct summand where the differential $\mu^1_\calC$ vanishes, plus a chain contractible complement.
As a warning, note that not every chain complex over $\Z$ admits such a splitting
(consider the complex $0 \rightarrow \Z \rightarrow \Z \rightarrow 0$ with nonzero map given by multiplication by two).
However, in the case of a bulk deformed Fukaya category each morphism space
is a finite dimensional chain complex over $\KL$ with a distinguished basis of homogeneous elements. 
In particular, the matrix coefficients of the differential are of the form $\hbar^k l$ for $k \in \Z$, $l \in \KL_0$, and any such element is automatically invertible if nonzero.
Using this and similar degree considerations, one can produce a splitting of $\KL$-modules of the form $\hom_\calC(X_0,X_1) = \op{Im}\mu^1_\calC \oplus H \oplus I$, where $\ker \mu^1_\calC = \op{Im}\mu^1_\calC \oplus H$. This means that the differential vanishes on $H$ and $\op{Im}{\mu^1_\calC} \oplus I$ defines a contractible subcomplex.

\subsubsection{Independence of Floer data, strip-like ends, and perturbation data}\label{subsubsec:Independence of Floer data, strip-like ends, and perturbation data}

Following \cite[10a]{seidelbook}, a simple algebraic trick can be used to show that the Fukaya category is, up to $\calA_\infty$ quasi-isomorphism, independent of the choices of Floer data, strip-like ends, and perturbation data.
Namely, suppose that $\calC_0$ and $\calC_1$ are two different versions of $\fuk(M,\theta)$ constructed using different such choices. We formally produce a bigger $\calA_\infty$ category $\calC_\tot$ which contains both $\calC_0$ and $\calC_1$.
The objects of $\calC_\tot$ are formal pairs $(L,i)$, where $i \in \{0,1\}$ and $L$ is an object of $\calC_i$. 
To construct morphisms and composition maps for $\calC_\tot$, we proceed as in \S\ref{subsec:fukaya category of a Liouville domain} by choosing a Floer datum for every pair of objects and choosing consistent universal strip-like ends and perturbation data for boundary-marked disks labeled by objects of $\calC_\tot$.
We additionally require that the relevant choices for $(L_0,i_0),...,(L_k,i_k)$ coincide with the choices made for $\calC_i$ in the case that $i_0 = ... = i_k = i$.
The upshot is that there are full and faithful $\calA_\infty$ embeddings $\calC_0 \rightarrow \calC_\tot$ and $\calC_1 \rightarrow \calC_\tot$. 
Moreover, one can check (say by arguing \`a la the PSS isomorphism) that $(L,0)$ and $(L,1)$ are quasi-isomorphic in $\calC_\tot$ for any closed exact Lagrangian $L$.
Lemma \ref{lem:quasi-iso trick} then produces an $\calA_\infty$ quasi-isomorphism $\calC_0 \rightarrow \calC_1$ which sends $(L,0)$ to $(L,1)$ for every $L$.

\subsubsection{Liouville subdomains}\label{subsubsec:Liouville subdomains}
Any fixed collection $\LL$ of closed exact Lagrangians in $(M,\theta)$ defines a full
$\calA_\infty$ subcategory $\fuk(\LL) \subset \fuk(M,\theta)$.
Of course, in principle $\fuk(\LL)$ depends strongly on the ambient Liouville domain $(M,\theta)$.
In fact, the next lemma shows that in favorable circumstances it does not.
\begin{lemma}\label{lemma:Liouville subdomains}
Let $(M,\theta)$ be a Liouville domain and $M' \subset M$ a subdomain such that 
the Liouville vector field $Z_\theta$ is outwardly transverse to $\bdy M'$.
Let $\LL'$ be a set of closed exact Lagrangian branes in $(M',\theta')$, and let $\LL$ be the same set, but viewing elements as Lagrangians in $(M,\theta)$.
There is an $\calA_\infty$ quasi-isomorphism $\fuk(\LL') \rightarrow \fuk(\LL)$ which is the identity on the level of objects.
\end{lemma}
\begin{proof}
When constructing $\fuk(M,\theta)$, assume that any chosen Hamiltonian function for Floer data or perturbation data pertaining to $\fuk(\LL)$ vanishes near $M \setminus \Int(M')$.
Similarly, assume that any almost complex structure $J$ pertaining to the construction of $\fuk(\LL)$ 
is {\em contact type} near $\bdy M'$. 
That is, we have $J^*\theta = e^r dr$ on $\Op(\bdy M')$, where
$r$ is the collar coordinate defined using the flow of $Z_\theta$ which satisfies $\theta = e^r\theta|_{\bdy M'}$.
Then by the integrated maximum principle (\cite[Lemma 7.2]{abouzaid2010open}, see also \cite[Lemma 7.5]{seidelbook}), all pseudoholomorphic curves factoring into the construction of $\fuk(\LL)$ are actually contained in $\Int(M')$. This means that, for suitably correlated choices in the construction of $\fuk(M',\theta')$, 
we can arrange that $\fuk(\LL)$ and $\fuk(\LL')$ coincide .
\end{proof}

As a special case, if $\LL = \{L_0,L_1\}$ consists of two Lagrangians,
the question of whether $L_0$ and $L_1$ are quasi-isomorphic does not depend on whether we view them in $(M',\theta')$ or $(M,\theta)$.
By \S\ref{subsubsec:Independence of Floer data, strip-like ends, and perturbation data}, it also does not depend on any choice of Floer data, strip-like ends, or perturbation data.

\subsubsection{Liouville homotopies}\label{subsubec:Liouville_homotopies}

Recall that a homotopy of Liouville domains is simply a one-parameter family $(M,\theta_t)$, $t \in [0,1]$, where each $(M,\theta_t)$ is a Liouville domain.
In this situation, the Fukaya categories $\fuk(M,\theta_0)$ and $\fuk(M,\theta_1)$ are $\calA_\infty$ quasi-equivalent. 
As a slightly more specific statement, we have:
\begin{lemma}
Let $(M,\theta_t)$, $t \in [0,1]$, be a homotopy of Liouville domains, let 
$\LL_t$ be a smoothly varying set of Lagrangian branes in $(M,\theta_t)$.
There is an $\calA_\infty$ quasi-isomorphism $\fuk(\LL_0) \rightarrow \fuk(\LL_1)$ which is the obvious isotopy-following map on the level of objects.
\end{lemma}

\begin{proof}
By a version of Moser's argument,
we can find a smooth family of exact symplectic embeddings $\Phi:(M,\theta_t) \hookrightarrow (\wh{M},\wh{\theta_0})$ with $\Phi_0$ the inclusion map.
Here $(\wh{M},\wh{\theta_0})$ denotes the completion of $(M,\theta_0)$ which is obtained by attaching the conical end $(\bdy M \times [0,\infty), e^r\theta_0|_{\bdy M})$.
Note that each Lagrangian in $\LL_0$, viewed in $(\wh{M},\wh{\theta_0})$, is
Hamiltonian isotopic to a Lagrangian in $\Phi_1(\LL_1)$ via the family $\Phi_t(\LL_t)$.
Since Hamiltonian isotopy implies quasi-isomorphism in the Fukaya category, 
by Lemma \ref{lem:quasi-iso trick}
it suffices to produce an $\calA_\infty$ quasi-isomorphism between $\fuk(\Phi_1(\LL_1))$ and $\LL_1$.
Without loss of generality, we can assume that $\Phi_1^*(\wh{\theta_0}) = \theta_1$. 
The result then follows from Lemma \ref{lemma:Liouville subdomains}.
\end{proof}

\begin{remark}
In the presence of a B-field or bulk deformation, it is still true that Hamiltonian isotopic Lagrangians are quasi-isomorphic, provided that the isotopy is disjoint from the support of the two-form $\Omega$ or smooth cycle $\mho$.
For more general isotopies, one should more carefully take into account certain {\em bounding cochains}, but this lies outside the scope of this paper.
\end{remark}

\subsubsection{Almost complex structures}

Finally, if $\Jo$ and $\Jo'$ are two different choices of reference almost complex structure,
the two resulting $\calA_\infty$ categories are $\calA_\infty$ quasi-isomorphic.
Indeed, we can construct them in parallel as follows. 
Firstly, choose the same Hamiltonian terms and strip-like ends in both constructions.
Now suppose that $J$ and $J'$ are two corresponding families of almost complex structures associated to some Floer datum or perturbation datum.
We require that for some sufficiently large subdomain $M' \subset M$ we have
\begin{itemize}
\item
$Z_\theta$ is outwardly transverse along $\bdy M'$ and $\Int M'$ contains all of the relevant Lagrangians
\item 
$J$ and $J'$ coincide on $M'$ and are contact type near $\bdy M'$.
\end{itemize}
By the integrated maximum principle, the corresponding pseudoholomorphic curves cannot escape $M'$. We can therefore arrange that the $\calA_\infty$ operations in the two constructions coincide.

In fact, let $\calJctct$ denote the space of $(d\theta)$-compatible almost complex structures on $M$ which are contact type near $\bdy M$.
The construction of the Fukaya category described in \S\ref{subsec:the fukaya category of a Liouville domain} works equally well if we replace $\calJfix$ with $\calJctct$,
and in this case there is no need to single out any one almost complex structure.
However, this approach seems slightly less convenient when discussing Liouville domains with corners.

\section{Turning on a B-field}\label{sec:turning on a B-field}

In this section we discuss how to twist symplectic invariants by a closed two-form $\Omega$, often called a ``B-field" in the physics literature.
The basic idea is to consider the same pseudoholomorphic curves as in the untwisted case, but to count each curve with an extra weight factor determined by the integral of $\Omega$ over that curve. The resulting twisted invariants are sensitive to additional information about homology classes of curves and, as we will see, can sometimes detect qualitative features which are invisible to their untwisted cousins.
In the existing literature B-fields have been applied to
\begin{itemize}
\item
Lagrangian Floer theory in \cite{cho2008non,fukaya2011lagrangian} in the context of displaceability questions for Lagrangians
\item
Hamiltonian Floer theory in \cite{usher2011deformed} in the context of symplectic capacities and quasimorphisms 
\item
symplectic cohomology in \cite{ritternovikov,ritter2010deformations} in the context of obstructing Lagrangian embeddings.
\end{itemize}

\sss

\noindent \textbf{Coefficients.} In order to apply the twisting construction, we work over a field $\K$ which is equipped with an injective group homomorphism $H: \R \rightarrow \K^*$ from the additive group of real numbers to the multiplicative group of invertible elements in $\K$. 
We set $t := H(1)$ and more generally $t^r := H(r)$ for any $r \in \R$. 
For example, we could take $\K = \C$ and $H(r) = e^r$, or take $\K$ to be the field of rational functions in a formal variable $t$ with real exponents and coefficients in an auxiliary field.
\begin{remark}
Symplectic invariants with Novikov coefficients also fit into the context of B-fields,
with $\Omega$ given by the symplectic form itself. In this case $t$ is usually taken to be some kind of Novikov parameter with respect to which $\K$ is complete, although this is not technically necessary when working with exact symplectic manifolds.
\end{remark}

\subsection{The twisted Fukaya category}\label{subsec:the twisted Fukaya category}

In this subsection we discuss how to incorporate a B-field into the construction of the Fukaya category from \S\ref{sec:fukaya categories}.
With our intended applications in mind, we expedite the definition by restricting to Lagrangians which are disjoint from the support of the B-field (but see Remark \ref{rmk:lags with connections} for a more general framework). 
Let $(M,\theta)$ be a Liouville domain and let $\Omega$ be a closed two-form on $M$.
Pick Floer data, strip-like ends, and perturbation data as in the construction of $\fuk(M,\theta)$.
We define $\fuk_\Omega(M,\theta)$ as the $\calA_\infty$ category over $\K$ with:
\begin{itemize}
\item
objects given by closed exact Lagrangian branes $L$ in $(M,\theta)$ such that $\Op(L)$ is disjoint from the support of $\Omega$ 
\item 
for objects $L_0,L_1$, $\hom(L_0,L_1)$ is the $\K$-module generated by $\genHLL{L_0}{L_1}$
\item
for $k \geq 1$, objects $L_0,...,L_k$, and
$x_i \in \genHLL{L_{i-1}}{L_i}$ for $i = 1,...,k$, 
we set
\begin{align*}
 \twmu^k(x_k,...,x_1) :=  \sum_{\substack{x_0 \in \genHLL{L_0}{L_k}\\ u \in \calM(x_0,...,x_k)^{\oo}}} t^{\int u^*\Omega}\signu x_0.
\end{align*}
\end{itemize}

Since $\twfuk(M,\theta)$ is defined using the same pseudoholomorphic polygons as $\fuk(M,\theta)$, to confirm the $\calA_\infty$ structure equations we just need to check that the new term
$t^{\int u^*\Omega}$ behaves appropriately.
Namely, suppose we have $\bL = (L_0,...,L_k)$ and $\bx \in \gen(\bL)$, and $u_0$ and $u_1$ 
are once-broken pseudoholomorphic polygons which 
together form the boundary of a one-dimensional component of $\ovl{\calM}(\bx)$.
The integral of $\Omega$ over a broken curve still makes sense by summing over each component,
and it suffices to check that 
\begin{align*}
\int u_0^*\Omega = \int u_1^*\Omega.
\end{align*}
Since $\Omega$ is closed and vanishes near $L_0 \cup ... \cup L_k$, this follows easily from Stokes' theorem.

\begin{remark}\label{rmk:lags with connections}\hspace{1cm}
\begin{enumerate}
\item
By equipping Lagrangians with additional decorations, we can enlarge the class of objects of $\twfuk(M,\theta)$.
Although we will not need this generality for our main results, we briefly explain the idea for context.
We consider the $\calA_\infty$ category over $\K$ with
\begin{itemize}
\item
objects given by pairs $(L,\nu)$, where $L$ is a closed exact Lagrangian brane in $(M,\theta)$ and $\nu$ is a one-form on $L$ with $d\nu = \Omega|_L$
\item 
for objects $(L_0,\nu_0),(L_1,\nu_1)$, $\hom((L_0,\nu_0),(L_1,\nu_1))$ is the $\K$-module generated by $\genHLL{L_0}{L_1}$
\item
for $k \geq 1$, objects $(L_0,\nu_0)$,...,$(L_k,\nu_k)$, and
$x_i \in \genHLL{L_{i-1}}{L_i}$ for $i = 1,...,k$, 
we set
\begin{align*}
 \twmu^k(x_k,...,x_1) :=  \sum_{\substack{x_0 \in \genHLL{L_0}{L_k}\\ u \in \calM(x_0,...,x_k)^{\oo}}} t^{\int u^*\Omega - \hol(u)}\signu x_0.
\end{align*}
Here the term $\hol(u)$ is given by
\begin{align*}
\hol(u) := \sum_{i = 0}^k\int_{\gamma_i}u^*\nu_i,
\end{align*}
where $\gamma_0,...,\gamma_k$ denote the (ordered) boundary components of the domain of $u$.
\end{itemize}
The $\calA_\infty$ structure equations again follow from an application of Stokes' theorem.
Note that we are still excluding any Lagrangian $L \subset M$ for which $\Omega|_L$ is not exact.

\item
There is a yet further extension, considered by Cho in \cite{cho2008non}, which involves arbitrary closed exact Lagrangian branes in $(M,\theta)$.
Namely, we take complex coefficients $\K = \C$, and objects of the form $(L,\xi,\nabla)$, where $L$ is a closed exact Lagrangian brane, $\xi \rightarrow L$ is a complex line bundle, and $\nabla: C^\infty(\xi) \rightarrow C^\infty(T^*L \otimes \xi)$ is a complex connection whose curvature two-form coincides with $2\pi i \Omega|_L$. For two objects $(L_0,\xi_0,\nabla_0)$,$(L_1,\xi_1,\nabla_1)$,
we set
\begin{align*}
\hom((L_0,\xi_0,\nabla_0),(L_1,\xi_1,\nabla_1)) := \bigoplus_{x \in \genHLL{L_0}{L_1}} \hom_\C  ((\xi_0)_{x(0)},(\xi_1)_{x(1)}).
\end{align*}
The $\calA_\infty$ operations are defined as above, except that now
the term $\hol(u)$ is defined by composing the holonomies of the given connections along the boundary of $u$.
In particular, if $L|_\Omega$ is exact, we can pick $\xi$ to be the trivial bundle $L \times \C$ and use a connection of the form $\nabla(f) = df + f\nu$ for a one-form $\nu$.
That is, this category contains the previous version as a full subcategory.
\end{enumerate}
\end{remark}

\subsection{Twisted symplectic cohomology}

In this subsection we recall the construction of $\sh_\Omega(M,\theta)$, the symplectic cohomology of $(M,\theta)$ twisted by $\Omega$.
For further details and different applications we refer the reader to Ritter's papers \cite{ritternovikov,ritter2010deformations,ritter2013topological}.

\subsubsection{Symplectic cohomology formalism}\label{subsubsec:symplectic cohomology formalism}

We first recall the construction of ``ordinary" symplectic cohomology.
For a more thorough treatment we recommend any of the excellent surveys \cite{oancea2004survey,seidel2008biased,wendlbeginner,abouzaid2013symplectic}.
Let $(\wh{M},\wh{\theta})$ denote the completion of $(M,\theta)$, given by attaching the conical end $(\bdy M \times [0,\infty),e^r\theta|_{\bdy M})$.
Assume that the Reeb flow associated to the contact form $\theta|_{\bdy M}$ is nondegenerate.
For $\tau > 0$, let $\calHtau$ denote the space of Hamiltonians $H: \wh{M} \rightarrow \R$ such that $H|_{\Op(\wh{M} \setminus M)} = \tau e^r$.
We call $\tau$ the ``slope at infinity" of $H \in H_\tau$,
and we only allow $\tau > 0$ not equal to the period of any Reeb orbit of $\theta|_{\bdy M}$.
For such a $\tau$ and a nondegenerate time-dependent Hamiltonian $H \in C^{\infty}(S^1,\calHtau)$, the set of $1$-periodic orbits is finite and we denote it by $\calP_H$.
Let $\calJ$ denote the space of $(d\wh{\theta})$-compatible almost complex structures on $\wh{M}$ which are contact type, i.e. satisfy $J^*\theta = e^r dr$, on 
$\Op(\wh{M} \setminus M)$.
The special shapes at infinity of $H \in C^{\infty}(S^1,\calHtau)$ and $J \in C^{\infty}(S^1,\calJ)$ ensure the maximum principle needed to make sense of the Floer complex for generic $(H,J)$.
The space of (parametrized) Floer trajectories with asymptotics $\gamma_-,\gamma_+ \in \calP_H$ is given by
\begin{align*}
\wh{\calM}(\gamma_-,\gamma_+) := \left\{ u: \R \times S^1 \rightarrow \wh{M}\;:\; \bdy_s u + J(\bdy_t u - X_H) = 0,\; \lim_{s \rightarrow \pm \infty}u(s,\cdot) = \gamma_{\pm}\right\}.
\end{align*}
Let $\calM(\gamma_-,\gamma_+)$ denote the quotient of $\wh{\calM}(\gamma_-,\gamma_+)$ by the free $\R$-action which translates in the $s$ coordinate. 
The Floer complex is generated as $\K$-module by $\calP_H$, with differential $\delta$ given on a orbit $\gamma_+ \in \calP_H$ by
\begin{align*}
\delta(\gamma_+) := \sum_{\substack{\gamma_- \in \calP_H \\ u \in \calM(\gamma_-,\gamma_+)^\oo}} \signu \gamma_-.
\end{align*}
We denote this complex by $\cf(H)$, with the understanding that the differential also depends on an accompanying choice of $J$.

By standard Floer theoretic techniques, $\cf(H)$ is independent of $J$ up to chain homotopy equivalence.
However, unlike for the case of a closed symplectic manfiold, $\hf(H)$ does depend on $H$.
To remove this dependence, one considers {\em monotone continuation maps}.
Namely, there is a continuation map $\Phi: \cf(H_+) \rightarrow \cf(H_-)$ whenever the slope at infinity of $H_-$ is larger than that of $H_+$.
The construction of $\Phi$ depends on the following choices:
\begin{itemize}
\item a {\em monotone homotopy} from $H_-$ to $H_+$,
i.e. an $\R$-dependent family of slopes $\tau_s$ and Hamiltonians $H_s \in C^\infty(S^1,\calH_{\tau_s})$,
such that $H_s = H_{\pm}$ for $\pm s \gg 0$ and $\bdy_s \tau_s \leq 0$ for all $s$
\item
a family of almost complex structures $J_s \in C^\infty(S^1,\calJ)$ such that $J_s = J_\pm$ for $\pm s \gg 0$.
\end{itemize}
Here the inequality for $\tau_s$ ensures a maximum principle for solutions of the continuation map equation.
Given generic such choices and $\gamma_{\pm} \in \calP_{H_{\pm}}$,
we set
\begin{align*}
\calM(\gamma_-,\gamma_+) := \left\{ u: \R \times S^1 \rightarrow \wh{M}\;:\; \bdy_s u + J(\bdy_t u - X_{H_s}) = 0,\; \lim_{s \rightarrow \pm \infty}u(s,\cdot) = \gamma_{\pm}\right\}.
\end{align*}
Note that since $H_s$ is $s$-dependent there is no longer a translation $\R$-action.
The continuation map $\Phi: \cf(H_+) \rightarrow \cf(H_-)$ is then given on an orbit $\gamma_+ \in \calP_{H_+}$ by 
\begin{align*}
\Phi(\gamma_+) := \sum_{\substack{\gamma_- \in \calP_{H_-} \\ u \in \calM(\gamma_-,\gamma_+)^\oo}} \signu \gamma_-.
\end{align*}
Standard Floer theoretic techniques show that 
\begin{itemize}
\item
$\Phi$ is a chain map
\item
up to chain homotopy, $\Phi$ is independent of the choice of monotone homotopy
\item 
the composition of two continuation maps is chain homotopic to a continuation map.
\end{itemize}

Finally, the symplectic cohomology of $(M,\theta)$ is defined as the direct limit 
\begin{align*}
\sh(M,\theta) := \lim_{\tau \rightarrow \infty} \hf(H)
\end{align*}
over generic $H \in C^{\infty}(S^1,\calHtau)$,
where the connecting maps are given by monotone continuation maps.
By the direct limit formalism, the result is manifestly independent of any choice of Hamiltonian or almost complex structure within the allowed class.
Less obviously, it turns out to be invariant under arbitrary symplectomorphisms of $(\wh{M},d\wh{\theta})$ (see \cite[\S 2c]{abouzaid2010altering} or the discussion in \ref{subsec:some functoriality properties} below).

\subsubsection{Adding a twist}

The twisted Floer complex $\cf_\Omega(H)$ coincides with $\cf(H)$ as a $\K$-module,
but with the twisted differential given on an orbit $\gamma_+ \in \calP_H$ by
\begin{align*}
\delta_\Omega(\gamma_+) := \sum_{\substack{\gamma_- \in \calP_H \\ u \in \calM(\gamma_-,\gamma_+)^\oo}}  t^{\int u^*\Omega} \signu \gamma_-.
\end{align*}
Since $\Omega$ is closed, one can check using Stokes' theorem that the new contributions agree on two cancelling ends of a one-dimensional component of $\ovl{\calM}(\gamma_-,\gamma_+)$, so we again have $\delta_\Omega^2  = 0$.
Similarly, for $H_- \in \calH_{\tau_-}$ and $H_+ \in \calH_{\tau_+}$ with $\tau_- \geq \tau_+$, the twisted monotone continuation map $\Phi_\Omega: \cf(H_+) \rightarrow \cf(H_-)$ is given on an orbit $\gamma_+ \in \calP_{H_+}$ by 
\begin{align*}
\Phi_\Omega(\gamma_+) := \sum_{\substack{\gamma_- \in \calP_{H_-} \\ u \in \calM(\gamma_-,\gamma_+)^\oo}} t^{\int u^*\Omega}\signu \gamma_-.
\end{align*}
By another application of Stokes' theorem, this satisfies $\Phi_\Omega \circ \delta_\Omega = \delta_\Omega \circ \Phi_\Omega$.

\begin{remark}\label{rem:local system}
The two-form $\Omega$ defines a certain local system $\K_{\Omega}$ on the free loop space $\calL M$ of $M$.
Namely:
\begin{itemize}
\item
to each point $p \in \calL M$ we associate a copy $(\K_\Omega)_p$ of $\K$
\item
to each smooth path $\eta: [0,1] \rightarrow \calL M$ from $p$ to $p'$ we associate the monodromy homomorphism
\begin{align*}
(\K_\Omega)_{p} \rightarrow (\K_\Omega)_{p'},\;\;\;\;\; k \mapsto t^{\int u^*\Omega}k,
\end{align*}
viewing $\eta$ as the cylinder $u: [0,1] \times S^1 \rightarrow M$.
\end{itemize}
It is thus natural to view $\sh_\Omega(M,\theta)$ as the symplectic cohomology of $(M,\theta)$ with coefficients in the local system $\K_{\Omega}$.
\end{remark}

\subsubsection{Independence of $\Omega$}\label{subsubsec:independence of Omega}

A basic fact about $\sh_{\Omega}(M,\theta)$ is that it depends only on the cohomology class $[\Omega] \in H^2(M;\R)$. In particular, it agrees with the untwisted version when $\Omega$ is exact. To see this, consider the effect of adding an exact two-form to $\Omega$, say $d\alpha$ for $\alpha$ a $1$-form on $M$.
The new twisted differential is then given by
\begin{align*}
\delta_{\Omega + d\alpha}(\gamma_+) &=   \sum_{\substack{\gamma_- \in \calP_{H} \\ u \in \calM(\gamma_-,\gamma_+)^\oo}}\left(t^{\int u^* \Omega + \int u^* d\alpha}\right)  \gamma_-\\
&=   \sum_{\substack{\gamma_- \in \calP_{H} \\ u \in \calM(\gamma_-,\gamma_+)^\oo}}\left(t^{\int u^* \Omega} \right)\left(t^{\int {\gamma_+}^*\alpha - \int {\gamma_-}^*\alpha}\right)  \gamma_-,
\end{align*}
i.e.
\begin{align*}
\delta_{\Omega + d\alpha}\left(t^{-\int {\gamma_+}^* \alpha}\gamma_+\right) =   \sum_{\substack{\gamma_- \in \calP_{H} \\ u \in \calM(\gamma_-,\gamma_+)^\oo}}\left(t^{\int u^* \Omega} \right) \left(t^{-\int {\gamma_-}^* \alpha} \gamma_- \right),
\end{align*}
so the twisting disappears after the change of basis 
$\gamma \mapsto t^{-\int \gamma^* \alpha}\gamma$.

\section{Bulk deformations}\label{sec:bulk deformations}

In this section we explain how to ``bulk deform" the symplectic invariants of a Liouville domain $(M,\theta)$.
This is analogous to twisting by a B-field, except with $\Omega$ now replaced by a closed 
$l$-form for $l > 2$. 
Although formally similar to twisting by a closed two-form, there are also some important differences which stem from the fact that a two-form can be integrated over an isolated curve while an $l$-form can only be integrated over a $(l-2)$-dimensional family of curves.
Indeed, whereas twisting incorporates additional information about homology classes of rigid curves, bulk deforming adds certain contributions from non-rigid curves.
As we will see, the added freedom to use higher index classes 
can be used to detect qualitative features which are invisible in the ordinary or twisted cases.

In order to probe higher index curves, we consider moduli spaces of pseudoholomorphic curves with interior marked points.
Rather than working directly with a closed $l$-form, for technical reasons we find it more convenient to take the Poincar\'e dual perspective and work 
with a smooth codimenson $l$ cycle $\mho$.
More precisely, $\mho$ is a smooth oriented $(2n-l)$-dimensional manifold with boundary, equipped with a smooth map $i_{\mho}: (\mho,\bdy \mho) \rightarrow (M,\bdy M)$.\footnote{By abuse of notation, we will often suppress $i_{\mho}$ from the notation
and speak of $\mho$ as if it were an embedded submanifold in $M$.
}
It then makes sense to count pseudoholomorphic curves in $M$ which become rigid after requiring some number of interior points to pass through $\mho$.
By appropriately combining these counts over any number of point constraints, 
we produce algebraic structures which mimic the undeformed versions.

Bulk deformations were first introduced as part of a very general Lagrangian Floer theory package in the seminal work of Fukaya--Oh--Ohta--Ono. Further applications to displaceability questions in toric manifolds are given in \cite{fukaya2011lagrangian}.
Usher also implements bulk deformations for the Hamiltonian Floer theory of closed symplectic manifolds in \cite{usher2011deformed}.
A different but closely related approach is taken by Barraud--Cornea in \cite{barraud2007lagrangian},
where the authors design a spectral sequence to extract information from higher index curves.

Our goal in this section is to construct the bulk deformed Fukaya category and bulk deformed symplectic cohomology of a Liouville domain.
In this setting we can give a direct treatment of transversality via Hamiltonian perturbations, following Seidel's approach in \cite{seidelbook}.
Although much of our discussion is likely well-known or expected by experts, we hope to complement the existing literature and give a self-contained treatment.

\sss

\noindent \textbf{Coefficients.} We assume that the codimension $l$ of $\mho$ is greater than two and {\em even} (see however Remark~\ref{rmk:parity}), and we work over a 
graded ring of the form $\KL := \KL_0[\hbar,\hbar^{-1}]$,
where $\KL_0$ is a field of characteristic zero and $\hbar$ is a formal variable of degree $2-l$.
The characteristic assumption is needed because fractional coefficients will appear.
The fact that $\hbar$ has nonzero degree will allow us to compensate for index shifts caused by extra point constraints and thereby produce a $\Z$ grading.
Note that the presence of an element of nonzero degree precludes $\KL$ from being a field.

\begin{remark}\label{rmk:parity}
For our main geometric applications we only consider the case that $l$ is even, and we will need $\hbar$ to be invertible (c.f. the proof of Proposition~\ref{prop:quasi-isomorphism criterion bulk deformed}).
However, for the general construction of bulked deformed invariants in this section we can equally well allow $l$ to be odd, and take $\KL$ to be any graded commutative ring having an element $\hbar$ of degree $2-l$.
The main difference in this case is the behavior of signs, hence the graded commutativity assumption, which necessitates the relation $\hbar^2 = 0$.
In particular, note that having both $l$ odd and $\hbar$ invertible are incompatible if $\KL_0$ is a nontrivial field, since then $1 = \hbar \hbar^{-1}\hbar\hbar^{-1} = -\hbar^2 \hbar^{-2} = 0$.
\end{remark}
\begin{remark} 
The maps we will write down are a priori infinite power series in $\hbar$, suggesting that we should complete $\KL$ with respect to $\hbar$. 
However, all but finitely many of these terms will be forced to vanish for degree reasons.
\end{remark}

\subsection{The bulk deformed Fukaya category}\label{subsec:bd fuk}

In this subsection we sketch a construction of the bulk deformed Fukaya category of $(M,\theta)$ via coherent perturbations.
Our perturbation scheme is formally similar to the one from \S\ref{sec:fukaya categories}, but with perturbations now depending on the locations of the interior marked points. 
Similar to the case for B-fields, we will restrict our attention to Lagrangians in $(M,\theta)$ which are disjoint from $\mho$.
One could also consider a more general class of objects along the lines of Remark \ref{rmk:lags with connections}, leading to the notion of {\em bounding cochains}
from \cite{fukaya2009lagrangian}, but this lies outside the scope of this paper.

\subsubsection{The moduli space $\calsphere_{q+1}$ and its compactification}
\label{subsubsec:the sphere moduli space}

As a warmup, we begin by discussing the moduli space of marked Riemann spheres and its Deligne--Mumford compactification.
For $q+1 \geq 3$, let $\calsphere_{q+1}$ denote the moduli space of Riemann spheres with $q+1$ ordered marked points, modulo biholomorphisms.
We declare the first marked point to be negative and the rest to be positive.
Note that $\calsphere_{q+1}$ is a smooth manifold of dimension $2q-4$.
There is also a universal family $\calsphereuniv_{q+1} \rightarrow \calsphere_{q+1}$ where the fiber
$\calsphereuniv_r$ over $r \in \calsphere_{q+1}$ represents the corresponding marked Riemann sphere.
Concretely, we can take
\begin{align*}
\calsphere_{q+1} = \conf_{q+1}(\CP^1)/\pslc,\;\;\;\;\; \calsphereuniv_{q+1} = \conf_{q+1}(\CP^1) \times_{\pslc} \CP^1.
\end{align*}
Note that we do not puncture the marked points in $\calsphereuniv_{q+1}$, and indeed they will play a slightly different role from the boundary marked points of $\calRuniv_{k+1}$.

Let $\calTsphere{q+1}$ denote the set of stable trees with $q+1$ ordered external edges.
We view an element of $\calTsphere{q+1}$ as a stable tree $T$ where:
\begin{itemize}
\item
the first external edge is distinguished as the root
\item
the remaining $q$ external vertices are equipped with an ordering.
\end{itemize}
As we will recall, $\calTsphere{q+1}$ models the stratification structure of the Deligne--Mumford compactification of $\calsphere_{q+1}$.

We compactify $\calsphere_{q+1}$ by allowing spheres to acquire certain nodes, and therefore 
any given stratum of $\calsphereovl_{q+1}$ is modeled on a product of factors of the form $\calsphere_{q'+1}$ for various $q' \leq q$. However, since the nodal points of each factor are not canonically ordered, it will be more natural to consider marked points with more general labels.
Namely, we consider Riemann spheres with marked points labeled by certain sets of tree edges.
More specifically, for $T \in \calTsphere{q+1}$ and $v \in V_\ii(T)$, let $\calsphere_{E(v)}$ be defined just like $\calsphere_{|v|}$, except that
\begin{itemize}
\item
 the marked points are indexed by the set $E(v)$ instead of being ordered
 \item
 the marked point indexed by the incoming edge is declared to be negative and the rest of the marked points are positive.
\end{itemize}
Evidently $\calsphere_{E(v)}$ is equivalent to $\calsphere_{|v|}$, although not canonically unless we pick an ordering of $E(v)$.
However, note that the subset of marked points of $\calsphere_{E(v)}$ which are indexed by external edges inherits a canonical ordering from that of $T$.
In particular, if $T_0 \in \calTsphere{q+1}$ denotes the tree with a unique internal vertex $v_0$, we freely 
identify $\calTsphere{E(v_0)}$ with $\calTsphere{q+1}$.
As a shorthand, let $\visphere$ denote the union of $V_\ii(T)$ over all $T \in \calTsphere{q+1}$ and $q+1 \geq 3$.

For each $v \in \visphere$, we endow the universal family $\calsphereuniv_{E(v)} \rightarrow \calsphere_{E(v)}$ with a choice of {\em universal cylindrical ends}, meaning pairwise disjoint fiberwise holomorphic embeddings
\begin{align*}
\epsilon_e'&: \calsphere_{E(v)} \times \R_\pm \times S^1 \hookrightarrow \calsphereuniv_{E(v)},
\;\;\;\;\; e \in E(v),
\end{align*}
with $-$ for the incoming edge and $+$ for the remaining edges, 
such that $\lim_{s \rightarrow \pm\infty}\epsilon'_e(s,\cdot)$ is the marked point indexed by $e$.
The situation closely parallels the case of strip-like ends for $\calRuniv_{k+1} \rightarrow \calR_{k+1}$, except for an additional $S^1$ ambiguity coming from rotations of $\R_\pm \times S^1$.
To remove the $S^1$ ambiguity, we pick an asymptotic marker\footnote{Recall that an {\em asymptotic marker} at a marked point is a choice of half line at the tangent space to that point.}
 at the negative marked point $p_-$ of $\calsphereuniv_r$ which
smoothly varies over $r \in \calsphere_{E(v)}$. 
The asymptotic marker at $p_-$ then naturally induces asymptotic markers at each the remaining marked points of $\calsphereuniv_r$.
Namely, identify $\calsphereuniv_r \setminus \{p_-,p_e\}$ with the cylinder $\R \times S^1$ and
take the asymptotic marker at $p_e$ which lies on the same line $\R \times \{\text{const}\}$ as the asymptotic marker at $p_-$.
Finally, for each cylindrical end, we require that $1 \in S^1$ matches up with the relevant asymptotic marker in the limit as $s \rightarrow \pm \infty$.

For $T \in \calTsphere{q+1}$ and $\e > 0$ small, set 
\begin{align*}
\calsphereovl_T := \prod_{v \in V_\ii(T)} \calsphere_{E(v)},\;\;\;\;\; \calsphereovl_T^\e := \calsphereovl_T \times \left(\D^2_\e\right)^{E_\ii(T)},
\end{align*}
where $\D^2_\e$ denotes the open disk of radius $\e$. 
We again use $\rho_e \in \D^2_\e$ to denote the gluing parameter corresponding to $e \in E_\ii(T)$,
viewing $\D^2_\e$ as $(-\e,0] \times S^1$ with $\{0\} \times S^1$ collapsed to a point.
Using our universal cylindrical ends, for each $e \in V_\ii(T)$ we have a gluing map
\begin{align*}
\phi_{T,e}: \{r \in \calsphereovl^\e_T\;:\; \rho_e  \neq 0\} \rightarrow \calsphereovl_{T/e}^\e,
\end{align*}
where we glue at the two marked points which are indexed by $e$.
The gluing procedure parallels the case of strip-like ends, except that the extra $S^1$ factor of $\rho_e$ corresponds to relative rotations of the cylindrical ends.
We then define the Deligne--Mumford compactification of $\calsphere_{q+1}$ as a topological space by
\begin{align*}
\calsphereovl_{q+1} := \left(\coprod_{T \in \calTsphere{q+1}} \calsphereovl_T^\e\right) /\sim,
\end{align*}
where $r \sim \phi_{T,e}(r)$ for any $r$ in the domain of $\phi_{T,e}$.
Elements of $\calsphereovl_{q+1}$ are {\em stable broken spheres}.
The tree $T_0$ corresponds to the open stratum $\calsphereovl_{T_0} = \calsphere_{q+1}$,
and in general the (real) codimension of the stratum $\calsphereovl_T$ is {\em twice} the number of internal edges of $T$.

\subsubsection{The moduli space $\calR_{k+1,q}$ and its compactification}\label{subsubsec:The moduli space Rkq and its compactification}

Next, for $k+1 + 2q \geq 3$, let $\calR_{k+1,q}$ denote the moduli space of Riemann disks with $k+1$ ordered boundary marked points and $q$ ordered interior marked points, modulo biholomorphisms.
As before, we ask that the order of the boundary marked points respects the boundary orientation.
Note that $\calR_{k+1,q}$ is a smooth manifold of dimension $k+2q - 2$, and for $q=0$ this agrees with $\calR_{k+1}$ from $\S\ref{subsubsec:stable disks}$.
We declare the first boundary marked point to be negative and all other marked points to be positive. 
Let $\calRuniv_{k+1,q} \rightarrow \calR_{k+1,q}$ denote the universal family where the fiber $\calRuniv_r$ over $r \in \calR_{k+1,q}$ represents the Riemann disk $r$ with its boundary marked points punctured and its interior marked points intact.

Let $\calTR{k+1,q}$ denote the set of planted trees $T$ with
\begin{itemize}
\item
a partition of the edges into two types, called ``plain" and ``round",
such that the root edge is plain and there are precisely $k$ plain leaves and $q$ {\em ordered} round leaves
\item
the plain edges form a subtree $T_\pl$ which is further equipped with a ribbon structure, and we require each of its leaves to be also external vertices for $T$
\item {\em stability:} for every internal vertex $v$, we have:
\begin{itemize}
\item if $v \in V_\ipl(T)$, then $|v|_\pl + 2|v|_\rd \geq 3$
\item if $v \in V_\ird(T)$, then $|v|_\rd \geq 3$.
\end{itemize}
\end{itemize}
Here we utilize the following notation:
\begin{itemize}
\item
 $V_\ipl(T)$ denotes the internal vertices of the subtree $T_\pl$ and $V_\ird(T)$ denotes the remaining internal vertices of $T$
 \item
 $E_\pl(v)$ and $E_\rd(v)$ denote the plain and round edges respectively which are incident to $v$ (note that $E_\pl(v)$ is empty unless $v$ is plain, i.e. lies in the subtree $T_\pl$)
 \item 
 $|v|_\pl := \left |E_\pl(v)\right|$ and  $|v|_\rd := \left |E_\rd(v)\right|$.
\end{itemize}
As usual we orient the edges of $T$ away from the root, and this induces an absolute ordering of the plain edges incident to each vertex.
As we will explain, $\calTR{k+1,q}$ models the stratification structure of the compactification of $\calR_{k+1,q}$,
with plain edges indexing boundary nodes (if the edge is internal) and boundary marked points (if external) and round edges indexing interior nodes (if internal) and interior marked points (if external).
In particular, $V_\ipl(T)$ and $V_\ird(T)$ will index disk and sphere components respectively,
hence the two different stability conditions.

As a shorthand, let $\vipR$ denote the union of $V_\ipl(T)$ over all $T \in \calTR{k+1,q}$ and $k+1+2q \geq 3$,
and define $\virR$ similarly.
For $v \in \virR$, we define $\calN_{E(v)}$ as in \S\ref{subsubsec:the sphere moduli space}, and we equip each family $\calsphereuniv_{E(v)} \rightarrow \calsphere_{E(v)}$ with fiberwise cylindrical ends.
Similarly, for $v \in \vipR$, let $\calR_{|v|_\pl,E_\rd(v)}$ be defined just like 
$\calR_{|v|_\pl,|v|_\rd}$ except that the interior marked points are indexed by the set $E_\rd(v)$ instead of  being ordered.
We pick universal strip-like ends
\begin{align*}
\epsilon_0&: \calR_{|v|_\pl,E_\rd(v)} \times \R_{-} \times [0,1] \hookrightarrow \calRuniv_{|v|_\pl,E_\rd(v)}\\
\epsilon_1,...,\epsilon_{k}&: \calR_{|v|_\pl,E_\rd(v)} \times \R_{+} \times [0,1] \hookrightarrow \calRuniv_{|v|_\pl,E_\rd(v)}
\end{align*}
and cylindrical ends
\begin{align*}
\epsilon_e': \calR_{|v|_\pl,E_\rd(v)} \times \R_{+} \times S^1 \hookrightarrow \calRuniv_{|v|_\pl,E_\rd(v)},\;\;\;\;\; e \in E_\rd(v)
\end{align*}
which are asymptotic to the corresponding boundary punctures and interior marked points.
Regarding the $S^1$ ambiguity for the cylindrical ends, observe that there are naturally induced asymptotic markers at the interior marked points of $\calRuniv_r$ for all $r \in \calR_{|v|_\pl,E_\rd(v)}$.
Namely, identify the interior of $\calRuniv_r \setminus \{p_e\}$ with $\D^2_1 \setminus \{0\}$, and then take the 
asymptotic marker at $p_e$ which points towards the negative boundary puncture.
As before, we require the cylindrical ends to align with these asymptotic markers.

For $T \in \calTR{k+1,q}$ and $\e > 0 $ small, set
\begin{align*}
\calRovl_T &:= \left(\prod_{v \in V_\ipl(T)} \calR_{|v|_\pl,E_\rd(v)}\right) \times \left(\prod_{v \in V_\ird(T)} \calsphere_{E(v)} \right),\\
\calRovl_T^\e &:= \calRovl_T  \times (-\e,0]^{E_\ipl(T)} \times \left( \D^2_\e\right)^{E_\ird(T)}.
\end{align*}
Using our universal strip-like and cylindrical ends, for each $e \in V_\ii(T)$ we have a gluing map 
\begin{align*}
\phi_{T,e}: \{ r \in \calRovl^\e_T\;:\; \rho_e \neq 0\} \rightarrow \calRovl^\e_{T/e}.
\end{align*}
We then define the compactification of $\calR_{k+1,q}$ as a topological space by
\begin{align*}
\calRovl_{k+1,q} := \left( \coprod_{T \in \calTR{k+1,q}} \calRovl_T^\e \right) / \sim,
\end{align*}
where $r \sim \phi_{T,e}(r)$ for any $r$ in the domain of $\phi_{T,e}$.
Observe that the codimension of the stratum $\calRovl_T$ is given by
$|E_\ipl(T)| + 2|E_\ird(T)|$.

\begin{remark}
Informally, $\calR_{k+1,q}$ has two new sources of noncompactness compared with $\calR_{k+1}$:
\begin{enumerate}
\item
an interior marked point can drift to the boundary (this is a codimension one phenomenon)
\item
two interior marked points can collide (this is a codimension two phenomenon).
\end{enumerate}
The compactification $\ovl{\calR}_{k+1,q}$ ``solves" the first issue by blowing a disk bubble with a single interior marked point, and the second issue by blowing a sphere bubble with two interior marked points.
 \end{remark}

\begin{remark}
For each $T \in \calTR{k+1,q}$, the stratum $\calRovl_T$
is equivalent to a product of factors of the form $\calR_{k'+1,q'}$ and $\calsphere_{q'}$, for various $k',q'$.
Even though the sphere moduli spaces $\calsphere_{q'}$ are part of the compactification structure of $\ovl{\calR}_{k+1,q}$, they will not make a direct appearance in the definition of the bulk deformed Fukaya category. 
This is because they appear with codimension at least two, and hence do not generically occur in the curve counts of interest. 
Nevertheless, since ruling out sphere bubbles is based on transversality, the spaces 
$\calsphere_{q'}$ must be incorporated into our general perturbation scheme.
\end{remark}

\subsubsection{Lagrangian labels}
Let $\bL = (L_0,...,L_k)$ be a tuple of closed exact Lagrangians in $(M,\theta)$.
We say that $T \in \calTR{k+1,q}$ is {\em labeled by $\bL$} if the underlying plain tree $T_\pl \in \calTR{k+1}$ is.
For any $v \in V_\ipl(T)$, let $\bL_v$ denote the labels encountered in order as we go around the plain edges incident to $v$.
The definition of $\calR_{q}(\bL)$ is the same as $\calR_{k+1,q}$ except that for each disk the boundary segments are labeled in order by $\bL$.
We similarly define the compactification $\ovl{\calR}_{q}(\bL)$,
universal family $\calRuniv_q(\bL) \rightarrow \calR_q(\bL)$, the set-indexed version $\calR_{E(v)}(\bL)$ for $v \in \vipR$, and so on.

\subsubsection{Consistent universal strip-like ends and disk-like neighborhoods}
\label{subsubsec:bdfuk consistent universal strip-like ends and disk-like neighborhoods}

As in \S\ref{subsubsec:consistent strip-like ends}, we will need our universal strip-like ends to be {\em consistent}.
Namely, for $T\in \calTR{k+1,q}$, strip-like ends can be glued via the map
\begin{align*}
\phi_T: \calRovl_T \times (-\e,0)^{E_\ipl(T)} \times \left( \D_\e^2 \setminus \{0\}\right)^{E_\ird(T)} \rightarrow \calR_{k+1,q}.
\end{align*}
The image of $\phi_T$ is therefore equipped with two a priori different families of strip-like ends, one coming from the universal choice for $\calR_{k+1,q}$ and one coming from gluing the universal choices for each $\calR_{|v|_{pl},E_\rd(v)}$ for $v \in V_\ipl(T)$
and $\calR_{E(v)}$ for $v \in V_\ird(T)$.
We require that these coincide in a neighborhood of $\calRovl_T$, for all $T \in \calTR{k+1,q}$.

Regarding the cylindrical ends, we now ignore their parametrizations and consider only their images. 
More specifically, for a Riemann surface $S$ and a cylindrical end $\epsilon': \R_{\pm}\times S^1 \hookrightarrow S$, there is a neighborhood
of the corresponding interior marked point $p \in S$ of the form $U = \op{Im}(\epsilon') \cup \{p\}$.
We say that $U$ is the induced ``disk-like neighborhood" for $p$. Note that our universal cylindrical ends induce disk-like neighborhoods $U_1,...,U_q \subset \calRuniv_r$ which vary smoothly
for $r \in \calR_{|v|_\pl,E_\rd(v)}$ and $v \in \vipR$.
Again, we have two a priori different families of disk-like neighborhoods on the image of each gluing map $\phi_T$, and we require these to coincide in a neighborhood of $\calRovl_T$, for all $T \in \calTR{k+1,q}$.

As already pointed out, for any $T \in \calTR{k+1,q}$ and $v \in V_i(T)$, the subset of external edges in $E_\rd(v)$ inherits an ordering from $T$. In particular, when all of the round edges incident to a plain vertex $v$ are external, we get a canonical identification of $\calR_{|v|_\pl,E_\rd(v)}$ with $\calR_{|v|_\pl,|v|_\rd}$.
In this case we will always assume that the universal choices for $\calR_{|v|_\pl,E_\rd(v)}$ are preserved under this identification.
This assumption will be implicitly used when verifying the $\calA_\infty$ equations for $\fuk_\mho(M,\theta)$, allowing us to conclude that certain curve counts defined by a priori different perturbation data indeed coincide.

\subsubsection{Consistent universal perturbation data}\label{subsubsec:bdfuk consistent universal perturbation data}

Finally, we need consistent universal perturbation data.
Similar to \S\ref{subsubsec:Floer data}, assume we have chosen a Floer datum for every pair of closed exact Lagrangians in $(M,\theta)$ which are disjoint from $\mho$. 
Let $S$ represent an element of $\calR_{k+1,q}(\bL)$. A perturbation datum for $S$
is a pair $(K,J)$ with $K \in \Omega^1(S,\calH)$ and $J \in C^\infty(S,\calJfix)$
satisfying the same conditions as in \S\ref{subsubsec:Consistent universal perturbation data}.
We also impose an extra condition regarding the interior marked points:
\begin{itemize}
\item 
on each disk-like neighborhood of $S$ we have $K \equiv 0$ and $J \equiv \Jo$.
\end{itemize}
Similarly, for $S$ representing an element of $\calsphere_{q}$, a perturbation datum for $S$ consists of 
$K \in \Omega^1(S,\calH)$ and $J \in C^\infty(S,\calJfix)$
such that:
\begin{itemize}
\item
on each disk-like neighborhood of $S$ we have $K \equiv 0$ and $J \equiv \Jo$.
\end{itemize}

We choose fiberwise perturbation data on:
\begin{itemize}
\item
$\calsphereuniv_{E(v)} \rightarrow \calsphere_{E_(v)}$ for all $v \in \virR$
\item 
$\calRuniv_{|v|_\pl,E_\rd(v)}(\bL)\rightarrow \calR_{|v|_\pl,E_\rd(v)}(\bL)$ for all $\bL$ and $v \in \vipR$.
\end{itemize}
As before, gluing via the map $\phi_T$ for $T \in \calTR{k+1,q}$ 
results in a priori distinct families of perturbation data, and
we say our choices are {\em consistent} if they satisfy the analogues of the two conditions stated at the end of \S\ref{subsubsec:Consistent universal perturbation data}.
We also require our choices to be invariant under the identifications mentioned at the end of \S\ref{subsubsec:bdfuk consistent universal strip-like ends and disk-like neighborhoods}.

\subsubsection{The moduli spaces $\calM_{q;\mho}$ and $\calM_{q;\mho}(\bx)$}

Let $S$ be a fixed Riemann sphere with $q \geq 3$ ordered marked points.
Assume that $S$ is equipped with the perturbation datum $(K,J)$ induced from our universal choices.
Let $Y$ denote the $(d \theta)$-dual of $K$ as in \S\ref{subsubsec:pseudoholomorphic polygons}.
By a {\em pseudoholomorphic sphere with domain $S$} we mean a map $u: S \rightarrow M$ which satisfies $(Du - Y)^{0,1} = 0$.
Let $\calM_S$ denote the space of pseudoholomorphic spheres with domain $S$, and set
\begin{align*}
\calM_q := \{(r,u)\;:\; r \in \calsphere_q,\; u \in \calM_{\calsphereuniv_r}\}. 
\end{align*}
Evaluating at the marked points gives a map
\begin{align*}
\ev_q: \calM_q \rightarrow M^{\times q},
\end{align*}
and we set
\begin{align*}
\calM_{q;\mho} := \calM_q \underset{\ev_q,i_{\mho}^{\times q}} \times \mho^{\times q},
\end{align*}
where the right hand sides denote the fiber product with respect to the maps $\ev_q$ and $i_{\mho}^{\times q}$.

Similarly, now suppose that $k+1 + 2q \geq 3$, and let $S$ be a fixed Riemann disk with $k+1$ boundary punctures, $q$ ordered interior marked points, and Lagrangian labels $\bL$.
Assume $S$ is equipped with the strip-like ends $\epsilon_0,...,\epsilon_k$ and perturbation datum $(K,J)$ induced from our universal choices. By a {\em pseudoholomorphic polygon with domain $S$} we mean a map $u: S \rightarrow M$ which sends each boundary component of $S$ to its corresponding Lagrangian label and satisfies $(Du - Y)^{0,1} = 0$.
For $\bx \in \gen(\bL)$, we say that $u$ has {\em asymptotics $\bx$} if
\begin{align*}
\lim_{s \rightarrow \pm \infty} (u \circ \epsilon_i)(s,\cdot) = x_i\;\;\;\;\; \text{for}\;\; i= 0,...,k.
\end{align*} 
Let $\calM_S(\bx)$ denote the space of pseudoholomorphic polygons with domain $S$ and asymptotics $\bx$, and set
\begin{align*}
\calM_q(\bx) := \{(r,u)\;:\; r \in \calR_q(\bL),\; u \in \calM_{\calRuniv_r}(\bx)\}.
\end{align*}
Evaluating at the interior marked points gives a map
\begin{align*}
\ev_q: \calM_q(\bx) \rightarrow M^{\times q},
\end{align*}
and we set
\begin{align*}
\calM_{q;\mho}(\bx) := \calM_q(\bx) \underset{\ev_q,i_{\mho}^{\times q}} \times \ \mho^{\times q}.
\end{align*}

The analogue of Proposition \ref{prop:moduli spaces are topological manifolds} in this setting is:
\begin{prop}\label{prop:bd moduli spaces are regular}
For generic perturbation data, the moduli spaces $\calM_{q;\mho}$ and $\calM_{q;\mho}(\bx)$ are regular and hence smooth manifolds.
\end{prop}
\noindent The dimension formula for $\calM_{q;\mho}(\bx)$ is
\begin{align*}
\dim \calM_{q;\mho}(\bx) = |x_0| - |x_1| - ... - |x_k| + k - 2 - q(l-2)
\end{align*}
(recall that $l$ is the codimension of $\mho$ in $M$).
We also define $\calM_{E(v);\mho}$ for $v \in \virR$ and $\calM_{E_\rd(v);\mho}(\bx)$ for $v \in \vipR$ in the same manner by using the relevant moduli spaces with set-indexed interior marked points and the corresponding induced strip-like ends and perturbation data.

\begin{proof}[Proof sketch of Proposition~\ref{prop:bd moduli spaces are regular}]
This is closely analogous to Proposition~\ref{prop:moduli spaces are topological manifolds} and follows standard (``classical'') transversality techniques in Floer theory, so we give only a sketch, focusing on the role played by the point constraints in $\mho$. We discuss regularity for $\calM_{q;\mho}(\bx)$, the case $\calM_{q;\mho}$ being similar and in fact simpler since there are no strip-like ends to worry about.

We begin by recalling the proof scheme of Proposition~\ref{prop:moduli spaces are topological manifolds}, which requires setting up some Fredholm theory.
Fix a collection of Lagrangian labels $\bL = (L_0,\dots,L_k)$ and corresponding asympotics $\bx = (x_0,\dots,x_k)$ as in \S\ref{subsubsec:pseudoholomorphic polygons}.
We wish to show that $\calM(\bx)$ is regular for generic perturbation data $(K,J)$, with $K \in \Omega^1(S,\calH)$ and $J \in C^\infty(S,\calJfix)$.
With respect to some fixed reference choice $(K_0,J_0)$, consider a curve $(r_0,u_0) \in \calM(\bx)$, where $r_0 \in \calR_{k+1}$ and $u_0 \in \calM_{\calRuniv_{r_0}}(\bx)$.
Let $S_0 := \calRuniv_{r_0}$ denote the domain of $u_0$, which is a Riemann disk with $k+1$ boundary punctures, and let $j_0$ denote its associated almost complex structure.

Let $\calB$ denote the Banach space of maps $u: S_0 \rightarrow M$ of Sobolev class $W^{1,p}$ for some $p > 2$ (so in particular these maps are continuous) which satisfy the Lagrangian boundary conditions corresponding to $\bL$ and converge to the corresponding Hamiltonian chords $\bx$ at the punctures.
Let $\calE_{S_0} \rightarrow \calB$ denote the Banach vector bundle whose fiber $(\calE_{S_0})_u$ over $u \in \calB$ is $L^p(\Omega^{0,1}(u^*TM))$, i.e. the space of complex anti-linear (with respect to $(j_0,J_0)$) one-forms on $S_0$ with values in $u^*M$, of class $L^p$.
As in \S\ref{subsubsec:pseudoholomorphic polygons}, let $Y_0$ denote the one-form on $S$ with values in Hamiltonian vectors fields on $M$ which is associated to $K_0$.
Then the perturbed nonlinear Cauchy--Riemann operator $u \mapsto (du - Y_0)^{0,1}$ defines a smooth section $\sigma_{S_0}: \calB \rightarrow \calE_{S_0}$ whose 
vanishing locus is $\calM_{S_0}(\bx)$.
Let $$D_{u_0}\sigma_{S_0}: T_u\calB \rightarrow (\calE_{S_0})_{u_0}$$ denote the corresponding linearized operator at $u_0$, which is a real Cauchy--Riemann type operator, and in particular Fredholm.

Next, we extend the base of $\calE_{S_0}$ by allowing the conformal structure on $S_0$ to vary.
Let $U$ be a small neighborhood of $r_0$ in $\calR_{k+1}$.
We denote by $\calE_{U} \rightarrow U \times \calB$ the Banach vector bundle whose fiber $(\calE_U)_{(r,u)}$ over $(r,u) \in U \times \calB$ is $L^p(\Omega^{0,1}(u^*TM))$, i.e. the space of complex anti-linear (with respect to $(j,J_0)$, where $j$ is the almost complex structure on $S_0$ corresponding to $r$) one-forms on $S$ with values in $u^*TM$, of class $L^p$. 
Again, the perturbed nonlinear Cauchy--Riemann operator defines a smooth section $\sigma_U: U \times \calB \rightarrow \calE_U$ whose vanishing locus describes $\calM(\bx)$ locally near $(r_0,u_0)$.
In particular, $(r_0,u_0) \in \calM(\bx)$ is regular if and only if the linearization 
$$D_{(r_0,u_0)}\sigma_U: T_{r_0}U \times T_{u_0}\calB \rightarrow (\calE_U)_{(r_0,u_0)}$$ is surjective.

In order to modify $(K_0,J_0)$ so as to achieve regularity, the idea is to consider a ``universally-perturbed'' analogue of $\sigma_U$, denoted by $\sigma_{\calP \times U}$, where the base is now $\calP \times U \times \calB$, with $\calP$ the space of all allowable perturbation data near $(K_0,J_0)$. 
There is an analogous Banach vector, denoted by $\calE_{\calP \times U} \rightarrow \calP \times U \times \calB$, with fiber over $(K,J,r,u)$ given by 
$L^p(\Omega^{0,1}(u^*TM))$, where complex anti-linearity is now with respect to $(j,J)$.
The vanishing locus of $\sigma_{\calP \times U}$ describes the ``universally-perturbed moduli space'' locally near $(K_0,J_0,r_0,u_0)$; this fibers over $\calP$, with fiber over $(K,J)$ the moduli space $\calM(\bx)$ constructed with perturbation data $(K,J)$.
There is also a universally-perturbed linearized operator 
$$D_{(K_0,J_0,r_0,u_0)}\sigma_{\calP \times U}: T_{(K_0,J_0)}\calP \times T_{r_0}U \times T_{u_0}\calB \rightarrow (\calE_{\calP \times U})_{(K_0,J_0,r_0,u_0)},$$
where $T_{(K_0,J_0)}\calP$ is the space of all allowable infinitesimal variations in perturbation data near $(K_0,J_0)$.
To first approximation we can view $T_{(K_0,J_0)}\calP$ as simply the space of pairs $(\delta K,\delta J) \in \Omega^1(S_0,\calH) \times C^{\infty}(S_0,T_{J_0}\calJfix)$, but technically we must impose some additional constraints (e.g. $(\delta K,\delta J)$ should vanish on the thin parts of the domain Riemann surface) in order to not jeopardize the consistency conditions from \S\ref{subsubsec:Consistent universal perturbation data}; see \cite[9k]{seidelbook} more details.\footnote{Another technical issue that arises is that one must restrict to a suitable subspace (e.g. Floer's $C_\epsilon$ space - see e.g. \cite[\S4.4.1]{wendl2010lectures}) in order to have a Banach rather than Fr\'echet space.}

The main point is then to show that $D_{(K_0,J_0,r_0,u_0)}\sigma_{\calP \times U}$ is surjective, since then the universally-perturbed moduli space $\sigma_{\calP \times U}^{-1}(0)$ is a smooth Banach manifold, and an argument using the Sard--Smale theorem gives regularity for generic perturbation data in $\calP$. 
As in \cite[Eq. 9.2.6]{seidelbook}, the universally-perturbed linearized operator can be written explicitly as follows:
\begin{align}\label{linearization}
D_{(K_0,J_0,r_0,u_0)}\sigma_{\calP \times U}(\delta K, \delta J, \rho,V) = (\delta Y)^{0,1} + \tfrac{1}{2}\delta J \circ (du - Y) \circ j + D_{(r_0,u_0)}\sigma_{U} (\rho,V),
\end{align}
for $V \in W^{1,p}(\Gamma(u^*TM))$.
Here the first two terms measure infinitesimal variations in the perturbation data, while the last term measures infinitesimal variations in the conformal structure $r_0$ on $S_0$ and the map $u_0: S_0 \rightarrow M$.

To prove surjectivity, suppose by contradiction that the cokernel is nontrivial.
Since $D_{(K_0,J_0,r_0,u_0)}\sigma_{\calP \times U}$ is Fredholm and hence has closed image, this means we can find a nonzero element $\alpha \in L^q(\Omega^{0,1}(u^*TM))$, $\tfrac{1}{p} + \tfrac{1}{q} = 1$, with trivial $L^2$ pairing with every element in the image of $D_{(K_0,J_0,r_0,u_0)}\sigma_{\calP \times U}$. In particular, by looking at the last term in ~\eqref{linearization} and restricting to variations of the form $(0,0,0,V)$, this implies that $\alpha$ is a weak solution of the formal adjoint $D^*_{u_0}\sigma_{S_0}$. Since the latter is conjugate to a real Cauchy--Riemann type operator, elliptic regularity implies that $\alpha$ is smooth, and unique continuation implies that it cannot vanish in an open set in $S$.
On the other hand, using the first term in ~\eqref{linearization} and the fact that $\delta K$ takes arbitrary values in some neighborhood of a point $z \in S$,
we can arrange that $(\delta Y)^{0,1}$ looks like a delta function near $z$,
and then $\langle D_{(K_0,J_0,r_0,u_0)}\sigma_{\calP \times U}(\delta K,0,0,0),\alpha\rangle_{L^2} = 0$ forces $\alpha$ to vanish in a neighborhood of $z$, a contradiction.

\sss

We now consider the analogous setup for $\calM_{q;\mho}(\bx)$. 
As before we consider a curve $(r_0,u_0) \in \calM_q(\bx)$, where now we have $r_0 \in \calR_{k+1,q}$ and $u_0 \in \calM_{\calRuniv_r}(\bx)$. We let $S_0$ denote the domain of $u_0$, which is a Riemann disk with $k+1$ boundary punctures and $q$ interior marked points, and we denote its almost complex structure by $j_0$.
Let $U$ be a small neighborhood of $r_0$ in $\calR_{k+1,q}$. 
Similar to before, we have a universally-perturbed linearized operator
$$D_{(K_0,J_0,r_0,u_0)}\sigma_{\calP \times U}: T_{(K_0,J_0)}\calP \times T_{r_0}U \times T_{u_0}\calB \rightarrow (\calE_{\calP \times U})_{(K_0,J_0,r_0,u_0)},$$ given by
\begin{align}\label{linearization2}
D_{(K_0,J_0,r_0,u_0)}\sigma_{\calP \times U}(\delta K, \delta J, \rho,V) = (\delta Y)^{0,1} + \tfrac{1}{2}\delta J \circ (du - Y) \circ j + D_{(r_0,u_0)}\sigma_{U} (\rho,V),
\end{align}
for $V \in W^{1,p}(\Gamma(u^*TM))$.
Here as before $\calP$ is the space of allowable infinitesimal variations $(\delta K,\delta J)$ in the perturbation data, subject to the same conditions as before, and we now additionally require that $\delta K$ and $\delta J$ vanish near the interior marked points.
Note that $\rho$ now measures variations in the larger moduli space $\calR_{k+1,q}$, rather than $\calR_{k+1}$.

Surjectivity of $D_{(K_0,J_0,r_0,u_0)}\sigma_{\calP \times U}$ follows essentially by the previous argument. 
Indeed, although $(\delta K,\delta J)$ must vanish near the interior marked points of $S_0$, we can arrange that $\delta K$ is arbitrary in some neighborhood of a point $z \in S$ which lies away from the marked points.
It follows that the universally-perturbed moduli space is a smooth Banach manifold near $(K_0,J_0,r_0,u_0)$.
Furthermore, it is useful to observe that the same surjectivity argument based on elliptic regularity and unique continuation applies even if we restrict to variations of the form $(\delta K,0,0,V)$, where $V$ is constrained to vanish at the interior marked points.

Observe that evaluating at the interior marked points $z_1,\dots,z_q \in S_0$ gives a natural map
$$\ev_q: \sigma_{\calP \times U}^{-1}(0) \rightarrow M^{\times q}.$$ 
The main point is to show that this is submersion, since then a Sard--Smale argument proves that the preimage of $\mho^{\times q}$ is regular when specialized to a generic choice of perturbation data $(K,J)$. 
It therefore suffices to show that the linearized evaluation map 
\begin{align*}
d\ev_q: T_{(K_0,J_0)}\calP \times T_{r_0}U \times T_{u_0}\calB \rightarrow T_{u(z_1)}M \times \dots \times T_{u(z_q)}M
\end{align*}
given by
\begin{align*}
(\delta K,\delta J,\rho,V) \mapsto (V(z_1),\dots,V(z_q))
\end{align*}
is surjective when restricted to $\ker D_{(K_0,J_0,r_0,u_0)}\sigma_{\calP \times U}$.
To see this, pick arbitrary $(\xi_1,\dots,\xi_q) \in T_{u(z_1)}M \times \dots \times T_{u(z_q)}M$.
Let $\xi$ be any smooth section of $u^*TM$ such that $\xi(z_i) = \xi_i$ for $i = 1,\dots,q$.
As observed in the previous paragraph, we can find $(\delta K,0,0,V)$ with $V(z_1) = \dots = V(z_q) = 0$ such that
$$D_{(K_0,J_0,r_0,u_0)}\sigma_{\calP \times U} (K,0,0,V) = -D_{u_0}\sigma_{S_0}(\xi).$$
Then it follows that we have $D_{(K_0,J_0,r_0,u_0)}(K,0,0,V+\xi) = 0$ and $(V+\xi)(z_i) = \xi_i$ for $i = 1,\dots,k$. Since $\xi_1,\dots,\xi_q$ were arbitrary, this proves that $\ev_q$ is a submersion.
\end{proof}

\begin{remark}
Observe that $\calM_q$ is just a perturbed version of the $q$-point Gromov--Witten moduli space of $(M,\theta)$, which is not very interesting since $(M,\theta)$ is exact.
At any rate, the perturbation data $(K,J)$ is in general necessary to make the fiber products above transverse.
\end{remark}

\subsubsection{The compactification $\calMovl_{q;\mho}(\bx)$}

Let $\calTRsemi{k+1,q}$ be defined just like $\calTR{k+1,q}$ except with the stability condition replaced by the {\em semistability} condition that for every internal vertex $v$ we have:
\begin{itemize}
\item
if $v \in V_\ipl(T)$, then $|v|_\pl + 2|v|_\rd \geq 2$
\item
if $v \in V_\ird(T)$, then $|v| \geq 3$.
\end{itemize}
In other words we allow plain vertices with valency two.
Suppose we have Lagrangian labels $\bL = (L_0,...,L_k)$ and $\bx = (x_0,...,x_k) \in \gen(\bL)$.
A {\em stable broken pseudoholomorphic polygon with asymptotics $\bx$ and $q$ interior marked points} consists of:
\begin{itemize}
\item
$T \in \calTRsemi{k+1,q}$ labeled by $\bL$
\item
$x_e \in \genHLL{L_e}{L_e'}$ for each $e \in E_\pl(T)$, where $(L_e,L_e')$ denotes the Lagrangian labels on either side of $e$, and such that $x_e = x_i$ if $e$ is the $i$th plain external edge of $T$
\item $u_v \in \calM_{E_\rd(v)}(\bx_v)$ for each $v \in V_\ipl(T)$
\item $u_v \in \calM_{E(v)}$ for each $v \in V_\ird(T)$
\item for each $e \in E_\ird(T)$, say
 with endpoints corresponding to the marked points $p$ and $p'$ of $u_{\iv(e)}$ and $u_{\tv(e)}$ respectively, we have $u_{\iv(e)}(p) = u_{\tv(e)}(p')$.
\end{itemize}
Let $\calMovl_T(\bx)$ denote the space of stable broken pseudoholomorphic polygons 
with asymptotics $\bx$ and $q$ interior marked points which are modeled on $T \in \calTRsemi{k+1,q}$.
Evaluating at the marked points indexed by round leaves, we get a map
\begin{align*}
\ev_q: \calMovl_T(\bx) \rightarrow M^{\times q}.
\end{align*}
We set 
\begin{align*}
\calMovl_{T;\mho}(\bx) := \calMovl_T(\bx) \underset{\ev_q,i_{\mho}^{\times q}}\times\mho^{\times q}
\end{align*}
and
\begin{align*}
\calMovl_{q;\mho}(\bx) := \coprod_{T \in \calTRsemi{k+1,q}}  \calMovl_{T;\mho}(\bx),
\end{align*}
equipped with the Gromov topology.
Let $\calM_{q;\mho}(\bx)^{\nn}$ denote the one-dimensional part of $\calM_{q;\mho}(\bx)$
and let $\ovl{\calM}_{q;\mho}(\bx)^{\nn}$ denote its closure in $\ovl{\calM}_{q;\mho}(\bx)$. 
As a restricted analogue of Proposition \ref{prop:compact top mfd with corners} in this setting, we have:  
\begin{prop}\label{prop:compactification structure}
$\calM_{q;\mho}(\bx)^{\oo}$ is a finite set of points, and $\ovl{\calM}_{q;\mho}(\bx)^{\nn}$ is a compact one-dimensional topological manifold with boundary given by
\begin{align*}
\bdy \ovl{\calM}_{q;\mho}(\bx)^{\nn} = \coprod_{T} \calMovl_{T;\mho}(\bx)^{\oo},
\end{align*}
where the disjoint union is over all $T \in \calTR{k+1,q}$ having one plain internal edge and no round internal edges.
\end{prop}
\begin{proof}
The compactness statements follow directly from Gromov's compactness theorem.
As for the description of $\bdy \ovl{\calM}_{q;\mho}(\bx)^{\nn}$, this requires gluing, which describes a neighborhood of the moduli space $\ovl{\calM}_{q;\mho}(\bx)^{\nn}$ near a given boundary stratum in terms of gluing coordinates in $(-\eps,0)$ for each boundary node and in $\D_\eps^2$ for each interior node, for $\eps > 0$ sufficiently small.
The gluing analysis along boundary nodes as the same as for Proposition \ref{prop:compact top mfd with corners}, which plays a fundamental role in defining Lagrangian Floer homology \cite{floer1988morse}.
As for gluing along interior nodes, since our Hamiltonian perturbations are trivial near the interior marked points this is precisely the variety encountered in Gromov--Witten theory; see e.g. \cite[\S10]{mcduff2012j}.
\end{proof}
\begin{remark}\label{rmk:disj_from_cycle}
Consider a Lagrangian $L$ which is disjoint from $\mho$. By picking suitable Floer data and perturbation data, we can arrange that $\calM_{q;\mho}(\bx) = \nil$ for all $q \geq 1$ and $\bx = (x_0,\dots,x_k)$ with $x_0,\dots,x_k \in \genHLL{L}{L}$ for $k \geq 0$. In particular, the case $k = 0$ says that we have $\calM_{q;\mho}(x_0) = \nil$, i.e. there are nontrivial operations involving only one boundary puncture. From now on we will assume this is the case. 

Indeed, by the Weinstein neighborhood theorem we can find a Liouville subdomain $M' \subset M$ which contains $L$ and is disjoint from $\mho$, and is identified with the $\epsilon$ disk cotangent bundle of $L$ for some Riemannian metric on $L$ and $\epsilon > 0$ sufficiently small.
We can choose Floer data and perturbation data such that all Hamiltonians vanish near $\bdy M'$, and all almost complex structures are contact type near $\bdy M'$. Then, as in the proof of Lemma~\ref{lemma:Liouville subdomains}, the integrated maximum principle guarantees that all of the Floer solutions comprising $\calM_{q;\mho}(\bx)$ are entirely contained in $M'$.
\end{remark}

\subsubsection{The bulk deformed Fukaya category}\label{subsubsec:the bdfuk}

The bulk deformed Fukaya category $\fuk_\mho(M,\theta)$ is the $\calA_\infty$ category over $\KL$ with:
\begin{itemize}
\item
objects given by closed exact Lagrangian branes in $(M,\theta)$ which are disjoint from $\mho$
\item
for objects $L_0,L_1$, $\hom(L_0,L_1)$ is the free $\KL$-module generated by $\genHLL{L_0}{L_1}$
\item
for $k \geq 1$, objects $L_0,...,L_k$, and $x_i \in \genHLL{L_{i-1}}{L_i}$ for $i = 1,...,k$, we set
\begin{align*}
\mu_{\mho}^k(x_k,...,x_1) := \sum_{q=0}^\infty \hbar^q\mu_{q;\mho}^k(x_k,...,x_1),
\end{align*}
where
\begin{align*}
\mu_{q;\mho}^k(x_k,...,x_1) := \frac{1}{q!}\sum_{\substack{x_0 \in \genHLL{L_0}{L_k}\\ u \in \calM_{q;\mho}(x_0,...,x_k)^{\oo}}} \signu x_0.
\end{align*}
\end{itemize}
By Proposition \ref{prop:compactification structure}, the sum defining $\mu^k_{q;\mho}(x_k,...,x_1)$ is finite and hence well-defined.
In fact, the sum defining $\mu_\mho^k(x_k,...,x_1)$ is also finite by index considerations.
Namely, if $\calM_{q;\mho}(x_0,...,x_k)^{\oo}$ is nonempty, 
we must have 
\begin{align*}
|x_0| = |x_1| + ... + |x_k| + 2 - k + q(l-2).
\end{align*}
In particular, the index of $x_0$ is a strictly increasing function of $q$.
Since $\genHLL{L_0}{L_k}$ is finite, this means that $\mu_{q;\mho}^k(x_k,...,x_1)$ must vanish for $q$ sufficiently large.

Using Proposition \ref{prop:compactification structure}, a straightforward analysis of the boundary of $\ovl{\calM}_{q;\mho}(\bx)^{\nn}$ shows that $\fuk_\mho(M,\theta)$ satisfies the $\calA_\infty$ structure equations.
Indeed, notice that an element $T \in \calTR{k+1,q}$ having one plain internal edge and no round internal edges is specified by the following data:
\begin{itemize}
\item
an arbitrary subset of $\{1,...,q\}$, say with $a$ elements for $0 \leq a \leq q$
\item
a subset of $\{1,...,k\}$ of the form $\{b+1,...,b+c\}$ for $0 \leq b \leq k$ and $1 \leq c \leq k-b$.
\end{itemize}
Note that there are $\frac{q!}{a!(q-a)!}$ choices for the former subset. We then find
\begin{align*}
\bdy\ovl{\calM}_{q;\mho}(x_0,...,x_k)^{\nn} = \bigcup_{T,y} \calM_{q-a;\mho}(x_0,...,x_b,y,x_{b+c+1},...,x_k)^\oo \times \calM_{a;\mho}(y,x_{b+1},...,x_{b+c})^\oo,
\end{align*}
and this translates to an equation of the form
\begin{align}\label{eqn:bdAinf}
\sum_{a,b,c}\pm \tfrac{q!}{(q-a)!a!}  (q-a)! \mu_{q-a;\mho}^{k-c+1}(x_k,...,x_{b+c+1},a!\mu_{a;\mho}^{c}(x_{b+c},...,x_{b+1}),x_b,...,x_1) = 0
\end{align}
which is precisely the $\calA_\infty$ structure equation after dividing both sides by $q!$.
\begin{remark}[on signs]
Following the orientation conventions of \cite[\S 11]{seidelbook}, the sign in ~\eqref{eqn:bdAinf} is $(-1)^{|x_1| + \dots + |x_b| + n}$.
Indeed, from the theory of determinant lines for Fredholm operators, we have, for each $(r,u) \in  \calM_{q;\mho}(x_0,\dots,x_k)$, a canonical isomorphism (modulo scaling) \cite[\S12.8]{seidelbook}
\begin{align}\label{eq:sign_iso}
\lambda^{\op{top}}T_{(r,u)}\calM_{q}(\bx) \cong \lam^{\op{top}}T_r\calR_{k+1,q} \otimes \orien(x_o) \otimes \orien(x_1)^{\vee} \otimes \dots \otimes \orien(x_k)^{\vee},
\end{align}
where $\orien(x_0),\dots,\orien(x_k)$ are the orientation spaces as in \S\ref{subsubsec:brane structures} (and we are of course assuming regularity).
After choosing a trivialization of each orientation space and orientations of the moduli spaces $\calR_{k+1,q}$, this determines an orientation on $\calM_{q}(\bx)$, and hence on
$\calM_{q;\mho}(\bx) = \calM_q(\bx) \underset{\ev_q,i_{\mho}^{\times q}} \times \ \mho^{\times q}$ via the orientation on $\mho$.
In turn, this gives rise to a sign whenever $u \in \calM_{q;\mho}(\bx)^{\oo}$ has index zero.
This is (up to an unpleasant but necessary additional combinatorial factor \cite[(12.24)]{seidelbook}), the sign $\signu \in \{-1,1\}$.

Since the isomorphism ~\eqref{eq:sign_iso} behaves naturally with respect to gluing, the main nontrivial bookkeeping issue arises from the orientations on $\calR_{k+1,q}$ for varying $k,q$.
Indeed, recall that in the case of the ordinary Fukaya category, i.e. $q = 0$, there is a natural orientation on $\calR_{k+1}$ induced by the boundary orientation on the disk via the description $\conf_{k+1}(\bdy\D^2)/\pslr$. Under the identifications of codimension one boundary strata with products of open strata of lower dimensions, the boundary orientation differs from the product orientation by an additional sign \cite[\S12.22]{seidelbook}. 
Similarly, in the presence of interior marked points, there are natural orientations on the moduli spaces $\calM_{k+1,q}$, induced by the boundary orientation for points in $\bdy \D^2$ and the complex orientation for points in $\Int \D^2$. 
In the unconstrained case, the same sign discrepancy holds, essentially because the orientations coming from the interior marked points are independent of their ordering.  In the constrained case, the same is true if the codimension of $\mho$ is even, whereas if $\mho$ is odd transpositions flip the sign.
\end{remark}
\begin{remark}
Note that we are counting curves with $q$ ordered marked points and then dividing by $q!$, which heuristically is the same as counting curves with $q$ unordered marked points.
However, our approach allows more freedom in choosing perturbation data and avoids working with orbifolds.
\end{remark}
\begin{remark}
If $l = 2$, the index argument given above fails to establish convergence of the sum defining $\mu^k$. However, see \S\ref{subsubsec:twisting versus bulk deforming} below.
\end{remark}

The argument in \S\ref{subsubsec:Independence of Floer data, strip-like ends, and perturbation data} applies mutatis mutandis to show that $\fuk_{\mho}(M,\theta)$ is independent of the choice of Floer data, strip-like ends, and perturbation data up to $\calA_{\infty}$ quasi-isomorphism. Indeed, we just need to check that $(L,0)$ and $(L,1)$ are still quasi-isomorphic objects, and by Remark~\ref{rmk:disj_from_cycle} this reduces to the argument for the undeformed case.
Similarly, the bulk deformed analogue of Lemma~\ref{lemma:Liouville subdomains} immediately holds, provided that the cycle we take in $M'$ is the restriction of $\mho$.
It is also natural to expect that $\fuk_{\mho}(M,\theta)$ depends only on the homology class $\mho \in H_{2n-l}(M,\bdy M)$, but this is somewhat more subtle since our definition only allows Lagrangians in the complement of $\mho$. At any rate we do not directly need this for our intended applications; see however \S\ref{subsubsec:independence of mho} for the analogous discussion for bulk deformed symplectic cohomology.

 \subsection{Bulk deformed symplectic cohomology}
\label{subsec:bulk deformed sh}

As before, let $(M,\theta)$ be a Liouville domain and let $i_\mho: (\mho,\bdy \mho) \rightarrow (M,\bdy M)$ be a smooth codimension $l$ cycle with $l \geq 4$ even.
In this subsection we construct $\sh_\mho(M,\theta)$, the symplectic cohomology of $(M,\theta)$ bulk deformed by $\mho$. 
Our construction follows the same direct limit formalism described in \S\ref{subsubsec:symplectic cohomology formalism}.
In fact, essentially every part of the standard symplectic cohomology package has a close analogue in the context of bulk deformations, although
some additional care is need to setup up the relevant moduli spaces.
We define $\sh_\mho$ as a direct limit
\begin{align*}
\sh_\mho(M,\theta) := \lim_{\tau \rightarrow \infty}\hf_\mho(H)
\end{align*}
over generic $H \in C^{\infty}(S^1,\calHtau)$.
Here $\hf_\mho(H)$ is the bulk deformed analogue of $\hf(H)$ and we have a bulk deformed continuation map $\Phi_\mho: \hf_\mho(H_+) \rightarrow \hf_\mho(H_-)$ 
whenever the slope at infinity of $H_-$ is larger than that of $H_+$.
More specifically, $\hf_\mho(H)$ is the cohomology of a cochain complex $\cf_\mho(H)$ over $\KL$ where:
\begin{itemize}
\item
the underlying $\KL$-module $\KL\langle\calP_H\rangle$ is freely generated by $\calP_H$
\item
the differential is of the form 
\begin{align*}
\delta_\mho = \delta_{0;\mho} + \hbar \delta_{1;\mho} + \hbar^2 \delta_{2;\mho} + ...,
\end{align*}
where $\delta_{0;\mho}$ is just the usual Floer differential and each $\delta_{q;\mho}$ is a linear map $\KL_0\langle\calP_H\rangle \rightarrow \KL_0\langle\calP_H\rangle$ of degree $1 + q(l-2)$.
\end{itemize}
In particular, since $\calP_H$ is finite, index considerations show that $\delta_{q;\mho}$ vanishes for $q$ sufficiently large. 
Note that $\delta_\mho$ has degree one since $\hbar$ has degree $2-l$.
Similarly, the bulk deformed continuation maps are of the form 
\begin{align*}
\Phi_{\mho} = \Phi_{0;\mho} + \hbar \Phi_{1;\mho} + \hbar^2 \Phi_{2;\mho} + ...,
\end{align*}
where $\Phi_{0;\mho}$ is a usual continuation map and each $\Phi_{q;\mho}$ is a linear map $\KL_0\langle \calP_H\rangle \rightarrow \KL_0\langle \calP_H\rangle$ of degree $q(l-2)$.
In particular, $\Phi_{q;\mho}$ vanishes for $q$ sufficiently large and $\Phi_{\mho}$ has degree zero.

As in \S\ref{subsubsec:symplectic cohomology formalism},
the map $\delta_{0;\mho}$ implicitly depends on a choice of $J \in C^{\infty}(S^1,\calJ)$ and the map $\Phi_{0;\mho}$ implicitly depends on a choice of monotone homotopy from $(H_-,J_-)$ to $(H_+,J_+)$.
Naively, $\delta_{q;\mho}$ and $\Phi_{q;\mho}$ are defined by counting solutions of the Floer equation and continuation map equation respectively for curves with $q$ point constraints in $\mho$. 
More precisely, as for $\fuk_\mho$, we make domain-dependent perturbations of these equations,
and these perturbations should be suitably compatible with various gluing maps in order to achieve the desired structure equations.
We next describe this perturbation scheme in detail.

\subsubsection{The moduli space $\calcyl_{q+2}$ and its compactification}\label{subsubsec:the moduli space calcyl}
We begin by introducing the moduli spaces which are relevant to the maps $\delta_{q;\mho}$. For $q \geq 1$, let $\calcyl_{q+2}$ denote the moduli space of Riemann spheres
with $q+2$ ordered marked points, the first of which is equipped with an asymptotic marker, modulo biholomorphisms. We declare the first marked point, called the ``output", to be negative, 
the second marked point, called the ``input", to be positive, and the rest of the marked points also to be positive.
Note that $\calcyl_{q+2}$ is a smooth manifold of dimension $2q-1$ and that
the definition of $\calcyl_{q+2}$ is almost the same as $\calsphere_{q+2}$ apart from the asymptotic markers. However, we will prefer to view $\calcyl_{q+2}$ as a certain moduli space of cylinders.
Let $\calcyluniv_{q+2} \rightarrow \calcyl_{q+2}$ denote the universal family where the fiber $\calcyluniv_r$ over $r \in \calcyl_{q+2}$ is the corresponding Riemann cylinder given by puncturing the input and output marked points. 
Concretely, we can take
\begin{align*}
\calcyl_{q+2}= \conf_{q}(\R \times S^1)/\R,\;\;\;\;\; \calcyluniv_{q+2} = \conf_{q}(\R \times S^1) \times_{\R} (\R \times S^1),
\end{align*}
where $\R$ acts diagonally by translating each copy of $\R \times S^1$.
Here the output and input correspond to $s = -\infty$ and $s=+\infty$ respectively and the asymptotic marker corresponds to $1 \in S^1$.

Let $\calTcyl{q+2}$ denote the set of stable trees with $q+2$ ordered external edges.
This is of course the same as $\calTsphere{q+2}$, but we prefer to view an element of $\calTcyl{q+2}$ as a stable tree $T$ where:
\begin{itemize}
\item
the first external edge is called the ``output", the second external edge is called the ``input", and the remaining $q$ external vertices are equipped with an ordering
\item
the edges of $T$ are oriented away from the output
\item
the edges lying on the path between the output and input are ``plain" and the remaining edges are ``round"
\end{itemize}
As in \S\ref{subsubsec:The moduli space Rkq and its compactification}, let $E_\ipl(T)$ and $E_\ird(T)$ denote the internal edges of $T$ which are plain and round respectively, and define $V_\ipl(T)$ and $V_\ird(T)$ similarly.
Following the usual shorthand, let $\vipcyl$ denote the union of $V_\ipl(T)$ for all $T \in \calTcyl{q+2}$ and $q+2 \geq 3$, and define $\vircyl$ similarly.
For $v \in \vircyl$, define $\calsphere_{E(v)}$ as in \S\ref{subsubsec:the sphere moduli space} and equip $\calsphereuniv_{E(v)} \rightarrow \calsphere_{E(v)}$ with fiberwise cylindrical ends.
For $v \in \vipcyl$, we define $\calcyl_{E(v)}$ just like $\calcyl_{|v|}$ except that instead of ordering the marked points we:
\begin{itemize}
\item
index the output by the incoming plain edge at $v$
\item
index the input by the outgoing plain edge at $v$
\item 
index the remaining $|v|_\rd$ marked points by the set $E_\rd(v)$.
\end{itemize}
For each $v \in \vicyl$, we endow the universal family $\calcyluniv_{E(v)} \rightarrow \calcyl_{E(v)}$ with fiberwise cylindrical ends
\begin{align*}
\epsilon_0': \calcyl_{E(v)} \times \R_- \times S^1 \hookrightarrow \calcyluniv_{E(v)}\\
\epsilon_1': \calcyl_{E(v)} \times \R_+ \times S^1 \hookrightarrow \calcyluniv_{E(v)}
\end{align*}
and 
\begin{align*}
\epsilon_e: \calcyl_{E(v)} \times \R_+ \times S^1 \hookrightarrow \calcyluniv_{E(v)},\;\;\;\;\; e \in E_\rd(v),
\end{align*}
where $\epsilon_0'$ is asymptotic to the output puncture, $\epsilon_1'$ is asymptotic to the input puncture, and $\epsilon_e'$ is asymptotic to the marked point indexed by $e$.
As in \S\ref{subsubsec:the sphere moduli space}, the asymptotic marker at the output puncture naturally induces asymptotic markers at the input puncture and the remaining marked points and we require the cylindrical ends to align with these.

For $T \in \calTcyl{q+2}$ and $\e > 0$ small, set
\begin{align*}
\calcylovl_T &:= \left(\prod_{v \in V_\ipl(T)} \calcyl_{E(v)}\right) \times \left(\prod_{v \in V_\ird(T)} \calsphere_{E(v)} \right)\\
\calcylovl_T^\e &:= \calRovl_T  \times (-\e,0]^{E_\ipl(T)} \times \left( \D^2_\e\right)^{E_\ird(T)}.
\end{align*}
Using the universal cylindrical ends, for each $e \in V_\ii(T)$ we have a gluing map
\begin{align*}
\phi_{T,e}: \{r \in \calcylovl_T^\e\;:\; \rho_e \neq 0\} \rightarrow \calcylovl_{T/e}^\e.
\end{align*}
Here for $e \in E_\pl(T)$ we have $\rho_e \in (-\e,0)$ and we glue by aligning the asymptotic markers at either end, whereas for $e \in E_\rd(T)$ we have $\rho_e \in \D^2_\e \setminus \{0\}$ and the $S^1$ factor corresponds to the angle between the asymptotic markers.
We define the compactification of $\calcyl_{q+2}$ as a topological space by
\begin{align*}
\calcylovl_{q+2} := \left( \coprod_{T \in \calTcyl{q+2}} \calRovl_T^\e \right) / \sim,
\end{align*}
where $r \sim \phi_{T,e}(r)$ for any $r$ in the domain of $\phi_{T,e}$.

From now on assume that, for each $q \geq 1$, the cylindrical ends at the output and input and the disk-like neighborhoods at the remaining marked points of $\calcyluniv_{q+2} \rightarrow \calcyl_{q+2}$ are consistent in the usual sense.

\subsubsection{The moduli space $\calcont_{q+2}$ and its compactification}\label{subsubsec:the moduli space calcont}

Next, we introduce the moduli spaces relevant to the maps $\Phi_{q;\mho}$.
For $q \geq 0$, let $\calcont_{q+2}$ denote the moduli space of Riemann spheres with $q+2$ ordered marked points, the first of which is equipped with an asymptotic marker, and a {\em sprinkle}, modulo biholomorphisms.
Before explaining what we mean by a sprinkle, some preliminary comments are in order.
Let $S$ be a Riemann cylinder representing an element of $\calC_{q+2}$, equipped with the cylindrical ends $\epsilon_0',...,\epsilon_{q+1}'$ induced from the universal family. 
Observe that there is well-defined line $\R \cong L \subset S$ which aligns with the asymptotic markers at the output and input punctures.
Namely, we take the preimage of $\R \times \{1\}$ under any biholomorphism $\psi: S \cong \R \times S^1$ which sends the asymptotic marker at the output puncture to $1 \in S^1$ at $s = -\infty$.
Naively, a sprinkle is just a point $p \subset L$, the role being to break the translational symmetry of $S$. Indeed, in order to define continuation maps of degree zero we need some mechanism for increasing the dimension of $\calcyl_{q+2}$ by one.

Unfortunately, this definition of sprinkle does not quite play well with gluing, so we will need to slightly bend the definition. 
We adapt the following concept and terminology from \cite{abouzaid2010open}.
Define a {\em popsicle stick} for $S$ to be 
a line $\R \cong L \subset S$ such that:
\begin{itemize}
\item
$(\epsilon_0')^{-1}(L)$ agrees with $\{1\} \times \R_-$ near $s = -\infty$
\item
$(\epsilon_1')^{-1}(L)$ agrees with $\{1\} \times \R_+$ near $s = +\infty$
\item
$\psi(L)$ is of the form $\{ t = \beta(s)\}$ for some function $\beta: \R \rightarrow \Op(1) \subset S^1$.
\end{itemize}
Given a popsicle stick $L \subset S$, we define a sprinkle to be simply a point $p \in L$.
We pick, for each $v \in \vipcyl$, fiberwise popsicle sticks for the universal family $\calcyluniv_{E(v)} \rightarrow \calcyl_{E(v)}$, and we assume that these are consistent with respect to gluing.
We  use these universal popsicle sticks to make sense of sprinkles and hence the preceding definition of $\calcont_{q+2}$ for $q \geq 1$.
Regarding the case $q = 0$, we similarly define $\calcont_{2}$ by endowing the cylinder $\R \times S^1$ with the standard popsicle stick $\R \times \{1\}$.

Before we can define $\calcontovl_{q+2}$, we need to discuss popsicle sticks for the set-indexed versions of $\calcyl_{q+2}$, and this in turn requires slightly more care 
with cylindrical ends.
Let $\calTcont{q+2}$ be defined in the same way as $\calTcyl{q+2}$, except that one plain internal vertex  $v_\spr \in V_\ipl$ is designated as the ``sprinkle vertex", and $v_\spr$ is allowed to have valency two.
As the usual shorthand, let $\vipcont$ denote the union of $E_\ipl(T)$ over all $T \in \calTcont{q+2}$ and $q \geq 0$,
with $\vircont$ defined similarly.
For $v \in \vipcont$ with $|v| \geq 3$,
we define $\calcyl_{E(v)}$ as in \S\ref{subsubsec:the moduli space calcyl},
and we equip $\calcyluniv_{E(v)} \rightarrow \calcyl_{E(v)}$ with fiberwise cylindrical ends and popsicle sticks as follows:
\begin{itemize}
\item
If $v$ is a vertex of the tree $T \in \calTcont{q+2}$ and $T$ has a sprinkle vertex of valency at least three,
then we forget the sprinkle, viewing $T$ as an element of $\calTcyl{q+2}$,
and we take the induced cylindrical ends and popsicle sticks from $\calcyluniv_{E(v)} \rightarrow \calcyl_{E(v)}$
with $v$ viewed as an element of $\vicyl$.
\item
Otherwise, if $T$ has a sprinkle vertex of valency two (necessarily distinct from $v$),
then we contract the edge preceding the sprinkle, view the result as an element of $\calTcyl{q+2}$, 
and take the corresponding induced cylindrical ends and popsicle sticks.
\end{itemize}
Similarly, for $v \in \vircont$ we define $\calsphere_{E(v)}$ as before and we use the above prescription to induce cylindrical ends on each family $\calsphereuniv_{E(v)} \rightarrow \calsphere_{E(v)}$.
Now for a sprinkle vertex $v_\spr \in \vicont$ with $|v_\spr| \geq 3$, we denote by $\calcont_{E(v_\spr)}$ the set-indexed analogue of $\calcont_{|v_\spr|}$, defined using the above universal popsicle sticks.
We endow the family $\calcontuniv_{E(v_\spr)} \rightarrow \calcont_{E(v_\spr)}$ with the cylindrical ends induced from $\calcyluniv_{E(v_\spr)} \rightarrow \calcyl_{E(v_\spr)}$.
Finally, in the case that $v_\spr$ has valency two, we define $\calcont_{E(v_\spr)}$ via the standard popsicle stick $\R \times \{1\} \subset \R \times S^1$.

For $T \in \calTcont{q+2}$ and $\e > 0$ small, set
\begin{align*}
\calcontovl_T &:= \calcont_{E(v_\spr)} \times 
\left(\prod_{\substack{v \in V_\ipl(T) \\ v \neq v_\spr}} \calcyl_{E(v)}\right) \times \left(\prod_{v \in V_\ird(T)} \calsphere_{E(v)} \right)\\
\calcontovl_T^\e &:= \calRovl_T  \times (-\e,0]^{E_\ipl(T)} \times \left( \D^2_\e\right)^{E_\ird(T)}.
\end{align*}
By design, we can easily incorporate sprinkles into the gluing construction, at least for gluing parameters sufficiently close to $0$.
For each $e \in V_\ii(T)$ we therefore have a gluing map
\begin{align*}
 \phi_{T,e}: \Op(\calcontovl_T) \subset \{ r \in \calcontovl_T^\e\;:\; \rho_e \neq 0\} \rightarrow \calcontovl_{T/e}.
\end{align*}
Adapting the usual outline, we use these to define $\calcontovl_{q+2}$ as a topological space.

\begin{remark}
The conditions in the definition of a popsicle stick for $s$ near $\pm \infty$ mean that when we glue two Riemann surfaces with popsicle sticks, the two popsicle sticks overlap in the glued surface and hence combine to give a well-defined popsicle stick.
Naively, given a Riemann cylinder with marked points, we could try to take the popsicle stick for a cylinder with marked points to be simply the ``standard popsicle stick'' $\R \times \{1\}$, but a priori this would not have suitable compatibility with our previously chosen cylindrical ends, and these dictate how to glue. 

Alternatively, we could insist on using the lines $\R \times \{1\}$, and then retroactively choose our cylindrical ends more restrictively so as to make these legitimate popsicle sticks. This approach is reasonable since we only need to achieve compatibility with the cylindrical ends at the input and output punctures, but this does not generalize to settings with multiple input punctures, in which case different input punctures place competing requirements on the popsicle stick at the output puncture. Apart from swapping cylindrical ends with strip-like ends, this is precisely the setting in \cite[\S2d]{abouzaid2010open}.
\end{remark}

\subsubsection{Consistent universal perturbation data}\label{subsubsec:consistent universal perturbation data}

In order to ensure a maximum principle, we need to pick perturbation data for $\calcyl_{q+2}$ and $\calcont_{q+2}$ with slightly more care than in the case of $\fuk_\mho$.
Consider a nondegenerate $H_0 \in C^{\infty}(S^1,\calHtau)$ for $\tau > 0$,  along with an accompanying generic $J_0 \in C^{\infty}(S^1,\calJ)$.
 Suppose $S$ represents an element of $\calcyl_{q+2}$, and let $\epsilon_-',\epsilon_+'$ denote the induced cylindrical ends at the output and input punctures respectively.
A perturbation datum for $S$ consists of 
a pair $(K,J)$ with $K \in \Omega^1(S,\calH)$ and $J \in C^{\infty}(S,\calJ)$ such that
\begin{itemize}
\item
$(\epsilon_\pm')^*K \equiv H_0 \otimes dt$ and $(\epsilon_\pm')^*J \equiv J_0$ 
\item
on each disk-like neighborhood of $S$ we have $K \equiv 0$ and $J \equiv \Jo$
\item 
$K$ is of the form $H \otimes \gamma$, where $H \in C^{\infty}(S,\calH_\tau)$ and $\gamma$ is a closed one-form.
\end{itemize}
Given $(H_0,J_0)$, we pick fiberwise perturbation data for each of the families
\begin{itemize}
\item
$\calcyluniv_{E(v)} \rightarrow \calcyl_{E(v)}$ for each $v \in \vipcyl$ 
\item
$\calsphereuniv_{E(v)} \rightarrow \calsphere_{E(v)}$ for each $v \in \vircyl$.
\end{itemize}
We assume these satisfy the analogues of the consistency conditions described in \S\ref{subsubsec:bdfuk consistent universal perturbation data} and are also invariant under the identifications mentioned at the end of \S\ref{subsubsec:bdfuk consistent universal strip-like ends and disk-like neighborhoods}.

Similarly, consider nondegenerate $H_- \in \calH_{\tau_-}$ and $H_+ \in \calH_{\tau_+}$ for $\tau_- \geq \tau_+$, along with accompanying generic $J_{_-},J_{_+} \in C^{\infty}(S^1,\calJ)$.
Suppose $S$ represents an element of $\calcont_{q+2}$, and let $\epsilon_-',\epsilon_+'$ denote the induced cylindrical ends at the output and input punctures respectively.
A perturbation datum for $S$ consists of 
a pair $(K,J)$ with $K \in \Omega^1(S,\calH)$ and $J \in C^{\infty}(S,\calJ)$ such that
\begin{itemize}
\item
 $(\epsilon_\pm')^*K \equiv H_\pm dt$ and $(\epsilon_\pm')^*J \equiv J_{\pm}$
\item
on each disk-like neighborhood of $S$ we have $K \equiv 0$ and $J \equiv \Jo$
\item 
$K$ is of the form $H\otimes \gamma$, where $H \in C^\infty(S,\calH_{\tau})$ for a function $\tau: S \rightarrow \R_+$ and $\gamma$ is a one-form satisfying $d(\tau \gamma) \leq 0$.
\end{itemize}
We assume that we have $(\epsilon'_\pm)^*H \equiv H_\pm$ and $(\epsilon'_\pm)^*\gamma \equiv dt$,
and we can also take $\gamma$ to be of the form $\gamma = \eta/\tau$,
in which case the conditions on $\eta$ are
\begin{itemize}
\item
$(\epsilon_\pm')^*\eta \equiv \tau_\pm dt$
\item
$d\eta \leq 0$.
\end{itemize}
Such an $\eta$ exists since $\tau_- \geq \tau_+$.
Notice that $K$ satisfies the condition for a maximum principle described in \cite[Remark 1.6.14]{abouzaid2013symplectic}.
\begin{remark}
In the case that $\gamma \equiv dt$, the condition $d(\tau\gamma) \leq 0$ becomes the familiar inequality $\bdy_s \tau \leq 0$ for monotone continuation maps.
\end{remark}

Now suppose we have already chosen perturbation data relevant to both $(H_-,J_-)$ and $(H_+,J_+)$.
We follow the rule from \S\ref{subsubsec:the moduli space calcont}
to induce fiberwise perturbation data for each of the families
\begin{itemize}
\item
$\calsphereuniv_{E(v)} \rightarrow \calsphere_{E(v)}$ for each $v \in \vircyl$
\item
$\calcyluniv_{E(v)} \rightarrow \calcyl_{E(v)}$ for each non-sprinkle vertex $v \in \vipcyl$,
\end{itemize}
where the data relevant to $v$ is induced from our choices for either $(H_-,J_-)$ or $(H_+,J_+)$, depending on whether $v$ comes before or after the sprinkle vertex.
We also pick fiberwise perturbation data for the family 
$\calcontuniv_{E(v_\spr)} \rightarrow \calcont_{E(v_\spr)}$ for each sprinkle vertex $v_\spr \in \vipcont$.
Together these should satisfy the usual consistency conditions for perturbation data and be invariant under the usual identifications.

\subsubsection{The moduli spaces $\calM_{q;\mho}(\gamma_-,\gamma_+)$}

Now suppose we have $H_0$ and $J_0$ as in \S\ref{subsubsec:consistent universal perturbation data}, 
and assume that have made corresponding choices of perturbation data.
For $q \geq 1$, suppose that $S$ represents an element of $\calcyl_{q+2}$, and assume that $S$ is equipped with the cylindrical ends $\epsilon_-',\epsilon_+'$ and perturbation datum $(K,J)$ induced from our universal choices.
As usual, let $Y$ denote the $(d\theta)$-dual of $K$.
For $\gamma_-,\gamma_+ \in \calP_{H_0}$, we denote by $\calM_S(\gamma_-,\gamma_+)$ the space of maps $u: S \rightarrow \wh{M}$
which satisfy 
\begin{itemize}
\item
$(Du - Y)^{0,1} = 0$
\item
$\lim_{s \rightarrow -\infty} (u \circ \epsilon_0')(s,\cdot) = \gamma_-$
\item
$\lim_{s \rightarrow +\infty} (u \circ \epsilon_1')(s,\cdot) = \gamma_+$.
\end{itemize}
Set 
\begin{align*}
\calM_q(\gamma_-,\gamma_+) := \{(r,u)\;:\; r \in \calcyl_{q+2},\; u \in \calM_{\calcyluniv_r}(\gamma_-,\gamma_+)\}.
\end{align*}
We also set $\calM_0(\gamma_-,\gamma_+) := \calM(\gamma_-,\gamma_+)$, the moduli space of unparametrized Floer trajectories as defined in \S\ref{subsubsec:symplectic cohomology formalism}.

Similarly, suppose we have $H_\pm$ and $J_\pm$ as in \S\ref{subsubsec:consistent universal perturbation data}, together with the corresponding choices of perturbation data.
For $\gamma_- \in \calP_{H_-}$ and $\gamma_+ \in \calP_{H_+}$, we define $\calM_q(\gamma_-,\gamma_+)$ for $q \geq 0$ in the same way by replacing $\calcyl_{q+2}$ with $\calcont_{q+2}$.
In particular, the space $\calM_0(\gamma_-,\gamma_+)$ can be viewed as a slight generalization of the space of continuation map trajectories from $\gamma_-$ and $\gamma_+$.

In either of the two cases above, evaluating at the marked points gives a map 
\begin{align*}
\ev_q: \calM_{q}(\gamma_-,\gamma_+) \rightarrow M^{\times q}
\end{align*}
and we set
\begin{align*}
\calM_{q;\mho}(\gamma_-,\gamma_+) := \calM_q(\gamma_-,\gamma_+) \underset{\ev_q,i_{\mho}^{\times q}} \times \mho^{\times q}.
\end{align*}
The analogue of Proposition \ref{prop:bd moduli spaces are regular} is:
\begin{prop}
For generic perturbation data, the moduli spaces $\calM_{q;\mho}(\gamma_-,\gamma_+)$ are regular and hence smooth manifolds.
\end{prop}
\noindent We also define the set-indexed analogues of $\calM_q$ and $\calM_q(\gamma_-,\gamma_+)$ in a similar fashion.

\subsubsection{The compactification $\calMovl_{q;\mho}(\gamma_-,\gamma_+)$}

For $\gamma_-,\gamma_+ \in\calP_{H_0}$, a {\em stable broken pseudoholomorphic cylinder with asymptotics $(\gamma_-,\gamma_+)$ and $q$ interior marked points} consists of:
\begin{itemize}
\item
$T \in \calTcylsemi{q+2}$
\item
$\gamma_e \in \calP_{H_0}$ for each $e \in E_\pl(T)$, such that $\gamma_e = \gamma_-$ if $e$ is the output edge and $\gamma_e = \gamma_+$ is $e$ is the input edge
\item
$u_v \in \calM_{E_\rd(v)}(\gamma_{e_-(v)},\gamma_{e_+(v)})$ for each $v \in V_\ipl(T)$, where
$e_-(v)$ and $e_+(v)$ denote the edges directly preceding and following $v$ respectively
\item 
$u_v \in \calM_{E(v)}$ for each $v \in V_\ird(T)$
\item for each $e \in E_\ird(T)$, say
 with endpoints corresponding to the marked points $p$ and $p'$ of $u_{\iv(e)}$ and $u_{\tv(e)}$ respectively, we have $u_{\iv(e)}(p) = u_{\tv(e)}(p')$.
\end{itemize}

Let $\calMovl_T(\gamma_-,\gamma_+)$ denote the space of stable broken pseudoholomorphic cylinders with asymptotics $(\gamma_-,\gamma_+)$ and $q$ interior marked points which are modeled on $T \in \calTcylsemi{q+2}$. Evaluating at the marked points indexed by round leaves, we get a map
\begin{align*}
\ev_q: \calMovl_T(\gamma_-,\gamma_+) \rightarrow M^{\times q}.
\end{align*}
We set
\begin{align*}
\calMovl_{T;\mho}(\gamma_-.\gamma_+) := \calMovl_T(\gamma_-,\gamma_+) \underset{\ev_q,i_{\mho}^{\times q}}\times\mho^{\times q}
\end{align*}
and
\begin{align*}
\calMovl_{q;\mho}(\gamma_-,\gamma_+) := \coprod_{T \in \calTcylsemi{q+2}}  \calMovl_{T;\mho}(\gamma_-,\gamma_+),
\end{align*}
equipped with the Gromov topology.
In this context, the basic gluing result is:
\begin{prop}
$\calM_{q;\mho}(\gamma_-,\gamma_+)^{\oo}$ is a finite set of points, and $\ovl{\calM}_{q;\mho}(\gamma_-,\gamma_+)^{\nn}$ is a compact one-dimensional topological manifold with boundary given by
\begin{align*}
\bdy \ovl{\calM}_{q;\mho}(\gamma_-,\gamma_+)^{\nn} = \coprod_{T} \calMovl_{T;\mho}(\gamma_-,\gamma_+)^{\oo},
\end{align*}
where the disjoint union is over all $T \in \calTcyl{q+2}$ having one plain internal edge and no round internal edges.
\end{prop}

Similarly, we define $\calTcontsemi{q+2}$ in the same way as $\calTcont{q+2}$ except that we allow non-sprinkle plain vertices of valency two (note that a sprinkle vertex of valency two is already stable).
For $\gamma_- \in \calP_{H_-}$ and $\gamma_+ \in \calP_{H_+}$,
we define $\calMovl_{q;\mho}(\gamma_-,\gamma_+)$ following the same pattern as in the previous paragraph, and the basic gluing result is:
\begin{prop}
$\calM_{q;\mho}(\gamma_-,\gamma_+)^{\oo}$ is a finite set of points, and $\ovl{\calM}_{q;\mho}(\gamma_-,\gamma_+)^{\nn}$ is a compact one-dimensional topological manifold with boundary given by
\begin{align*}
\bdy \ovl{\calM}_{q;\mho}(\gamma_-,\gamma_+)^{\nn} = \coprod_{T} \calM_{T;\mho}(\gamma_-,\gamma_+)^{\oo},
\end{align*}
where the disjoint union is over all $T \in \calTcont{q+2}$ having one plain internal edge and no round internal edges.
\end{prop}

\subsubsection{The bulk deformed differential and continuation map}

We now define the promised map $\delta_{q;\mho}: \KL_0\langle \calP_H\rangle \rightarrow \KL_0\langle \calP_H\rangle$ on an orbit $\gamma_+ \in \calP_H$ by
\begin{align*}
\delta_{q;\mho}(\gamma_+) = \dfrac{1}{q!}\sum_{\substack{\gamma_- \in \calP_H\\ u \in \calM_q(\gamma_-,\gamma_+)^\oo}} s(u)\gamma_-.
\end{align*}
A simple analysis of $\bdy\ovl{\calM}_{q;\mho}(\gamma_-,\gamma_+)^{\nn}$ shows that the resulting $\delta_\mho$ satisfies $\delta_\mho^2 = 0$.
Indeed, algebraically this corresponds to special case of our argument in \S\ref{subsubsec:the bdfuk} showing that the differential on the $\bdfuk$ squares to zero.

Similarly, we define the map $\Phi_{q;\mho}: \KL_0\langle \calP_{H_+}\rangle \rightarrow \KL_0\langle \calP_{H_-}\rangle$ on an orbit $\gamma_+ \in \calP_{H_-}$ by
\begin{align*}
\Phi_{q;\mho}(\gamma_+) = \dfrac{1}{q!}\sum_{\substack{\gamma_- \in \calP_{H_-}\\ u \in \calM_q(\gamma_-,\gamma_+)^\oo}} s(u)\gamma_-.
\end{align*}
In this case a similar analysis of $\bdy\ovl{\calM}_{q;\mho}(\gamma_-,\gamma_+)^{\nn}$
shows that $\Phi_\mho \circ \delta_\mho = \delta_\mho \circ \Phi_\mho$, i.e. $\Phi_\mho$ is a chain map.
Intuitively, we view elements in the one-dimensional moduli space as cylinders with several marked points (unordered, thanks to the factorial) and a sprinkle. These degenerate into pairs of rigid cylinders, with the marked points distributed between them and the sparkle landing on one of them; the fact the total signed count of boundary points in the moduli space is zero translates into the chain map relation.

\subsubsection{Invariance properties}\label{subsubsec:invariance properties}

By adapting the same outline we have been following to other standard pieces of Floer theory, one can prove the following facts:
\begin{itemize}
\item
The monotone continuation map $\Phi_\Omega$ is independent of the various choices involved in its construction up to chain homotopy.
\item
The composition of two monotone continuation maps is again a monotone continuation map up to chain homotopy.
\end{itemize}
In the first case, given two different constructions of $\Phi_\Omega$, one picks data which now depends on an additional parameter $r \in [0,1]$ and interpolates between the two corresponding families of choices. After setting up the compactified moduli spaces appropriately, the count of solutions of the parametrized problem gives precisely a chain homotopy between the two constructions of $\Phi_\Omega$.
In the second case, one considers a moduli space similar to $\calcont$ except with {\em two} sprinkles. 
By requiring the sprinkles to be distinct and with the first one closer to the input puncture, we get a compactification which includes (a) cylinders where the two sprinkles coincide, which is just a copy of $\calcont$, and (b) broken cylinders with the two sprinkles separated into different components.
In this case counting solutions induces a chain homotopy between a continuation map and a composition of two continuation maps.

In particular, the monotone continuation maps form a directed system.
It then follows from the direct limit formalism that $\sh_\mho(M,\theta)$ is independent of all choices of cylindrical ends and perturbation data.
To see that it also does not depend on the choice of reference almost complex structure $\Jo$, we can similarly incorporate varying $\Jo$ into the direct limit formalism. Namely, we can use different reference almost complex structures for different Hamiltonians $H_-$ and $H_+$, and then generically interpolate between these two when constructing the continuation map $\cf_\mho(H_+) \rightarrow \cf_\mho(H_-)$.

\subsubsection{Independence of $\mho$}\label{subsubsec:independence of mho}

The type of argument discussed in \S\ref{subsubsec:invariance properties} can also be used to show that $\sh_\mho(M,\theta)$ is invariant under smooth homotopies of the cycle $\mho$.
In fact, one can also show that  $\sh_\mho(M,\theta)$ depends only on the homology class of $\mho$ via a change of coordinates which resembles the argument in \S\ref{subsubsec:independence of Omega} for the twisted case.
Namely, suppose that $B$ is a smooth oriented manifold with corners whose boundary is of the form $\bdy B = \mho \cup C$, where $C$ is a smooth manifold with boundary $\bdy C = \bdy \mho$ and $\mho$ and $C$ are otherwise disjoint.
Let $i_B: B \rightarrow M$ be a smooth map such that $(i_B)|_{\mho} = i_\mho$
and $i_B(C) \subset \bdy M$.
In this situation we argue that $\sh_{\mho}(M,\theta)$ reduces to $\sh(M,\theta)$, the undeformed version of symplectic cohomology (defined over $\KL$).

\begin{remark}[aside on pseudocycles]
Strictly speaking, our argument will only show that $\sh_\mho(M,\theta)$ is independent of the bordism class of $\mho$. In order to get to the level of ordinary homology,
one can make use of {\em pseudocycles} and bordisms thereof, as discussed in \cite[\S 6.5]{mcduff2012j}. 
Namely, by \cite{schwarz1999equivalences}, pseudocycles up to bordism are equivalent to integral homology classes.
Moreover, one can easily adapt \cite[Definition 6.5.1]{mcduff2012j} to define a smooth pseudocycle 
$i_\mho: (\mho,\bdy \mho) \rightarrow (M,\bdy M)$. Basically,  we allow $\mho$ to be noncompact,
 but each noncompact end has image of codimension at least two. 
For such an $\mho$ we can proceed exactly as in the case of a smooth cycle, with the additional ends effectively invisible to all of our curve counts because of their high codimension.
\end{remark}

Fix a nondegenerate Hamiltonian $H \in \calHtau$ for some $\tau > 0$, along with all the auxiliary data needed to define $\cf_\mho(H)$. We consider a certain moduli space of pseudoholomorphic cylinders with one point constraint in $B$ and $q-1$ point constraints in $\mho$, defined by
\begin{align*}
\calM_{\substack{q;B}}(\gamma_-,\gamma_+) := \calM_q(\gamma_-,\gamma_+) \underset{\ev_q, i_B \times i_\mho^{\times (q-1)}}\times (B \times \mho^{\times (q-1)})
\end{align*}
where the fiber product is with respect to the maps
\begin{align*}
\ev_q: \calM_q(\gamma_-,\gamma_+) \rightarrow M^{\times q}\;\;\;\;\;\text{and}\;\;\;\;\;
i_B \times i_\mho^{\times (q-1)}: B \times \mho^{\times (q-1)} \rightarrow M^{\times q}.
\end{align*}
Now for each $q \geq 1$ we define a linear map $F_q: \KL_0\langle \calP_H \rangle \rightarrow \KL_0\langle \calP_H \rangle$ as follows. For $\gamma_+ \in \calP_H$, set
\begin{align*}
F_q(\gamma_+) := \dfrac{1}{(q-1)!}\sum_{\substack{\gamma_- \in \calP_H\\ u \in \calM_{q;B}(\gamma_-,\gamma_+)^\oo}} s(u) \gamma_-.
\end{align*}
Following the usual pattern we can construct $\calMovl_{q;B}(\gamma_-\gamma_+)^\nn$ as a compact oriented topological one-manifold with boundary.
In essence the boundary consists of once-broken cylinders, with the $q$ point constraints distributed arbitrarily between the two components, as well as unbroken cylinders with $q$ point constraints in $\mho$. 
After accounting for the orientations and orderings of marked points, this can be summarized by the following structure equation:
\begin{align}\label{eqn:have}
\delta_{q;\mho} = \dfrac{1}{q} \sum_{1 \leq i \leq q} \left( -F_i \circ \delta_{q-i;\mho} + \delta_{q-i;\mho}\circ F_i\right).
\end{align}

Using the maps $F_q$, one can define a chain isomorphism $\cf_{\mho}(H) \rightarrow \cf(H)$ for each $H \in C^{\infty}(S^1,\calHtau)$.
This induces an isomorphism $\hf_{\mho}(H) \cong \hf(H)$, which in turn naturally commutes with continuation maps to induce an isomorphism $\sh_{\mho}(M,\theta) \cong \sh(M,\theta)$.
The map $\cf_{\mho}(H) \rightarrow \cf(H)$ can be made explicit, but the combinatorics are significantly more complicated than the analogous situation for twisted symplectic cohomology 
discussed in \S\ref{subsubsec:independence of Omega}. 
This is most naturally understood from the point of deformation theory via differential graded Lie algebras, as we now explain.

Recall that the differential on $\cf_\mho(H)$ is of the form $\bdy + \sum_{q=1}^{\infty}\hbar^q \delta_{q;\mho}$, where $\bdy := \delta_{0;\mho}$ denotes the undeformed differential on $\cf(H)$. 
Let $\End(\cf(H))$ denote the space of linear endomorphisms of $\cf(H)$. 
This is naturally a differential graded Lie algebra (DGLA), with differential 
$$\delta(A) = \bdy \circ A - (-1)^{|A|} A \circ \bdy$$
and Lie bracket
$$ [A,B] = A \circ B - (-1)^{|A||B|} B \circ A$$
for $A,B \in \End(\cf(H))$.
This DGLA controls deformations of the chain complex $\cf(H)$,
one consequence of which is that $m := \sum_{q=1}^{\infty}\hbar^q \delta_{q;\mho}$ is a Maurer--Cartan element in $\End(\cf(H))$, i.e. it satisfies
$$\delta(m) + \tfrac{1}{2}[m,m] = 0,$$
or equivalently
$$\bdy \circ m + m \circ \bdy + m \circ m = 0.$$

Now let $\KL[t,dt]$ denote the graded commutative algebra freely generated by formal variables $t,dt$ with $|t| = 0$ and $|dt| = 1$.
Following e.g. \cite{manetti2005deformation}, we equip the tensor product $\End(\cf(H)) \otimes_{\KL} \KL[t,dt]$ with the structure of a DGLA, with differential
$$ \delta(A(t) + B(t)dt) = \delta A(t) + (-1)^{|A(t)|} \dot{A}(t)dt + \delta B(t) dt$$
and 
Lie bracket 
$$ [A(t) + B(t)dt, P(t) + Q(t)dt] = [A(t),B(t)] + [A(t),Q(t)]dt + (-1)^{|P(t)|}[B(t),P(t)]dt $$
for $A(t) + B(t)dt, P(t) + Q(t)dt \in \End(\cf(H)) \otimes \KL[t,dt]$.
In particular, an element $A(t) + B(t)dt$ is Maurer--Cartan if and only 
\begin{enumerate}
\item
$A(t) \in \End(\cf(H))$ is Maurer--Cartan for each $t \in [0,1]$
\item we have
\begin{align}\label{eqn:want}
\dot{A}(t) = \delta B(t) + [A(t),B(t)]. 
\end{align}
\end{enumerate}

Put
$$M := \sum_{q=1}^\infty t^q\hbar^q \delta_{q;\mho} + \left( \sum_{q=1}^{\infty} t^{q-1}\hbar^q F_q\right )dt.$$
We claim that $M$ is a Maurer--Cartan element in $\End(\cf(H))$.
Indeed, ~\eqref{eqn:have} gives precisely the $t^{q-1}$ term of ~\eqref{eqn:want}.
The above shows that $M|_{t = 1} = m$ and $M|_{t =0} = 0$ are homotopic in $\End(\cf(H))$.
According to \cite[Thm. 5.5]{manetti2005deformation}, this implies that they are gauge equivalent, i.e. we have
$$  e^a \ast m := m + \frac{e^{[a,-]} - \op{Id}}{[a,-]}([a,m] - \delta m)  =  0  $$
for some $a \in \End(\cf(H))$.
This is equivalent to $e^a (\bdy + m) e^{-a} = \bdy$ (see \cite{manetti2005deformation}), and hence $e^a$ gives the desired chain isomorphism $\cf_{\mho}(H) \rightarrow \cf(H).$

\subsection{Some functoriality properties}\label{subsec:some functoriality properties}

We describe here some additional context which is helpful for interpreting the results in this paper.
We first discuss the transfer map for twisted and bulk deformed symplectic cohomology. Among other things, this is used to prove that these are invariant under symplectomorphisms which preserve the (co)homology class of $\Omega$ or $\mho$.
We then briefly discuss twisted and bulk deformed wrapped Floer cohomology and their module structures over symplectic cohomology. 
Since the computational techniques in this paper apply most directly to wrapped Floer cohomology, this will allow us 
conclude that symplectic cohomology is nontrivial whenever wrapped Floer cohomology is nontrivial.

\subsubsection{The transfer map}

Let $(M,\theta)$ be a Liouville domain and let $(W,\la) \subset (M,\theta)$ be a Liouville subdomain.
In particular this means that $(M,\theta)$ is itself a Liouville domain, and we assume for simplicity that $\la = \theta|_W$.
In this case there is a transfer map, first constructed by Viterbo in \cite{viterbo1999functors}, which is a unital $\field$-algebra map $\sh(M,\theta) \rightarrow \sh(W,\la)$.
One basic consequence is that $\sh(M,\theta)$ is an exact symplectomorphism invariant of $(\wh{M},d\wh{\theta})$.
Namely, using the transfer map we can alternatively define $\sh(M,\theta)$ as the inverse limit of the symplectic cohomologies over all Liouville subdomains in $(\wh{M},\wh{\theta})$.
This alternative definition of symplectic cohomology is straightforwardly equivalent to the standard definition, and moreover it manifestly only depends on the exact symplectomorphism type of $(\wh{M},\wh{\theta})$. 
In fact, by \cite[Lemma 11.2]{cieliebak2012stein}, if two finite type Liouville manifolds are symplectomorphic then they are actually exact symplectomorphic, so this means that $\sh(M,\theta)$ is invariant under general symplectomorphisms of $(\wh{M},d\wh{\theta})$.
 
 To construct the transfer map, the basic idea is to utilize the action filtration on the symplectic cochain complex. One
 proceeds by considering a special class of ``step-shaped" Hamiltonians on $\wh{M}$ which are approximately zero in $\Int(W)$, linear with some slope $\tau_a > 0$ near $\bdy W$, approximately constant in $\Int(M) \setminus W$, and linear with some slope $\tau_b > 0$ on $\wh{M} \setminus \Int(M)$.
 Such a Hamiltonian belongs to the class $\calH_{\tau_b}$, and therefore we can compute $\sh(M,\theta)$ via a sequence of such Hamiltonians with $\tau_b \rightarrow \infty$.
 Moreover, the actions of orbits of $H$ are essentially controlled by the slopes $\tau_a$ and $\tau_b$, 
 and with some care we find a sequence of such Hamiltonians such that $\tau_a,\tau_b \rightarrow \infty$
 and any orbit in $\Int(W)$ has negative action while any orbit in $\wh{M} \setminus \Int(W)$ has positive action.
 This means that the orbits in $\wh{M} \setminus \Int(W)$ form a subcomplex, and hence the orbits in $\wh{W}$ form a quotient complex.
 In fact, by further picking almost complex structures on $\wh{M}$ which are contact type near $\bdy W$,
 the integrated maximum implies that Floer trajectories between orbits in $W$ are entirely contained in $W$.
 That is, the quotient complex is indistinguishable from the Floer complex of a certain Hamiltonian in $\calH_{\tau_a}$.
 Finally, after passing to cohomology, arguing similarly for monotone continuation maps and taking a direct limit, the quotient map induces the transfer map $\sh(M,\theta) \rightarrow \sh(W,\la)$.
 
Now suppose the $\Omega$ is a closed two-form on $M$. Ritter shows in \cite{ritter2013topological} that there is also a twisted transfer map of the form $\sh_\Omega(M,\theta) \rightarrow \sh_{\Omega|_W}(W,\la)$. 
Since the relevant Hamiltonian orbits, pseudoholomorphic cylinders, and action values are the same as in the untwisted case, one just needs to think a little bit about the role of $\Omega$ and $\Omega|_W$ to see that the same proof outlined above still holds. 

In fact, if $i_\mho: (\mho,\bdy \mho) \rightarrow (M,\bdy M)$ is a smooth cycle of codimension $l > 2$, say transverse to $\bdy W$, we can apply essentially the same proof to construct a transfer map
$\sh_\mho(M,\theta)\rightarrow \sh_{\mho|_W}(W,\la)$.
Indeed, the bulk deformed symplectic cochain complexes of $M$ and $W$ are generated by the same types of Hamiltonian orbits, and we can take all Hamiltonian terms in the constructions of $\delta_\mho$ and $\Phi_\mho$ to be of the same step-shaped form.
In order to arrange that the action filtration also behaves as expected, we need to pick perturbation data with slightly more care.
Recall that the action of a loop $\gamma: S^1 \rightarrow \wh{M}$ with respect to the time-dependent Hamiltonian $H$ is given with our conventions by 
\begin{align*}
\calA_H(\gamma) := -\int_{S^1} \gamma^*\theta + \int_0^1 H(\gamma(t))dt.
\end{align*}
If $u: S \rightarrow \wh{M}$ is a pseudoholomorphic marked cylinder as in the construction of $\delta_\mho$ or $\Phi_\mho$ with asymptotic orbits $\gamma_\pm$, we need $\calA_{H_-}(\gamma_-) \geq \calA_{H_+}(\gamma_+)$.
The conditions described in \S\ref{subsubsec:consistent universal perturbation data} suffice for a maximum principle, but they do not guarantee that the differential and continuation maps increase action since they only apply on the cylindrical end $\wh{M} \setminus M$.
The stronger condition needed for an action filtration is 
\begin{align*}
dK(\cdot,p) \leq 0 \;\;\;\;\;\text{for all}\;\;\;\;\; p \in \wh{M},
\end{align*}
where $K(\cdot,p)$ denotes the one-form on $\wh{M}$ obtained by evaluating Hamiltonian functions at $p$.
Note that the condition $d(\tau\gamma) \leq 0$ from \S\ref{subsubsec:consistent universal perturbation data} 
is essentially equivalent to the above condition holding on the cylindrical end.
For example, when $K = Hdt$ for $H$ a family of Hamiltonians depending only on the $s$ coordinate,
this condition becomes $\bdy_sH \leq 0$.

To see that such a condition can be satisfied, suppose that we have functions $H_\pm: \wh{M} \rightarrow \R$ with $H_- \geq H_+$. 
It will suffice to arrange that our perturbation term $K \in \Omega^1(S,\calH)$ satisfies:
\begin{itemize}
\item
 $(\epsilon_\pm')^*K \equiv H_\pm dt$
\item
on each disk-like neighborhood of $S$ we have $K \equiv 0$
\item 
$dK(-,p) \leq 0$ for all $p \in \wh{M}$.
\end{itemize} 
Let $\wt{\epsilon}'_-$ be an extension of the negative cylindrical $\epsilon'_-$ end to $(-\infty,\delta) \times [0,1]$ for some small $\delta > 0$.
Let $\gamma$ be a closed one-form on $S$ satisfying $(\epsilon_+')^*\gamma = dt$ and $(\wt{\epsilon}_-')^*\gamma = dt$.
Then it suffices to take $K$ of the form $H \otimes \gamma$, such that 
$H \in C^{\infty}(S,\calH)$ is equal to $H_+$ outside of the image of $\wt{\epsilon}_-'$,
while near the negative cylindrical end we have $\bdy_s (H_- \circ \wt{\epsilon}'_-) \leq 0$.
Note that the space of all such $K$ satisfying the above conditions is convex and hence contractible, so as before we can also find allowable perturbation data in families.

\subsubsection{Wrapped Floer cohomology as a module over symplectic cohomology}\label{subsubsec:wrapped Floer as a module}

Let $L$ be an exact Lagrangian in a Liouville domain $(M,\theta)$ with Legendrian boundary $\bdy L$ in $(\bdy M, \theta|_{\bdy M})$.
We further assume that $\theta|_L$ vanishes near $\bdy L$ (this can always be a achieved by a suitable Hamiltonian isotopy - see \cite[Lemma 4.1]{abouzaid2010open} and also \cite{ritter2013topological}),
and that the Reeb chords of $\theta|_{\bdy M}$ with endpoints on $L$ are nondegenerate.
In this case, near $\bdy M$ we have that $L$ is a cylinder over $\bdy L$ with respect to the Liouville flow, and therefore we can naturally complete $L$ to a noncompact Lagrangian $\wh{L} \subset \wh{M}$.
The self wrapped Floer cohomology of $L$, denoted by $\hw(L,L)$, is the open string analogue of $\sh(M,\theta)$.
It is defined in close analogy with symplectic cohomology by taking a direct limit over linear Hamiltonians $H \in \calHtau$ of the Floer cohomology $\field$-modules $\hf(L,L;H)$.

Slightly more generally, if $\Omega$ is a closed two-form on $M$ with support disjoint from $\Op(L)$, we can define the twisted wrapped Floer cohomology $\hw_\Omega(L,L)$ by using $\K$ coefficients and weighting counts of Floer strips $u$ by $t^{\int u^*\Omega}$.
Similarly, if $i_\mho: (\mho,\bdy \mho) \rightarrow (M,\bdy M)$ is a smooth cycle of codimension $l > 2$, 
we define $\hw_\mho(L,L)$ as an $\KL$-module by following a similar outline to the one we used for $\sh_\mho(M,\theta)$ in \S\ref{subsec:bulk deformed sh}.
In this case the domains of the relevant curves are Riemann disks with one output boundary puncture, one input boundary puncture, and some number $q \geq 0$ of interior marked points, and the corresponding moduli space is modeled on $\calR_{2,q}$.

As explained for example in \cite{ritter2013topological}, $\hw(L,L)$ admits the structure of a unital $\field$-module over $\sh(M,\theta)$.
Namely, by counting pseudoholomorphic maps of the form $\D^2 \setminus \{-1,1,0\} \rightarrow \wh{M}$
which are asymptotic to a Hamiltonian orbit at $0$ and to Hamiltonian chords with endpoints on $L$ at $-1$ and $1$,
after passing to cohomology and taking a direct limit
we get a map of the form 
\begin{align*}
\sh(M,\theta) \otimes \hw(L,L) \rightarrow \hw(L,L).
\end{align*}
Note that we are using a fixed conformal structure on the domain disk and therefore this is a map a degree zero (compare this to the closed-open map as in \cite{ganatra2013symplectic}).
 One can also check that it behaves as expected with respect to produce structures and units.

Similarly, in the twisted or bulk deformed settings we can construct maps unital algebra maps of the form
\begin{align*}
\sh_\Omega(M,\theta) \otimes \hw_\Omega(L,L) \rightarrow \hw_\Omega(L,L)
\end{align*}
and
\begin{align*}
\sh_\mho(M,\theta) \otimes \hw_\mho(L,L) \rightarrow \hw_\mho(L,L).
\end{align*}
In any of these contexts, a standard consequence of unitality is that vanishing symplectic cohomology implies vanishes wrapped Floer cohomology. Contrapositively, the existence of a Lagrangian as above with nontrivial self wrapped Floer cohomology implies that the ambient Liouville domain has nontrivial symplectic cohomology.

\subsection{Some broad view remarks and conjectures}

\subsubsection{Derived local systems}

As an extension of Remark \ref{rem:local system}, bulk deformed symplectic cohomology can viewed as symplectic cohomology with coefficients in a certain {\em derived local system}.
Just as local systems on a topological space $X$ correspond to modules over the fundamental group, derived local systems correspond to chain complexes over $C_*(\Omega X)$, where $\Omega X$ denotes the based loop space of $X$.
Roughly, a derived local system associates to every point of $X$ a chain complex over $C_*(\Omega X)$ and to every $p$-dimensional family of paths from $p$ to $p'$ a degree $p$  map between the associated chain complexes. In the case of bulk deforming by a cycle $\mho$ in $X$, the associated derived local system on $\Omega X$ is given by viewing a family of paths in $\Omega X$ as a family of cylinders in $X$ and then considering the intersection number with $\mho$.

\subsubsection{Twisting versus bulk deforming}\label{subsubsec:twisting versus bulk deforming}

Suppose that we try to apply the above the construction of $\sh_\mho(M,\theta)$ in the case where $\mho$ is of codimension $l = 2$. In this case we do not have a priori convergence of the sums defining $\delta_\mho$. In fact, suppose that $u: \R \times S^1 \rightarrow \wh{M}$ is an isolated pseudoholomorphic curve which is transverse to $\mho$.
If we imagine taking perturbation data which is independent of the location of the marked points, we see that each intersection point of $u$ 
 with $\mho$ contributes an infinite sum to $\delta_\mho$ of the form 
\begin{align*}
1 + \hbar + \hbar^2/2! + ... = e^{\hbar}.
\end{align*}
Therefore heuristically we have
\begin{align*}
\delta_\mho(\gamma_+) = \sum_{\substack{\gamma_-\\ u \in \calM(\gamma_-,\gamma_+)^\oo}}  e^{\hbar [u]\cdot [\mho]}\gamma_-.
\end{align*}
In particular, if there is an element $t \in \KL$ such that $e^t = \hbar$,
then this becomes
\begin{align*}
\delta_\mho(\gamma_+) = \sum_{\substack{\gamma_-\\ u \in \calM(\gamma_-,\gamma_+)^\oo}}  t^{ [u]\cdot [\mho]}\gamma_-.
\end{align*}
Modulo replacing $\mho$ with a Poincar\'e dual two-form $\Omega$, this looks just like the twisted differential for $\sh_{\Omega}(M,\theta)$.

\subsubsection{The $\calL_\infty$ structure on symplectic cohomology}

By work of Fabert \cite{fabert2013} and ongoing work of Borman--Sheridan,
there is an $\calL_\infty$ structure underlying symplectic cohomology.
In particular, there is a standard procedure to deform the differential of this $\calL_\infty$ algebra via any element which satisfies the Maurer--Cartan equation. 
It seems plausible that any cohomology class $c$ in $(M,\theta)$ can be represented by such a Maurer--Cartan element (say as a linear combination of Morse critical points), and that the resulting deformed symplectic cohomology is isomorphic to our construction of the symplectic cohomology bulk deformed by $c$.

Furthermore, recall that there is a closed-open map \cite{abouzaid2010geometric,ganatra2013symplectic} from the symplectic cohomology of $(M,\theta)$ to the Hochschild cohomology of the wrapped Fukaya category of $(M,\theta)$.
It also seems plausible that this map could be upgraded to chain level homomorphism of $\calL_\infty$ algebras, and hence that the Maurer--Cartan element corresponding to $c$ pushes forward to a Maurer--Cartan in the Hochschild cochains of the  wrapped Fukaya category.
In this case it seems natural to ask whether the Fukaya category of $(M,\theta)$ bulk deformed by (the Poincar\'e dual to) $c$ can be viewed as the deformed $\calA_\infty$ category with respect to this Maurer--Cartan element.

\section{Lefschetz fibrations}\label{sec:Lefschetz fibrations}

The main goal of this section is to understand pseudoholomorphic sections of certain Lefschetz fibrations with boundary conditions specified by matching cycles. 
The material is mostly minor variations of facts from the literature, except for \S\ref{subsec:the model computation},
where we make a seemingly new observation about sections of the model Lefschetz fibration.
We will subsequently rehash these results into statements about Fukaya categories in \S\ref{sec:squared dehn twists}.

\subsection{The basics}\label{subsec:the basics}

We begin with the basic notions of Lefschetz fibrations insofar as they will be used in this paper. Roughly, a Lefschetz fibration is a map $E \rightarrow \D^2$ whose singularities look like those of a complex Morse function. 
Here $\D^2$ denotes the closed unit disk, the total space $E$ will generally be a compact manifold with corners, and we will call the Lefschetz fibration ``exact" if every regular fiber is endowed with the structure of a Liouville domain. 
We next give a formal definition, with the caveat that the precise nuances will not play an essential role. To first establish some notation, let:
\begin{itemize}
\item
$\pi_\std: \C^n \rightarrow \C$ denote the model Lefschetz map, given by $\pi_\std(z_1,...,z_n) = z_1^2 + ... + z_n^2$
\item
$\Theta_\std := \frac{i}{4} \sum_{i=1}^n \left( z_i d\ovl{z_i} - \ovl{z_i}dz_i\right)$ denote the standard K\"ahler potential on $\C^n$.
\end{itemize}
\begin{definition}\label{def:lefschetz fibration}
An {\em exact Lefschetz fibration over $\D^2$} is a triple $(E^{2n},\Theta,\pi)$, where:
\begin{enumerate}
\item
$E^{2n}$ is a compact manifold with corners, $\Theta$ is a one-form on $E^{2n}$, and $\pi: E^{2n} \rightarrow \D^2$ is a smooth map
\item 
{\em Compatibility with $\Theta$}: $d\Theta$ is nondegenerate on the vertical tangent space $\ker(D_p\pi)$ for all nonsingular points $p \in E$ of $\pi$. Also, for each $z \in \D^2$, $\Theta_z := \Theta|_{E_z}$ restricts to a positive contact form along the boundary of the fiber $E_z := \pi^{-1}(z)$.
\item
{\em Lefschetz type singularities}: $\pi$ has finitely many critical points $p_1,...,p_k$ which lie in the interior of $E$ and map to pairwise distinct critical values. Each critical point $p_i$ and its critical value $q_i := \pi(p_i)$ 
have neighborhoods which are oriented diffeomorphic to neighborhoods of the origin in $\C^n$ and $\C$ respectively, such that $\pi$ is identified with $\pi_\std$ and $\Theta$ is identified with $\Theta_\std$.
\item
{\em Triviality of the horizontal boundary}: For any $p \in \bdy_hE$, the horizontal tangent space, consisting of all vectors in $T_pE$ which are $d\Theta$-orthogonal to $\ker(D_p\pi)$, is tangent to $\bdy_hE$.
Here the {\em vertical boundary} of $E$ is $\bdy_vE := \pi^{-1}(\bdy \PP)$ and the {\em horizontal boundary} is $\bdy_hE := \bdy E \setminus \Int(\bdy_v E)$.
\end{enumerate}
\end{definition}

In the case that $(E,\Theta)$ is in fact a Liouville domain with corners, meaning that $d\Theta$ is nondegenerate and the Liouville vector field $Z_\Theta$ is outwardly transverse along each boundary face of $E$, we call $(E,\Theta,\pi)$ a {\em Liouville Lefschetz fibration}. 
For any exact Lefschetz fibration $(E,\Theta,\pi)$ over $\D^2$, one can check that $d\Theta + Cd\theta_\std$ is nondegenerate for all $C > 0$ sufficiently large, where $\theta_\std$ denotes the standard Liouville one-form on $\D^2$. In particular, this makes $(E,\Theta + C\theta_\std)$ into a Liouville domain with corners.
After smoothing the corners we get a Liouville domain which is determined up to Liouville deformation equivalence by $(E,\Theta,\pi)$.

\sss

Let $(E,\Theta,\pi)$ be an exact Lefschetz fibration over $\D^2$ with critical values $q_1,...,q_k \in \Int(\D^2)$ as in Definition \ref{def:lefschetz fibration}.
Away from the critical points of $\pi$, there is a symplectic connection which associates to each $p \in E$ the horizontal tangent space in $T_pE$.
Using this connection, we can associate to any immersed path $\gamma: [0,1] \rightarrow \D^2$ a parallel transport map $\Phi_{\gamma}: E_{\gamma(0)} \rightarrow E_{\gamma(1)}$ which is well-defined away from the critical points of $\pi$.
In particular, if $\gamma$ is disjoint from $\{q_1,...,q_k\}$, $\Phi_{\gamma}$ is an exact symplectomorphism $(E_{\gamma(0)},\Theta_{\gamma(0)}) \cong (E_{\gamma(1)},\Theta_{\gamma(1)})$.

Now assume there is a distinguished base point $* \in \bdy \D^2$.
A path $\eta: [0,1] \rightarrow \D^2$ is called a {\em vanishing path} if $\eta(0) = *$, $\eta(1)$ is a critical value of $\pi$, and $\eta|_{(0,1)}$ is an embedding into $\Int(\D^2) \setminus \{q_1,...,q_k\}$. To any vanishing path we can associate:
\begin{itemize}
\item
the {\em thimble} $T_\eta$, which is the Lagrangian disk in $E$ consisting of all points which lie over $\eta$ and get parallel transported along $\eta$ to the unique critical point of $\pi$ in $E_{\eta(1)}$\footnote{Strictly speaking, this should be interpreted using parallel transport away from $E_{\eta(1)}$ and then taking the limit as we approach $E_{\eta(1)}$.}
\item
the {\em vanishing cycle} $V_\eta$, which is the exact Lagrangian sphere $T_\eta \cap E_*$ in $(E_*,\Theta_*)$.
\end{itemize}
Each vanishing cycle is naturally equipped with a {\em framing}, meaning a parametrization $S^n \overset{\cong}\rightarrow V_\eta$, defined up to precomposing with an orthogonal diffeomorphism of $S^n$ (see \cite[\S 16a]{seidelbook}).
We define a {\em basis of vanishing paths} to be a collection of vanishing paths $\eta_1,...,\eta_k$ which intersect pairwise only at $*$, and such that $\eta_i(1) = q_i$ for $i = 1,...,k$. Here $\eta_1,...,\eta_k$ are ordered clockwise according to the local orientation at $*$.
We get a corresponding collection of vanishing cycles $V_{\eta_1},...,V_{\eta_k}$ in $(E_*,\Theta_*)$, which we call a {\em basis of vanishing cycles}. 
Conversely, given any list of framed exact Lagrangians spheres $V_1,..,V_k$ in a Liouville domain $(M,\theta)$, we can construct a Liouville Lefschetz fibration over $\D^2$ such that a regular fiber is identified with $(M,\theta)$ and a basis of vanishing cycles is identified with $V_{\eta_1},...,V_{\eta_k}$.

\subsection{Matching cycles}
 The {\em matching cycle construction} provides a rather combinatorial approach to producing Lagrangian spheres in the total space of a Liouville Lefschetz fibration $(E,\Theta,\pi)$.
Namely, let $\gamma$ be an embedded path in the base which intersects the critical values of $\pi$ at precisely $\gamma(0)$ and $\gamma(1)$.
As in the vanishing cycle construction, we get Lagrangian disks $T_0$ and $T_1$ by considering those points in $E$ which lie above $\gamma_{[0,1/2]}$ and $\gamma_{[1/2,1]}$ respectively and are parallel transported to the corresponding critical point. The boundaries $V_0 = \bdy T_0$ and $V_1 = \bdy T_1$ are framed exact Lagrangian spheres in $(E_{\gamma(1/2)},\Theta_{\gamma(1/2)})$.
If $V_0$ and $V_1$ coincide (including their framings), we call $\gamma$ a {\em naive matching path}. In this case the union $L_\gamma := T_0 \cup T_1$ is a (framed) exact Lagrangian sphere in $(E,\Theta)$, which we call the {\em naive matching cycle} associated to $\gamma$. 

More generally, suppose that $V_0$ and $V_1$ are isotopic through framed exact Lagrangians in $(E_{\gamma(1/2)},d\theta|_{\gamma(1/2)})$. Let $I$ denote such an isotopy, called a {\em matching isotopy}. In this case $\gamma$ is called a {\em matching path}, and we can still construct a {\em matching cycle} after suitably deforming $(E,\Theta,\pi)$.
More precisely, as explained in \cite[\S 16g]{seidelbook}, we can find a deformation $(E,\Theta_t,\pi_t)$, constant outside of $\Op(E_{\gamma(1/2)})$, after which $\gamma$ becomes a naive matching path.
In fact, fixing the homotopy class of the matching isotopy $I$ rel endpoints, the resulting $L_\gamma$ is well-defined up to a further deformation and simultaneous isotopy of $L_\gamma$ through framed exact Lagrangians. 
Alternatively, we can pull back $L_\gamma$ to the original $(E,\Theta)$ after applying a Moser-type isotopy, and the result is well-defined up to isotopy through framed exact Lagrangians.

A particularly nice situation is when a Liouville Lefschetz fibration $(E,\Theta,\pi)$ has a fiber which itself admits an auxiliary Liouville Lefschetz fibration, and the vanishing cycles of the former are matching cycles in the latter. In this case we say that $(E,\Theta,\pi)$ is a {\em matching type} Lefschetz fibration. Note that this is just a slightly less sophisticated analogue of a {\em Lefschetz bifibration} as considered in \cite[15e]{seidelbook}. Our main interest in matching type Lefschetz fibrations is that questions about Floer theory of the vanishing cycles can sometimes be converted to questions about pseudoholomorphic sections of the auxiliary Lefschetz fibration.

\sss

Finally, it will be important for us to understand the effect of Dehn twisting one matching cycle about another, at least in the following special case. Consider two matching paths $\gamma_0$ and $\gamma_1$ which satisfy $\gamma_0(1) = \gamma_1(0)$ and are otherwise disjoint. Let $\gamma$ denote the path obtained by concatenating $\gamma_0$ and $\gamma_1$ and then taking a small right-handed pushoff to disjoin it from $\gamma_0(1)$. 
It is easy to check that $\gamma$ is again a matching path, and after taking care with framings and matching isotopies, we have
\begin{prop}{\cite[Lemma 16.13]{seidelbook}}\label{prop:matching cycle half-twist}
The matching cycle $L_\gamma$ is, up to Hamiltonian isotopy, given by Dehn twisting $L_{\gamma_0}$ about $L_{\gamma_1}$.
\end{prop}
\begin{remark}
More generally, half-twisting about a matching path in the base of a Lefschetz fibration induces a Dehn twist about the corresponding matching cycle in the total space.
\end{remark}

\subsection{Boundary conditions and pseudoholomorphic sections}\label{subsection:boundary conditions and pseudoholomorphic sections}

Let $(E,\Theta,\pi)$ be a Lefschetz fibration over $\D^2$.
In order to discuss pseudoholomorphic sections, we first need to explain what types of boundary conditions we wish to allow.
By a {\em polygon in $\D^2$} we mean a simply-connected closed subset $\PP \subset \D^2$ whose boundary is smooth apart from $d \geq 0$ (convex) corners $v_1,...,v_d \in \bdy \PP$, each of which is a critical value of $\pi$.
In the sequel we will sometimes focus our attention on the part of the Lefschetz fibration lying above $\PP$,
and we refer to any restricted triple of the form $(E_\PP := \pi^{-1}(\PP),\Theta|_{E_{\PP}},\pi|_{E_{\PP}})$ as a {\em Lefschetz fibration over $\PP$}.

\begin{definition}
An {\em exact Lefschetz boundary condition over $\PP$} is an immersed submanifold ${Q \subset \pi^{-1}(\bdy \PP)}$ such that:
\begin{itemize}
\item for each $z \in \bdy \PP \setminus \{v_1,...,v_d\}$, the fiber $Q_z := Q \cap E_z$ is a connected exact Lagrangian submanifold of $(E_z,\Theta_z)$
\item for each vertex $v_i \in \bdy \PP$, $Q_{v_i}$ is the unique critical point of $\pi$ in $E_{v_i}$
\item if $\gamma: [0,1] \hookrightarrow  \bdy \PP$ is an embedded path which is disjoint from $\{v_1,...,v_d\}$ except for possibly $\gamma(1)$,
then $Q_{\gamma(0)}$ is mapped to $Q_{\gamma(1)}$ under parallel transport.
\end{itemize}
\end{definition}
In particular, if $\PP$ has no vertices, for example if it is the entire disk $\D^2$, then $\pi|_Q: Q \rightarrow \bdy \D^2$ is a smooth fiber bundle and $Q$ is precisely a ``Lagrangian boundary condition" as in \cite[\S 17a]{seidelbook}.
On the other hand, if $\PP$ has at least two vertices then $Q$ is simply a union of naive matching cycles.
In general, if $(E,\Theta,\pi)$ is a Liouville Lefschetz fibration the parallel transport condition implies that $Q$ is an immersed exact Lagrangian submanifold of $(E,\Theta)$.

\sss

We also need to take some care in picking almost complex structures.
Let $(E,\Theta,\pi)$ be an exact Lefschetz fibration over $\D^2$, and let $J_\std$ and $i_\std$ denote the standard almost complex structures on $\C^n$ and $\C$ respectively.
We also denote the standard almost complex structure on $\D^2$ again by $i_\std$.
Let $\calJ_\pi$ denote the space of almost complex structures $J$ on $E$ such that:
\begin{enumerate}
\item
$D\pi \circ J = i_{\std} \circ D\pi$.
\item
for any $p \in E$, $(d\Theta)(\cdot,J\cdot)$ is symmetric and positive definite when restricted to the vertical tangent space $\ker(D_p\pi)$
\item 
on $\Op(\bdy_hE)$, we have $\Theta \circ J = d(e^r)$, where $r$ is the fiberwise collar coordinate
as in \S\ref{subsubsec:Liouville subdomains}
\item
each critical point $p$ of $\pi$ and its critical value $q = \pi(p)$ have model neighborhoods as in Definition \ref{def:lefschetz fibration}(3),(4) such that $J$ and $i$ are identified with the restrictions of $J_\std$ and $i_\std$ respectively.
\end{enumerate}
Note in particular that $J$ restricts to a contact type compatible almost complex structure on each regular fiber of $\pi$. A basic fact is that $\calJ_\pi$ is contractible.
By restricting this data, in particular replacing $i_{\std}$ with its restriction $i_{\PP}$ to $\PP$, we similarly define $\calJ_\pi$ when $(E,\Theta,\pi)$ is a Lefschetz fibration over a polygon $\PP \subset \D^2$.

\sss

Now suppose that $Q$ is an exact Lefschetz boundary condition over $\PP$ and we have $J \in \calJ_\pi$.
The space $\calM_{Q,J}$ of {\em pseudoholomorphic sections with boundary condition $Q$} is by definition the space of maps $u: \PP \rightarrow E$ such that:
\begin{enumerate}
\item
$\pi(u(z)) = z$ for all $z \in \PP.$
\item
$Du \circ i_\PP = J \circ Du$.
\item
$u(\bdy \PP) \subset Q$.
\end{enumerate}
The conditions on $\calJ_\pi$ ensure a maximum principle for the elements of $\calM_{Q,J}$ and the exactness assumptions rule out bubbling, so $\calM_{Q,J}$ is compact by a version of Gromov's compactness theorem. Moreover, $\calM_{Q,J}$ is regular and hence a smooth manifold for generic $J$. More precisely, let $\calJ_\pi^\reg \subset \calJ_\pi$ denote the subspace of those $J$ for which $\calM_{Q,J}$ is regular.
 Then for any $J \in \calJ_\pi$ and any nonempty open set $U \subset \Int(\PP)$, there is a $J' \in \calJ_\pi^\reg$ which is $C^\infty$-close to $J$ and coincides with $J$ outside of $\pi^{-1}(U)$.
In fact, after a small local deformation of $(E,\Theta,\pi)$, we can further assume that $J'$ is {\em horizontal}, meaning that $d\Theta(\cdot,J\cdot)$ is symmetric (see \cite[Lemma 2.4]{seidel2003long} and \cite[\S 17a]{seidelbook}).
In particular, if $(E,\Theta,\pi)$ is a Liouville Lefschetz fibration, such an almost complex structure is compatible with $d\Theta$.
From now on we will assume that chosen elements of $\calJ_\pi^\reg$ are horizontal unless stated otherwise. 

For each regular value $p \in \bdy \PP$, evaluating sections at $p$ gives a map
$\ev_p: \calM_{Q,J} \rightarrow E_p$. 
By a standard homotopy argument, the cobordism class of $\ev_p$ (and in particular its homology class) is invariant under suitable deformations of $(E,\Theta,\pi)$, $Q$, and $J$.
Moreover, we can arrange that $\ev_p$ is transverse to any fixed smooth cycle in $E_p$ after performing a $C^\infty$-small deformation of $J$ supported in $\pi^{-1}(U)$,
for any given open subset $U \subset \Int(\PP)$ (see \cite[Lemma 2.5]{seidel2003long}).

\subsection{Gluing pseudoholomorphic sections}

We will make use of a gluing theorem for pseudoholomorphic sections of Lefschetz fibrations.
This was introduced by Seidel in \cite{seidel2003long} to establish a long exact sequence in Floer cohomology.
Although we will only need to apply it in a special case, we state the result in somewhat greater generality. 
Strictly speaking our stated version is more general than the setup in \cite{seidel2003long} since we allow immersed boundary conditions, but there is no essential difference since the gluing and perturbations take place away from the corner points of $\PP$ and $\PP'$.
 Let $(E,\Theta,\pi),(E',\Theta',\pi')$ be exact Lefschetz fibrations over $\PP,\PP'$ with exact Lefschetz boundary conditions $Q,Q'$ respectively. Fix non-vertex points $p \in \bdy\PP$ and $p' \in \bdy\PP'$ and strip-like ends 
\begin{align}
\epsilon_p: \R_+ \times [0,1] \hookrightarrow \PP,\;\;\;\;\; \epsilon_{p'}: \R_- \times [0,1] \hookrightarrow \PP'
\end{align}
at $p$ and $p'$ respectively.
The glued polygon $\PP\bcs\PP'$ is formed by removing $\epsilon_p([\ell,+\infty) \times [0,1])$ and $\epsilon_{p'}((-\infty,-\ell]\times [0,1])$ from $\PP \setminus \{p\}$ and $\PP' \setminus \{p'\}$, and then identifying the remaining surfaces via $\epsilon_p(s,t) \sim \epsilon_{p'}(s-\ell,t)$. 
As usual, this construction involves a choice of gluing parameter $\ell \in (0,\infty)$.

Now suppose there is an exact symplectomorphism $\Phi$ from $(E_p,\Theta_p)$ to $(E_{p'},\Theta'_{p'})$ which satisfies $\Phi(Q_p) = Q'_{p'}$. In this case we can form the fiber connect sum $E\bcs E'$, which is equipped with an exact Lefschetz fibration $\pi\bcs\pi': E\bcs E' \rightarrow \PP \bcs \PP'$ and an exact Lefschetz boundary condition $Q \bcs Q'$.
The construction is straightforward if we assume that there exists a neighborhood $U$ of $p$ and a diffeomorphism $F: E_p \times U \rightarrow \pi^{-1}(U)$
such that $F^*\Theta = \Theta_p$ and $F^{-1}(Q)  = Q_p \times (U \cap \bdy \PP)$, and similarly near $p'$. 
In general, we can reduce to this situation by suitable deformations (see \cite[\S 2.1]{seidel2003long}).

Take $J \in \calJ_\pi^\reg$ and $J' \in \calJ_{\pi'}^\reg$.
Assume $F^*J$ is a split almost complex structure on $E_p \times U$,
and similarly for $J'$ near $p'$.
Assume also that $\Phi_*J|_{E_p} = J'|_{E_{p'}}$.
In this case, $J$ and $J'$ can be naturally glued to form $J \bcs J' \in \calJ_{\pi\bcs\pi'}$.
As mentioned at the end of \S\ref{subsection:boundary conditions and pseudoholomorphic sections}, we can further arrange that $\Phi \circ \ev_p: \calM_{Q,J} \rightarrow E_{p}$ 
and $\ev_{p'}: \calM_{Q',J'} \rightarrow E'_{p'}$ are mutually transverse.

\begin{prop}\cite[\S 2]{seidel2003long}\label{prop:gluing for sections}
For $\ell$ sufficiently large, we have $J \bcs J' \in \calJ^\reg_{\pi\bcs\pi'}$,
and there is a diffeomorphism
\begin{align*}
\calM_{Q\bcs Q',J\bcs J'} \cong \calM_{Q,J} \underset{\Phi\circ \ev_p,\ev_{p'}}\times \calM_{Q',J'},
\end{align*}
where the right hand side denotes the fiber product with respect to $\Phi \circ \ev_p$ and $\ev_{p'}$.
\end{prop}
\noindent As a basic further observation, suppose $q \in \bdy\PP$ is a regular value which is disjoint from the gluing region. 
Then the evaluation map $\ev_q: \calM_{Q\bcs Q',J\bcs J'} \rightarrow (E \bcs E')_q = E_q$ is identified with the composition of the projection $ \calM_{Q,J} \underset{\Phi\circ \ev_p,\ev_{p'}}\times \calM_{Q',J'} \rightarrow \calM_{Q,J}$ and the evaluation map $\ev_q: \calM_{Q,J} \rightarrow E_q$.

\subsection{The model computation}\label{subsec:the model computation}

We conclude this section by discussing sections of the model Lefschetz fibration $\pi_\std: E_\std^{2n} \rightarrow \D^2$, where
\begin{align*}
E_\std^{2n} := \{x \in \C^n\;:\; |\pi_\std(x)| \leq 1,\; ||x||^4 - |\pi_\std(x)|^2 \leq 4\}
\end{align*}
and $E_\std$ is equipped with the restriction of the standard K\"ahler potential $\Theta_\std$.
For $z \in \C$, let
\begin{align*}
\Sigma_z^{n-1} := \{\pm \sqrt{z}x\;: \; x \in S^{n-1} \subset \R^n \subset \C^n\}.
\end{align*}
Note that $\pi_\std(\Sigma_z) = z$. In fact, for any $z \neq 0$, $\Sigma_z$ is an exact Lagrangian sphere in the fiber $(\pi_\std^{-1}(z),\Theta_\std|_{\pi_\std^{-1}(z)})$ of the model Lefschetz map.
We consider the model exact Lefschetz boundary condition $Q_\std$, given by
\begin{align*}
Q_\std := \bigcup_{z \in \bdy \D^2}\Sigma_z.
\end{align*}
In this special case, the space $\calM_{Q_\std,J_\std}$ of $(i_\std,J_\std)$-pseudoholomorphic sections with boundary condition $Q_\std$ can be explicitly computed:
\begin{lemma}{\cite[Lemma 2.16]{seidel2003long}}
$\calM_{Q_\std,J_\std}$ consists of the maps $u_a(z) = za + \ovl{a}$, for $a \in \C^n$ such that $\pi_\std(a) = 0$ and $||a||^2 = 1/2$. Moreover, each $u_a$ is regular.
\end{lemma}
Let $\ev_1: \calM_{Q_\std,J_\std} \rightarrow \Sigma_1$ denote the evaluation map at $1 \in \bdy \D^2$. We point out that $\ev_1$ can be identified with the projection $S(T^*S^{n-1}) \rightarrow S^{n-1}$, where $S(T^*S^{n-1})$ denotes the unit sphere cotangent bundle of $S^{n-1}$. In particular, $\ev_1$ is null-cobordant.
Slightly more generally, let $(E,\Theta,\pi)$ be any exact Lefschetz fibration over $\D^2$ with a single critical point, and let $Q$ be an exact Lefschetz boundary condition which is ``standard", meaning
that $Q_*$ is Hamiltonian isotopic to the vanishing cycle in $(E_*,d\Theta_*)$.
In this case, by deforming $(E,\Theta,\pi)$ and $Q$ to the preimage of a small disk around the critical value and applying a maximum principle, we are essentially reduced to the situation of the model Lefschetz fibration.
Using this, Seidel proves:
\begin{prop}{\cite[Proposition 2.13]{seidel2003long}}\label{prop:standard fibrations} For any $J \in \calJ_\pi^\reg$ and $p \in \bdy \D^2$, the evaluation map $\ev_p: \calM_{Q,J} \rightarrow E_p$ is null-cobordant. 
\end{prop}

\sss

Proposition \ref{prop:standard fibrations} can be interpreted as a triviality  
statement for sections of Lefschetz fibrations with standard boundary conditions.
In contrast, we will show that the space of sections becomes nontrivial after turning on a suitable B-field or bulk deformation.
Assume now that $n$ is even, say $n = 2m$. 
Let $\eta_\std: [0,1] \rightarrow \D^2$ be the vanishing path for $(E_\std,\Theta_\std,\pi_\std)$ which lies on the real line, say given by $\eta_\std(t) = 1-t$, and let $T_{\eta_\std} = \cup_{z \in [0,1]}\Sigma_z$ be the associated thimble.
Although $T_{\eta_\std}$ is Lagrangian, we view it as just a smooth cycle. As defined it intersects $Q_\std$ in the entire sphere $\Sigma_1$, but we can disjoin it from $Q_\std$ by a small perturbation.
Explicitly, consider the perturbation $T_\e \subset E_\std$ given by:
\begin{align*}
T_\e := \left\{ r(\e_c s_1 - i\e t_1,\e_c t_1 + i\e s_1,...,\e_c s_m - i\e t_m,\e_c t_m + i\e s_m) \in \C^n\;:\; \sum_{i=1}^m (s_i^2 + t_i^2) = 1,\; 0 \leq r \leq 1 \right\},
\end{align*}
where $s_1,...,s_m,t_1,...,t_m,r$ are real variables, $\e > 0$ is a small fixed number and $\e_c := \sqrt{1 + \e^2}$. Note that $T_\e$ coincides with $T_{\eta_\std}$ in the limiting case $\e = 0$.
One can easily check that $\pi_\std(T_\e) = [0,1]$ 
and that $T_\e$ is disjoint from $Q_\std$.
\begin{lemma}\label{lem:intersection computation}
There is a unique $(a,z) \in \C^n \times \D^2$ satisfying the system:
\begin{enumerate}
\item
$\pi_\std(a) = 0$
\item
$||a||^2 = 1/2$
\item
$u_{a}(1) = (1,0,...,0)$
\item
$u_{a}(z) \in T_\e$.
\end{enumerate}
Moreover, this solution is regular: the space of pairs $(a,z)$ satisfying (1),(2),(3), equipped with the map to $\C^n$ sending $(a,z)$ to $u_a(z)$, is tranverse to $T_\eps$.
\end{lemma}
\noindent Note that the first two conditions are equivalent to $u_{a}$ being an element of $\calM_{Q_\std,J_\std}$.
\begin{proof}
For $a = (a_1,...,a_n)$, the condition $u_a(1) = (1,0,...,0)$ amounts to the equations
\begin{align*}
\op{Re}(a_1) = 1/2,\;\;\; \op{Re}(a_2) = ... = \op{Re}(a_n) = 0,
\end{align*}
so we can write $a = (1/2 + ic_1,ic_2,...,ic_n)$ for $c_1,...,c_n \in \R$.
Now consider the condition $u_a(z) \in T_\e$.
Since $\pi_\std(T_\e) = [0,1]$ and $T_\e \cap \Sigma_1 = \nil$, we can write $z = R$ for some $R \in [0,1)$.
We then have
\begin{align*}
u_a(R) = (R/2 + ic_1R + 1/2 - ic_1,ic_2R - ic_2,...,ic_nR - ic_n),
\end{align*}
and this must be of the form
\begin{align*}
\sqrt{R}(\e_c s_1 - i\e t_1,\e_c t_1 + i\e s_1,..., \e_c s_m - i\e t_m, \e_c t_m + i\e s_m)
\end{align*}
for some $s_1,t_1,...,s_m,t_m \in \R$ satisfying $\sum_{i=1}^m(s_i^2 + t_i^2) = 1$.
In particular, comparing real and imaginary parts, we have:
\begin{alignat*}{5}
R/2 + 1/2 &= \sqrt{R} \ec s_1  \quad& c_1 R - c_1 &= -\sqrt{R}\e t_1\\
0 &= \sqrt{R}\ec t_1  \quad& c_2 R - c_2 &= \sqrt{R}\e s_1\\
0 &= \sqrt{R}\ec s_2  \quad& c_3 R - c_3 &= -\sqrt{R}\e t_2\\
0 &= \sqrt{R}\ec t_2  \quad& c_4 R - c_4 &= \sqrt{R}\e s_2\\
&...  \quad& ... & \\
0 &= \sqrt{R}\ec t_m  \quad&  c_n R - c_n &= \sqrt{R}\e s_m.\\
\end{alignat*}
This forces $s_2,...,s_m,t_1,...,t_m,c_1,c_3,...,c_n$ to vanish.
Combining the remaining nontrivial equations with the conditions $\pi_\std(a) = 0$ and $||a|| = 1/2$, the system reduces to:
\begin{align*}
c_2^2 = 1/4\\
R/2 + 1/2 = \sqrt{R} \e_c s_1\\
c_2 R - c_2 = \sqrt{R} \e s_1\\
s_1^2 = 1\\
0 \leq R < 1.
\end{align*}
Substituting $c_2 = \pm 1/2$ into the second and third equation yields
$R = \dfrac{\pm \e_c + \e}{\pm \e_c - \e}$. The condition $R < 1$ forces the bottom sign, 
and also $s_1 = 1$.
This determines all of the variables, and we just need to check that the remaining equation $-R/2 + 1/2 = \sqrt{R}\e$ is consistent with $R = \dfrac{\e_c - \e}{\e_c + \e}$.
This follows using $\e_c^2 - \e^2 = 1$.

\sss

As for the regularity statement, put $$\calS := \left\{(a,z) \in \C^n \times \D^2\;:\; \pi_\std(a) = 0,\; ||a||^2 = 1/2,\; u_a(1) = (1,0,...,0)\right\}.$$
Note that this can be identified with a fiber of the projection map $S(T^*S^{n-1}) \rightarrow S^{n-1}$,
and in fact we have
$$ \calS = \left\{(1/2,ic_2,...,ic_n,z) \in \C^n \times \D^2\; : \; \sum_{i=2}^n c_i^2 = 1/4\right\}.$$
Let $u: \calS \rightarrow \C^n$ denote the map sending $(a,z)$ to $u_a(z) = za + \ovl{a} = (\tfrac{z+1}{2},ic_2(z-1),...,ic_n(z-1)).$
Put 
$$ \calT := \left\{ (r,s_1,t_1,\dots,s_m,t_m)\;:\; \sum_{i=1}^m(s_i^2 + t_i)^2 = 1,\; 0 \leq r \leq 1 \right\},$$
and let $v: \calT \rightarrow \C^n$ be the map defined by
$$v(r,s_1,t_1,\dots,s_m,t_m) = r(\e_c s_1 - i\e t_1,\e_c t_1 + i\e s_1,...,\e_c s_m - i\e t_m,\e_c t_m + i\e s_m).$$
By the above, the images of $u$ and $v$ intersect uniquely at the point $p^0 = \sqrt{\tfrac{\eps_c -\eps}{\eps_c + \eps}}(\eps_c,i\eps,0,...,0) \in \C^n$,
and we have
$u(a^0,z^0) = p^0$ for 
$a^0 := (1/2,-i/2,0,...,0) \in \C^n$ and $z^0 = \tfrac{\eps_c - \eps}{\eps_c + \eps},$
and we have $v(r^0,s_1^0,t_1^0,...,s_m^0,t_m^0) = p^0$ for
$r^0 = \sqrt{\tfrac{\eps_c -\eps}{\eps_c + \eps}}$, $s_1^0 = 1$, and $s_2^0 = ... = s_m^0 = t_1 = ... = t_m = 0$.
We need to show that the linearized images $u_*(T_{(a^0,z^0)}\calS), v_*(T_{(r^0,s_1^0,t_1^0,...,s_m^0,t_m^0)}\calT) \subset T_{p^0}\C^n$ intersect transversely.

Let $z = x + iy$ denote the coordinates on $\D^2$, and let $(z_1,...,z_n) = (x_1 + iy_1,...,x_n + iy_n)$ denote the coordinates on $\C^n$.
The tangent space $T_{(a^0,z^0)}\calT$ is naturally identified with the span of $\bdy_{x},\bdy_y,\bdy_{c_3},...,\bdy_{c_n}$, and at the point $(a^0,z^0)$ we have
\begin{align*}
u_*\bdy_x &= \tfrac{1}{2}\bdy_{x_1} - \tfrac{1}{2}\bdy_{y_2}\\
u_*\bdy_y &= \tfrac{1}{2}\bdy_{y_1} + \tfrac{1}{2} \bdy_{x_2}\\
u_*\bdy_{c_3} &= \tfrac{-2\eps}{\eps_c + \eps}\bdy_{y_3}\\
...\\
u_*\bdy_{c_n} &=  \tfrac{-2\eps}{\eps_c + \eps}\bdy_{y_n}.
\end{align*}
Similarly, the tangent space $T_{(r^0,s_1^0,t_1^0,...,s_m^0,t_m^0)}\calT$ is naturally identified with the span of $\bdy_r,\bdy_{s_2},...,\bdy_{s_m},\bdy_{t_1},...,\bdy_{t_m}$, and at the point $(r^0,s_1^0,t_1^0,...,s_m^0,t_m^0)$ we have
\begin{align*}
v_*\bdy_r &= \eps_c\bdy_{x_1} + \eps\bdy_{y_2}\\
v_* \bdy_{t_1} &= -r^0\eps\bdy_{y_1} + r^0\eps_c\bdy_{x_2}\\
v_*\bdy_{s_2} &= r^0\eps_c\bdy_{x_3} + r^0\eps\bdy_{y_4}\\
...\\
v_*\bdy_{s_m} &= r^0\eps_c\bdy_{x_{n-1}} + r^0\eps\bdy_{y_n}\\
v_*\bdy_{t_m} &= -r^0\eps\bdy_{y_{n-1}} + r^0\eps_c\bdy_{x_n}.
\end{align*}
By inspection, these two tangent spaces span $T_{p^0}\C^n$, and hence by dimension considerations they intersect transversely.
\end{proof}

Specializing to the case $2n = 4$, $\calM_{Q_\std,J_\std}$ is a disjoint union of two circles, say $C_0$ and $C_1$, and
$\ev_1: \calM_{Q_\std,J_\std} \rightarrow \Sigma_1$ is a two-fold covering map.
Let $\Omega_\e$ be a closed two-form representing the Poincar\'e dual of $T_\e$ as an element of
$H^2(E_\std,Q_\std;\R)$.
Lemma \ref{lem:intersection computation} shows that, up to switching the roles of $C_0$ and $C_1$, we have
\begin{align*}
\int u^*\Omega_\e = 0\;\; \text{   for   } u \in C_0\\
\int u^*\Omega_\e \neq 0\;\; \text{   for   } u \in C_1.
\end{align*}

\section{Squared Dehn twists, B-fields, and bulk deformations}\label{sec:squared dehn twists}

\subsection{A quasi-isomorphism criterion}\label{subsec:quasi-iso criterion}

In this subsection we give a simple quasi-isomorphism criterion for two Lagrangians. The premise is roughly that we can check quasi-isomorphism of two Lagrangians by counting certain pseudoholomorphic strips with a boundary marked point, rather than counting pseudoholomorphic triangles.
In more detail, 
let $L_0$ and $L_1$ be Lagrangian spheres in a Liouville domain with trivial first Chern class,
and let $J$ be a $(d\theta)$-compatible almost complex structure which is contact type near $\bdy M$.
Assume $L_0$ and $L_1$ intersect transversely in precisely two points $a$ and $b$, and set $\calM(a,b)$ to be the space of maps $u: \R \times [0,1] \rightarrow M$ such that:
\begin{itemize}
\item
$u$ is $(i_\std,J)$-holomorphic
\item
$u(\R \times \{0\}) \subset L_0$ and $u(\R \times \{1\}) \subset L_1$
\item $\lim\limits_{s \rightarrow -\infty} u(s,\cdot) = a$ and $\lim\limits_{s \rightarrow +\infty} u(s,\cdot) = b$.
\end{itemize}
Assume that $\calM(a,b)$ is regular.
Note that $a$ and $b$ define Floer cochains $a,b \in CF(L_0,L_1)$ after choosing Floer data such that $H_{L_0,L_1} \equiv 0$ and $J_{L_0,L_1} \equiv J$.
Reversing the order of $L_0$ and $L_1$, we also have dual cochains
$a^\vee,b^\vee \in CF(L_1,L_0)$
 with respect to the Floer data $H_{L_1,L_0} \equiv 0$ and $J_{L_1,L_0} \equiv J$.
In particular, choosing gradings on $L_0$ and $L_1$, we have
associated degrees $|a|,|b|,|a^\vee|,|b^\vee| \in \Z$, with $|a^\vee| = n - |a|$ and $|b^\vee| = n - |b|$.
After a shift of gradings, we can further assume that $|b| = 0$. 

\begin{prop}\label{prop:quasi-isomorphism criterion twisted}
Let $\ev_{(0,0)}: \calM(a,b) \rightarrow L_0$ denote the evaluation map at the point
$(0,0) \in \R \times [0,1]$, let $p \in L_0 \setminus \{a.b\}$ be a regular value of $\ev_{(0,0)}$, and set $\calM^p(a,b) := \ev^{-1}_{(0,0)}(p)$.
Let $\Omega$ be a closed two-form on $M$ with support disjoint from $\Op(L_0 \cup L_1)$.
Then $L_0$ and $L_1$ are quasi-isomorphic in $\twfuk(M,\theta)$ if the quantity
\begin{align*}
\sum_{u \in \calM^p(a,b)^{\oo}} t^{u^*\Omega}\signu \in \K \tag{$*$}
\end{align*}
is nonzero.
\end{prop}
\begin{proof}
Since $\calM^p(a,b)$ is zero-dimensional, we must have $|a| = n$. 
In particular, for index reasons there are no rigid (up to translation) Floer strips for $CF(L_0,L_1)$ and $CF(L_1,L_0)$, so $a,b,a^\vee,b^\vee$ are actually Floer cocycles.
We can construct $\twfuk(\{L_0,L_1\})$ using the Floer data for $(L_0,L_1)$ and $(L_1,L_0)$
mentioned above.
Regarding the Floer datum for $(L_0,L_0)$, we take $J_{L_0,L_0} \equiv J$, and construct $H_{L_0,L_0}$ as follows.
Let $h: L_0 \rightarrow \R$ be a Morse function with a unique local minimum at $p$.
Let $\wt{h}: T^*L_0 \rightarrow \R$ be the pullback of $h$ under the projection $T^*L_0 \rightarrow L_0$, cut off to zero outside of a small neighborhood of $L_0$.
Now let $H_{L_0,L_0} := \e\wt{h}$ for $\e > 0$ sufficiently small, where we implant $\wt{h}$ into $M$ using a Weinstein neighborhood of $L_0$.
Assuming $h$ is Morse--Smale with respect to the metric $g_{J_0} := (d\theta)(\cdot,J_0\cdot)$,
an argument originating with Floer \cite[Thm 2]{floer1989witten} identifies $\hom(L_0,L_0)$ with the Morse cochain complex of $(h,g_{J_0})$. In particular, $p$ represents the unit $e_{L_0}$ in $\Hhom(L_0,L_0)$.

Now consider the Riemann disk with three boundary marked points and Lagrangian labels $(L_0,L_1,L_0)$, and let $(K_\Delta,J_\Delta)$ denote the associated choice perturbation data for $\twfuk(M,\theta)$.
By a standard gluing argument, for suitable $K$ sufficiently small and $J_\Delta$ sufficiently close to $J_0$,
we can arrange that $\calM(p,b,a^\vee)$ is in bijective correspondence with $\calM^p(a,b)$,
and in fact $(*)$ is precisely the coefficient of $p$ in $\twmu^2(a^\vee,b)$.

Observe that $\Hhom(L_0,L_0) \cong H(L_0;\K)$ 
is the $\K$-algebra generated by $e_{L_0}$ and an element $f_{L_0}$ of degree $n$,
subject to the relations $e_{L_0}^2 = e_{L_0}$, $e_{L_0}f_{L_0} = f_{L_0}e_{L_0} = f_{L_0}$ and $f_{L_0}^2 = 0$.
In particular, $(*)$ nonzero implies that
$[a^\vee]\cdot [b]$ is invertible, and hence $x\cdot [b] = e_{L_0}$ for some $x \in \Hhom(L_1,L_0)$.
Namely, we have $[a^\vee]\cdot [b] = Ce_{L_0}$ for some nonzero $C\in\K$, and we put $x := C^{-1}[a^\vee]$.
Following an argument from \cite[\S 4.4]{keating2014dehn}, it follows that $[b] \cdot x$ is nonzero and idempotent, and this forces $[b] \cdot x = e_{L_0}$.
It follows that $[b]$ and $x$ are inverses.
\end{proof}

There is also the following parallel version in the context of bulk deformations.
\begin{prop}\label{prop:quasi-isomorphism criterion bulk deformed}
Let $p \in L_0 \setminus \{a,b\}$ be a regular value of $\ev_{(0,0)}: \calM(a,b) \rightarrow L_0$, set ${\calM^p_1(a,b) := \ev_{(0,0)}^{-1}(p) \times \R \times [0,1]}$,
and let $\ev: \calM_1^p(a,b) \rightarrow M$ denote the associated evaluation map.
Let $i_{\mho}: (\mho,\bdy \mho) \rightarrow (M \setminus \Op(L_0 \cup L_1),\bdy M)$ be a smooth half-dimensional cycle. Assume $\ev$ and $\mho$ are transverse,
and set $\calM_{1;{\mho}}^p(a,b) := \calM_1^p(a,b) \underset{\ev,i_{\mho}}\times \mho$.
Then $L_0$ and $L_1$ are quasi-isomorphic in $\fuk_\mho(M,\theta)$ if the quantity
\begin{align*}
\sum_{u \in \calM_{1;{\mho}}^p(a,b)^{\oo}}\signu \in \Z\tag{**}
\end{align*}
is nonzero.
\end{prop}
\noindent Note that we are implicitly assuming $n > 2$, since we are bulk deforming by a cycle of codimension $n$. 
\begin{proof}
We can again assume $|b| = 0$, and since $\calM_{1;{\mho}}^p(a,b)$ is zero-dimensional we must have $|a| = 2n - 2$,
and therefore $a,b,a^\vee,b^\vee$ define Floer cocycles.
Similar to the proof of Proposition \ref{prop:quasi-isomorphism criterion twisted},
we can construct $\fuk_\mho(M,\theta)$ such that the quantity $(**)$ is the coefficient of $p$ in $\mu^2_{1;{\mho}}(a^\vee/\hbar,b)$.
We have
$$[a^{\vee}] \cdot [b] = \hbar C e_{L_0} + D f_{L_0}$$
for nonzero $C \in \KL_0$ and some $D \in \KL$.
Since $|a^\vee| = n - |a| = 2-n$ and $|\hbar| = 2- n$, by degree considerations we must have
$|D| = 2 - 2n$, and hence $D$ vanishes unless $2- 2n$ is divisible by $2-n$, in which case we have $D = \hbar^{k}C'$ for $k = \tfrac{2-2n}{2-n}$ and some $C' \in \KL_0$.
In any case, $[a^\vee] \cdot [b]$ is invertible, with inverse $C^{-2}\hbar^{-1}(Ce_{L_0} - \hbar^{k-1}C'f_{L_0})$.
Putting $x := ([a^\vee] \cdot [b])^{-1} \cdot [a^{\vee}]$, we have $x \cdot [b] = e_{L_0}$ and 
$[b] \cdot x \in \Hhom(L_1,L_1)$ is a nonzero idempotent of degree zero. 
Note that $[b] \cdot x$ must be of the form $Ae_{L_1} + Bf_{L_1}$ for some $A \in \KL_0$ and $B \in \KL$ with $|B| = -n$, and the idempotence condition then forces $B = 0$ and $A = 1$.
It follows that $[b]$ and $x$ and inverses.
\end{proof}

\begin{remark}\label{rem:converse to quasi-iso criterion}
A similar analysis shows that the converses of Proposition \ref{prop:quasi-isomorphism criterion twisted} and Proposition
\ref{prop:quasi-isomorphism criterion bulk deformed} also hold (provided we specify $|a| = n$).
\end{remark}

\begin{remark}
There is a Morse--Bott version of the Fukaya category, constructed in detail by Sheridan in \cite[\S 4]{sheridan2011homological}, for which the endomorphism space of any Lagrangian is by definition the Morse complex of a chosen Morse function. One can also adapt twistings and bulk deformations to this Morse--Bott setup.
From this point of view the two propositions above become even simpler.
\end{remark}

\subsection{Completing the proofs of Theorem \ref{thm:main thm twisted} and Theorem \ref{thm:main thm bulk deformed}}
In this subsection we complete the proofs of Theorem \ref{thm:main thm twisted} and Theorem \ref{thm:main thm bulk deformed} from the introduction.
We begin by discussing a special case, namely the example studied by Maydanskiy in \cite{maydanskiy2009} (see also \cite{harris2012distinguishing,murphysiegel}).
This turns out to be in some sense our universal example.

For $n \geq 2$ even, let $A_2^{2n}$ denote the $(2n)$-dimensional $A_2$ Milnor fiber, i.e. the Liouville domain given by plumbing together two copies of the unit disk cotangent bundle $D^*S^n$ of $S^n$.
There is a Liouville Lefschetz fibration $\pi: A_2^{2n} \rightarrow \D^2$ with fiber $D^*S^{n-1}$ and three vanishing cycles, each Hamiltonian isotopic to the zero section.
With respect to this auxiliary Lefschetz fibration, the two core Lagrangian spheres in $A_2^{2n}$, which we will denote by $L$ and $S$, are matching cycles.
The associated matching paths $\gamma_L$ and $\gamma_S$ intersect exactly once at a critical value of $\pi$, as in Figure \ref{matchingcycles}.
Let $\gamma_{L'}$ denote a small pushoff of $\gamma_L$ which has the same endpoints and is otherwise disjoint from $\gamma_L$, and let $L'$ denote the associated matching cycle. After a suitable deformation, we can assume that $\gamma_L,\gamma_L',\gamma_S$ are naive matching paths. 
In particular, $L$ and $L'$ intersect transversely in precisely two points.
\begin{figure}
 \centering
 \includegraphics[scale=.8]{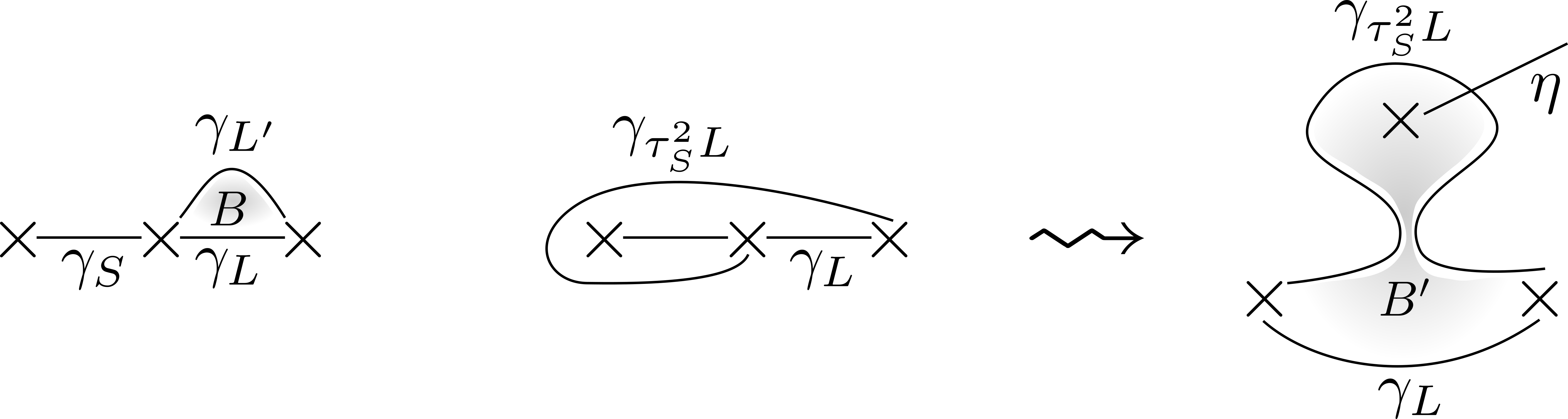}
 \caption{Some matching paths.}
 \label{matchingcycles}
\end{figure} 

Note that $\gamma_L$ and $\gamma_L'$ cobound a bigon $B \subset \D^2$.
Restricting $\pi$ to $E_B := \pi^{-1}(B)$, we get an exact Lefschetz fibration over $B$
with exact Lefschetz boundary condition $L \cup L'$.
Applying Proposition \ref{prop:matching cycle half-twist} twice, we see that $\tau_S^2 L$ is a matching cycle with matching path $\gamma_{\tau_S^2 L}$.
In particular, $\gamma_{\tau_S^2 L}$ and $\gamma_L$ cobound a bigon $B' \subset \D^2$,
and by restricting $\pi$ to $E_{B'} := \pi^{-1}(B')$ we get an exact Lefschetz fibration over $B'$ with exact Lefschetz boundary condition $L \cup \tau_S^2 L$. 
As suggested by Figure \ref{matchingcycles}, $\pi_{B'}$ is, up to deformation, the fiber connect sum of $\pi_B$ with the model Lefschetz fibration $\pi_\std: E_\std \rightarrow \D^2$,
and the boundary condition $L \cup \tau_S^2 L$ is the fiber connect sum of $L \cup L'$ with $Q_\std$.

Consider the case $2n = 4$.
We can combine Proposition \ref{prop:gluing for sections} with Lemma \ref{lem:intersection computation} to understand sections of $\pi_{B'}$ with boundary condition $L \cup \tau_S^2 L$, at least for sufficiently large gluing parameter and suitable $J \in \calJ^\reg_{\pi_B \bcs \pi_\std}$.
Indeed, let $J'$ denote the almost complex structure on $E_B$ induced by $J$ in the gluing limit. Consider $q \in \Int(\gamma_L)$ and $p \in \pi_{B'}^{-1}(q)$, and let 
$\ev_q: \calM_{L \cup \tau_S^2 L,J} \rightarrow \pi_{B'}^{-1}(q)$  and $\ev_q: \calM_{L \cup L',J'} \rightarrow \pi_{B}^{-1}(q)$ denote the natural evaluation maps.
Since $L'$ is a small Hamiltonian pushoff of $L$, the moduli space $\calM^p_{L \cup L',J'}$ of pseudoholomorphic sections of $\pi_B$ passing through $p$ consists of a single regular rigid curve, and hence the fiber product in Proposition~\ref{prop:gluing for sections} reduces $\calM^p_{L \cup \tau_S^2 L,J}$ to the model situation of Lemma~\ref{lem:intersection computation}.
The upshot is that
the moduli space $\calM^p_{L\cup \tau_S^2 L,J} := \ev_q^{-1}(p)$ consists of two points $u_0$ and $u_1$, with $\int u_0^*\Omega_\e = 0$ and
$\int u_1^*\Omega_\e \neq 0$.
Here $\Omega_\e$ is a closed two-form on $A_2^4$ with support disjoint from $\Op(L \cup \tau_S^2 L)$,
implanted from the one on $E_\std$ described at the end of \S\ref{subsec:the model computation}.
In other words, there are precisely two sections of the Lefschetz fibration $\pi_{B'}$ with the given boundary conditions which pass through the generic point $p$, and these are distinguished by $\Omega_\e$.
Proposition \ref{prop:quasi-isomorphism criterion twisted} now shows that $L$ and $\tau_S^2 L$ are quasi-isomorphic in $\twfukblank_{\Omega_\e}(A_2^4)$.

Similarly, for $2n > 4$, Proposition \ref{prop:gluing for sections} and Lemma \ref{lem:intersection computation} together show that $\calM^p_{L\cup \tau_S^2 L,J}$ has dimension $l-2$, and is cut down to a single point after adding an interior point constraint in $T_\e$.
It then follows from Proposition \ref{prop:quasi-isomorphism criterion bulk deformed} that 
$L$ and $\tau_S^2 L$ are quasi-isomorphic in $\fuk_{{T_\e}}(A_2^4)$, where $i_{T_\e}: (T_\e,\bdy T_\e) \rightarrow (A_2^4,\bdy A_2^4)$ is given by the natural extension of $T_\e \subset E_\std$ from \S\ref{subsec:the model computation}.

In any dimension, a similar analysis using Remark \ref{rem:converse to quasi-iso criterion} shows that $L$ and $\tau_S^2 L$ are not quasi-isomorphic in $\fuk(A_2^{2n})$.
This can also be seen more directly by considering a test thimble $T_\eta$ with vanishing path $\eta$ as in Figure \ref{matchingcycles} (see the proof of \cite[Lemma 7.3]{maydanskiy2009}).

Finally, consider the more general situations of Theorem \ref{thm:main thm twisted} and Theorem \ref{thm:main thm bulk deformed}.
By a version of the Weinstein neighborhood theorem (see \cite{symington2001} for the case $n=2$), a neighborhood $U$ of $L \cup S$ is symplectomorphic to a neighborhood of the core spheres in $A_2^{2n}$.
After a deformation we can assume that $U$ is (exact) symplectomorphic to $A_2^{2n}$,
and we can also arrange that $\tau_S^2L$ is contained in $U$.
As in Lemma~\ref{lemma:Liouville subdomains}, the quasi-isomorphism question for $L$ and $\tau_{S^2}L$ is unchanged if we restrict the ambient space from $M$ to $U$.
Identifying $U$ with $A_2^{2n}$, the assumptions on $\Omega$ and ${\mho}$ imply that their restrictions to $U$
are, up to a scaling factor, (co)homologous to $\Omega_\e$ and ${T_\e}$ respectively.
This reduces the quasi-isomorphism question for $L$ and $\tau_S^2 L$ to the universal example already considered.


\section{From the fiber to the total space}\label{sec:from fiber to total space}

\subsection{The general setup}\label{subsec:the general setup}

Consider a Liouville Lefschetz fibration with a fixed basis of vanishing paths.
As mentioned at the end of $\S\ref{subsec:the basics}$, the total space 
is determined up to Liouville deformation equivalence by the ordered list of vanishing cycles in the fiber.
In light of this observation, it is natural to ask precisely how pseudoholomorphic curve counts in the fiber and total space are related.
For example, is it true that the Fukaya category with the vanishing cycles as objects 
determines the wrapped Fukaya category and symplectic cohomology of the total space?

In fact, at least as early as \cite{seidel2009symplectic}, Seidel gave explicit conjectural formulas for these two invariants in terms of the directed Fukaya category $\fukdir(V_1,...,V_k)$ and full Fukaya category $\fuk(V_1,...,V_k)$ of the vanishing cycles  (see $\S\ref{subsec:a picard--lefschetz}$ for definitions).
Namely, Seidel cooks up an auxiliary curved $\calA_\infty$ category $\euD$ which is explicitly defined in terms of $\fukdir(V_1,...,V_k)$ and $\fuk(V_1,...,V_k)$ and involves a formal variable $t$. 
By its construction, the objects of $\fukdir(V_1,...,V_k)$ can be pulled back to modules $\Pi(V_1),...,\Pi(V_k)$ over $\euD$.
The conjectures are then:
\begin{enumerate}
\item
the full subcategory $\wfuk(T_1,...,T_k)$ of the wrapped Fukaya category of the total space with objects the thimbles is quasi-isomorphic to the full subcategory of $\euD$-modules with objects $\Pi(V_1),...,\Pi(V_k)$
\item
the symplectic cohomology of the total space is isomorphic to the Hochschild homology of $\euD$
\end{enumerate}
(here $\euD$-modules and Hochschild homology are defined by taking into account the $t$-adic topology of $\euD$; see \cite[\S 6]{Seidelsubalgebras}).

Using the Legendrian surgery formulas from \cite{bourgeois2012effect} and techniques from symplectic field theory, proofs of both statements are given in the appendix of \cite{bourgeois2012effect}.
As a byproduct of the proof, one gets a rather explicit geometric understanding of $\euD$ in terms of Morse--Bott configurations of curves.
An alternative approach to the first statement has also been announced in the manuscript \cite{abouzaidseidelmanuscript}.
The precise formulations of these statements will not be relevant for us, but an immediate corollary is the following meta-principle, which is also stated as Property 2.6 from \cite{abouzaid2010altering}:
\begin{thm}\label{thm:meta-principle}
The wrapped Floer cohomology $\field$-modules $HW(T_i,T_i)$ depend only on $\fuk(V_1,...,V_k)$ up to order-preserving quasi-isomorphism.
\end{thm}
Here by {\em order-preserving quasi-isomorphism} we mean an $\calA_\infty$ quasi-isomorphism between two $\calA_\infty$ categories, each with $k$ ordered objects, which sends the $i$th object to the $i$th object for $i = 1,...,k$.
In $\S\ref{subsec:a picard--lefschetz}$ we explain how to deduce Theorem \ref{thm:meta-principle} using Lefschetz fibration tools already available in the literature.
In $\S\ref{subsec:incorporating twistings}$ we then discuss extensions of Thereom \ref{thm:meta-principle} in the presence of twistings and bulk deformations.

\sss

We point out that the statement (2) is closely related to (1) by general principles.
Indeed, much is already known or conjecturally known about the relationship between the wrapped Fukaya category of a Lefschetz fibration and the symplectic cohomology of the total space.
For one thing, as described in \S\ref{subsubsec:wrapped Floer as a module} it is a standard observation that the self wrapped Floer cohomology of any object in the wrapped Fukaya category admits the structure of a unital module over symplectic cohomology, and by considering units we immediately have:
\begin{prop}
Nontriviality of the wrapped Fukaya category implies nontriviality of symplectic cohomology.
\end{prop}
In the converse direction, work of Ganatra \cite{ganatra2013symplectic} shows that
the symplectic cohomology of any Liouville manifold is isomorphic to the Hochschild (co)homology of its wrapped Fukaya category, provided a certain nondegeneracy condition holds.
Moreover, it is expected that the Lefschetz thimbles of any Liouville Lefschetz fibration split-generate the wrapped Fukaya category
and satisfy this nondegeneracy condition (this is the subject of work in progress of Abouzaid--Ganatra). In particular, these statements combined would allow us to deduce (2) from (1).

\sss

\subsection{A Picard--Lefschetz approach}\label{subsec:a picard--lefschetz}

In this subsection we sketch a proof of Theorem \ref{thm:meta-principle} using techniques from symplectic Picard--Lefschetz theory. The main ingredient is \cite{seidelpart1}, and the proof will essential follow from suitably interpreting the main result of that paper. 
For completeness we also recall the definitions and main properties of 
the various players in this story.
We make use of various notions from $\calA_\infty$ algebra, such as $\calA_\infty$ bimodules over $\calA_\infty$ categories and functor categories between $\calA_\infty$ categories.

\subsubsection{The full Fukaya category of vanishing cycles $\euB$}

Let $(E^{2n},\Theta,\pi)$ be a Liouville Lefschetz fibration over $\D^2$.
Among other things, this means that $(E^{2n},\Theta)$ is a Liouville domain with corners
and $\Theta$ restricts to a Liouville form on the nonsingular fibers of $\pi: E \rightarrow \D^2$.
Let $\eta_1,...,\eta_k \subset \D^2$ be a basis of vanishing paths, with corresponding vanishing cycles
$V_1,...,V_k \subset E_*$ and thimbles $T_1,...,T_k \subset E$.

The first important algebraic object associated to a Lefschetz fibration is the 
{\em full Fukaya category of vanishing cycles}, which we denote by $\fuk(V_1,...,V_k)$ or simply $\euB$.
Namely, $\fuk(V_1,...,V_k)$ is the full $\calA_\infty$ subcategory of $\fuk(E_*,\Theta_*)$ with objects $V_1,...,V_k$. 
This of course depends on the ambient symplectic manifold $(E_*,\Theta_*)$ but we suppress it from the notation.

\subsubsection{The directed Fukaya category of vanishing cycles $\euA$}

Next, there is the {\em directed Fukaya category of vanishing cycles}, denoted by $\fukdir(V_1,...,V_k)$
or simply $\euA$. It is the subcategory of $\fuk(V_1,...,V_k)$ with objects $V_1,...,V_k$ and morphisms
\begin{itemize}
\item
$\hom_{\euA}(V_i,V_j) := \hom_{\euB}(V_i,V_j)$ if $i < j$
\item
$\hom_{\euA}(V_i,V_j) := \{0\}$ if $i > j$
\item
$\hom_{\euA}(V_i,V_i)$ is generated by a chosen cocycle representative of the cohomological unit in $\hom_{\euB}(V_i,V_i)$.
\end{itemize}
Equivalently, one can define the reduced version $\ovl{\euA}$ by omitting the unit morphisms in $\hom(V_i,V_i)$, and then reproduce $\euA$ up to quasi-isomorphism by formally adjoining strict units (see \cite[\S 2]{seidelpart1}).
This latter approach has the advantage that, assuming the vanishing cycles $V_1,...,V_k$ are in general position\footnote{We say that Lagrangians $V_1,...,V_k$ are in {\em general position} if any two intersect transversely and there are no triple intersections.}, we can define $\euA$ 
without using any Hamiltonian perturbations.
That is, we can construct $\ovl{\euA}$ following the general perturbation strategy for $\fuk$, but picking the Hamiltonians in Floer data and the Hamiltonian-valued one-forms in perturbation data to vanish identically.
This perspective is exploited in \cite{seidelpart1}.
We will assume from now on that $V_1,...,V_k$ are indeed in general position.

\subsubsection{The boundary map $\delta$ of the inclusion $\euA \rightarrow \euB$}\label{subsubsec:boundary map}

It turns out that a lot of the interesting symplectic geometry of $(E,\theta)$ is not contained in either of the abstract categories $\euA,\euB$ individually, but rather in their interaction via the inclusion map $\euA \rightarrow \euB$.
More precisely, restricting $\calA_\infty$ operations makes $\euB$ into an $\euA$-bimodule,
and we also have the diagonal $\euA$-bimodule, which we denote simply by $\euA$.
We can thus view $\euA \rightarrow \euB$ as a homomorphism of $\euA$-bimodules,
and for general reasons there is a quotient $\euA$-bimodule $\euB/\euA$ and a boundary homomorphism
$\delta: \euB/\euA \rightarrow \euA$ (well-defined up to homotopy).
By a version of Poincar\'e duality for Floer theory, $\euB/\euA$ is naturally identified with the dual diagonal bimodule $\calA^\vee$, and therefore we can also view $\delta$ as a bimodule homomorphism $\euA^\vee \rightarrow \euA$.
Our goal is to show that $\delta$ contains the information needed to produce the wrapped Floer cohomology $\field$-modules $HW(T_i,T_i)$ up to isomorphism.
In particular, this will imply that they only depend on $\euB$ up to order-preserving $\calA_\infty$ quasi-isomorphism.

\begin{remark}
As a side note, the relationship between $\delta$ and $\euD$ is explained in \cite{Seidelsubalgebras}.
We observe that $\delta$ induces a natural transformation between the two convolution functors
$\cdot\otimes \euB/\euA$ and $\cdot\otimes \euA$ from $\bimod(\euA)$ to itself. 
Since the latter convolution functor is quasi-equivalent to the identity functor,
we get a natural transformation $\EuScript{N}$ from $\cdot \otimes \euB/\euA$ to the identity. 
There is a notion of {\em localizing $\bimod(\euA)$ along the natural transformation $\EuScript{N}$}.
The main result of \cite{Seidelsubalgebras} states a precise sense in which $\euD$ gives a model for this localization.
In particular, this could be used to reformulate statements (1) and (2) above by replacing $\euD$ with the abstract localization of $\euA$.
\end{remark}

\subsubsection{The Fukaya category of thimbles $\euT$}\label{subsubsec:the fukaya category of thimbles}

So far we have discussed $\euA$ and $\euB$ as algebraic objects associated entirely to the fiber and the vanishing cycles contained it.
We now begin to relate these to invariants of the total space.
When considering pseudoholomorphic curves in $E$ with the thimbles $T_1,...,T_k$ as boundary conditions, we must exercise care in our choice of Hamiltonian terms in order to ensure the maximum principle.
Let $h: \D^2 \rightarrow \R$ be a once-wrapping Hamiltonian, say with $h$ being $C^\infty$ small away from $\bdy \D^2$
and linear with slope $1$ near $\bdy \D^2$.
Let $H = h \circ \pi$ be the pullback to $E$, with associated flow denoted by $\phi_H^t$.
As they stand, the projections $\pi(T_1),...,\pi(T_k)$ intersect nongenerically at $*$, but 
we can easily perturb away this issue using $\phi_H^t$.
Namely, fix small real numbers $0 < c_1 < ... < c_k < \e$, and set
$T_i' := \phi_H^{c_i}(T_i)$ for $i = 1,...,k$.
For generic choices of $c_1,...,c_k$, the perturbed thimbles $T_1',...,T_k'$ are in general position.

We can now construct a directed $\calA_\infty$ category, denoted by $\fukdir(T_1',...,T_k')$ or simply $\euT$, following same approach we took for $\euA$ with trivial Hamiltonian terms.
The techniques of \cite[\S 4,\S 5]{seidelpart1} guarantee the needed compactness in this setting by arguing via the projection $\pi$.
Strictly speaking we should first complete $(E,\Theta,\pi)$ to a Lefschetz fibration over $\C$, but we suppress this and other related technical details for ease of exposition.
Note that the pairwise intersections $T_i' \cap T_j'$ are naturally in bijection with the pairwise intersections $V_i \cap V_j$. In fact 
$\euT$ and $\euA$ are quasi-isomorphic as $\calA_\infty$ algebras, and even coincide on the nose for suitable choices. 
By an according abuse of notation, we will sometimes equate $\euT$ with $\euA$ in the sequel.

\subsubsection{The wrapped bimodules $\euU^c$}\label{subsubsec:the wrapped bimodules}

There is also a total space analogue of $\euB$, viewed as an $\euA$-bimodule, but this now depends on a choice of convention for how to handle ``intersection points at infinity" between the thimbles.
Following \cite{seidelpart1}, we introduce a family of $\euT$-bimodules $\euU^c$ depending on a real parameter $c \in \R$,
defined whenever $c$ is not an integer translate of $c_j - c_i$ for some $i \neq j$.
In general, recall that an $\calA_\infty$ bimodule $\calP$ over $\euA$ consists of a vector space $\euP(Y_0,X_0)$ 
for any two objects $X_0,Y_0$ of $\euA$ and multi-linear maps $\mu_{\euP}^{r|1|s}$ of the form
\begin{align*}
\hom_{\euA}(X_{r-1},X_r) \otimes ... \otimes \hom_{\euA}(X_0,X_1) \otimes \euP(Y_0,X_0) \otimes \hom_{\euA}(Y_1,Y_0) \otimes ... \otimes \hom_{\euA}(Y_s,Y_{s-1}) \rightarrow \euP(Y_s,X_r)
\end{align*}
for all $r,s \geq 0$ and objects $X_i,Y_i$ of $\euA$,
subject to suitable $\calA_\infty$ relations.
As a chain complex, we define $\euU^c(T_i',T_j')$ to be the Floer complex $CF(T_i',T_j')$ with respect to Floer data $(H_{i,j},J_{i,j})$, where $H_{i,j} = cg(t)H$ for $g$ a fixed nondecreasing function $g: [0,1] \rightarrow \R$ which vanishes near $0$ and $1$ and satisfies $\int_0^1g(t) = 1$.
The higher bimodule terms of $\euU^c$ are 
constructed using a slight modification of the usual perturbation scheme for $\fuk$.
Namely, we choose consistent perturbation data over the universal family of Riemann disks $S$ where:
\begin{itemize}
\item $S$ has $a+b+2$ punctures, with one puncture designated as the input and one designated as the output
\item the boundary segments between punctures are labeled by elements of $\{1,...,k\}$ (or equivalently $\{T_1',...,T_k'\}$),
where the labels increase as we follow the boundary orientation from the output to the input, and also from the input to the output.
\end{itemize}
So far this is just equivalent to a labeled version of $\calRuniv_{a+b+2}$.
Note that we can identify any such disk $S$ with the infinite strip $\R \times [0,1]$ with
$a$ punctures on $\R \times \{0\}$ and $b$ punctures on $\R \times \{1\}$.
We then pick the perturbation data on $S$ to be of the following form:
\begin{itemize}
\item
$K = H \otimes cg(t)dt$
\item
$J$ coincides with $J_{i_\pm,j_\pm}$ for $\pm s \gg 0$, where $(i_+,j_+)$ and $(i_-,j_-)$ are the labels at the input and output punctures respectively, and $J$ agrees with the corresponding choices made for $\euT$ near the remaining $p+q$ punctures.
\end{itemize}
The higher bimodule structure maps are then given by counting solutions to the inhomogeneous pseduoholomorphic curve equation with varying domain $S$ and perturbation data as above.
\begin{remark}
Our parameter $c$, which measures how much wrapping is taking place, is comparable to the parameter appearing in \cite[Def. 4.2]{seidelpart1}, the two setups being essentially equivalent after symplectically completing $\D^2$ to $\C$.
\end{remark}

\subsubsection{The continuation homomorphisms $\Gamma^{c_+,c_-}$}\label{subsubsec:continuation homo}

For $c_+ < c_-$, there is a $\euT$-bimodule homomorphism $\Gamma^{c_+,c_-}: \euU^{c_+} \rightarrow \euU^{c_-}$ which generalizes the usual continuation maps in Floer cohomology.
The construction is formally similar to that of $\euT$,
using a similar moduli space to the one in \S\ref{subsubsec:the wrapped bimodules} except that each strip $S \cong \R \times [0,1]$ is now also decorated with a sprinkle $p \in L \cong \R \times \{1/2\}$
(as in \S\ref{subsubsec:the moduli space calcont}, the effect of the sprinkle is to break the $\R$-translation symmetry).
On such a strip $S$, we assume the perturbation data is of form 
\begin{itemize}
\item
$K = H \otimes F(s)g(t)dt$, where $F(s) \equiv c_{\pm}$ for $\pm s \gg 0$, and $F'(s) \leq 0$ for all $s$
\item
$J$ agrees with the corresponding choices made for $\euU^{c_{\pm}}$ for $\pm s \gg \infty$ and those made for $\euT$ near the remaining $a+b$ punctures.
\end{itemize}
The specific form of $K$ and the inequality $F'(s) \leq 0$ effectively guarantee that the maximum principle still holds for solutions of the inhomogenous pseudoholomorphic curve equation.

As explained in \cite[\S 6d]{seidelpart1}, $\Gamma^{c,c}$ is homotopic to the identity and $\Gamma^{c_0,c_2}$ is homotopic to $\Gamma^{c_1,c_2}\circ \Gamma^{c_0,c_1}$, whenever these are defined. Moreover, $\Gamma^{c_+,c_-}$ is a quasi-isomorphism
provided that $\euU^c$ is defined for all $c \in [c_+,c_-]$.
This means that $\euU^c$ only changes for discrete values of $c$, and for $\e$ as in \S\ref{subsubsec:the fukaya category of thimbles} we have natural quasi-isomorphic identifications:
\begin{itemize}
\item
$\euU^{\e}$ with the diagonal $\euT$-bimodule $\euT$
\item
$\euU^{-\e}$ with the dual diagonal $\euT$-bimodule $\euT^\vee$.
\end{itemize}
By (pre)composing $\Gamma^{-\e,\e}$ with these identifications,
we get a $\euA$-bimodule homomorphism $\euA^\vee \rightarrow \euA$.
The main result of \cite{seidelpart1} states that this agrees with $\delta$ from $\S\ref{subsubsec:boundary map}$, at least up to homotopy and precomposing with a quasi-isomorphism from $\euA^\vee$ to itself.
For our purposes this result has the following significance. 
In general, let us say that two morphisms $f:X \rightarrow Y$ and $f': X' \rightarrow Y'$ in a strict category are {\em equivalent} if there are isomorphisms $\Phi: X \rightarrow X'$ and $\Psi: Y' \rightarrow Y$
such that $f = \Psi\circ f'\circ\Phi$.
Similarly, let us say that two closed morphisms $f: X \rightarrow Y$ and $f': X' \rightarrow Y'$ in a cohomologically unital $\calA_\infty$ category are {\em quasi-equivalent} if $[f]$ and $[f']$ are equivalent in the cohomology level category.
Then the main result of \cite{seidelpart1} implies that $\Gamma^{-\e,\e}$ and $\delta$ are quasi-equivalent morphisms in $\bimod(\euA)$

\subsubsection{The Fukaya--Seidel category and global monodromy}
The category $\euA$ sits inside of a bigger $\calA_\infty$ category $\fs(\pi)$, the Fukaya--Seidel category of the Lefschetz fibration $(E,\Theta,\pi)$.
Roughly, the objects of $\fs(\pi)$ are compact Lagrangians in $E$ along with Lefschetz thimbles for any possible choice of vanishing path.
As before, the noncompactness of the thimbles poses some additional technical difficulties, and Seidel circumvents these in \cite{seidelbook} using a branched double cover trick.
As part of the general package, the once-wrapping symplectomorphism $\phi_H^1$ of the total space induces a global monodromy functor 
\begin{align*}
\sigma: \fs(\pi) \rightarrow \fs(\pi),
\end{align*}
along with a continuation-type natural transformation $N$ from $\sigma$ to the identity functor $\identity$ of $\fs(\pi)$.
Let $\fun(\fs(\pi),\fs(\pi))$ denote the $\calA_\infty$ category whose objects are (cohomologically unital) functors $\fs(\pi) \rightarrow \fs(\pi)$ and morphisms are natural transformations (see \cite[\S 2e]{seidelbook}).
The natural transformation $N$ extends the once-wrapping continuation map $\Gamma^{-1+\e,\e}: \euU^{-1+\e} \rightarrow \euU^{\e}$ in the sense that it maps to the morphism $\Gamma^{-1+\e,\e}$ under the restriction functor
\begin{align*}
\EuScript{R_A}: \fun(\fs(\pi),\fs(\pi)) \rightarrow \bimod(\euA).
\end{align*}
Here $\EuScript{R_A}$ is the composition of:
\begin{enumerate}
\item
the natural functor $\fun(\fs(\pi),\fs(\pi)) \rightarrow \bimod(\fs(\pi))$ which on the level of objects is given by pulling back the diagonal $\fs(\pi)$-bimodule on the left side via the functor $\fs(\pi) \rightarrow \fs(\pi)$
\item the restriction functor $\bimod(\fs(\pi)) \rightarrow \bimod(\euA)$ which on the level of objects restricts the bimodule operations from $\fs(\pi)$ to $\euA$.
\end{enumerate}
We also claim that $\EuScript{R_A}$ is cohomologically full and faithful.
Indeed, in the composition above, the first functor is cohomologically full and faithful by \cite[Lemma 2.7]{abouzaid2015khovanov}.
The second functor is a quasi-equivalence by the fact that the thimbles $T_1,...,T_k$ generate
$\fs(\pi)$ (see \cite[Theorem 18.24]{seidelbook}), together with general Morita theory 
for $\calA_\infty$ bimodules (see \cite[\S 4.1]{sheridan2015formulae} and specifically the proof of \cite[Lem. A.3]{sheridan2015formulae}).

\subsubsection{The wrapped Fukaya category}\label{subsubsec:wrapped fuk}

The connection of the above discussion with the wrapped Fukaya category of the total space is as follows.
It is well-known that the wrapped Floer cohomology $\field$-module $HW(T_i,T_i)$ can be computed as a direct limit
\begin{align*}
HW(T_i,T_i) \cong \lim_{k \rightarrow \infty} HF(\phi_{H}^{k+\e} T_i, T_i),
\end{align*}
where the connecting maps in the directed system are continuation maps.
Compared with the general definition of wrapped Floer cohomology for Lagrangians in a Liouville manifold, the content of this statement is that it suffices to use a Hamiltonian which wraps only in the base direction of the Lefschetz fibration (see \cite{mcleanlefschetz} for the symplectic cohomology version).
Equivalently, this means that we have
\begin{align*}
HW(T_i,T_i) \cong \lim_{k \rightarrow \infty} Hhom_{\fs(\pi)}(\sigma^k T_i,T_i),
\end{align*}
where the connecting maps in the above direct limit are induced by precomposition with $N$.

\subsubsection{Putting it all together}

By the discussion in \S\ref{subsubsec:continuation homo},
$\Gamma^{-1+\e,\e}$ factors as the composition $\Gamma^{-\e,\e}\circ \Gamma^{-1+\e,-\e}$,
and $\Gamma^{-1+\e,-\e}$ is a quasi-isomorphism.
In particular, this shows that $\EuScript{R_A}(N) \simeq \Gamma^{-1+\e,\e}$ is quasi-equivalent to $\Gamma^{-\e,\e}$, and hence to $\delta$, as a morphism in $\bimod(\euA)$.
This has the following algebraic consequence.
Since $\EuScript{R_A}$ is cohomologically full and faithful, 
$\delta$ determines $N$ up to quasi-equivalence. 
In particular, at least after passing to cohomology level categories,
$N$ is determined as a natural transformation $\sigma \rightarrow \identity$
up to replacing $\sigma$ and $N$ by $\sigma'$ and $N'$ respectively, such that we have a cohomology level commutative diagram of the form
\begin{align*}
\xymatrix
{
\sigma' \ar^{N'}[rr]\ar_{F}^{\simeq}[d]\ar[d] && \identity\ar^{G}_{\simeq}[d]\\
\sigma \ar_{N}[rr] && \identity,
}
\end{align*}
where $F$ and $G$ are natural quasi-isomorphisms.
In particular, after replacing $F$ by $F\circ G^{-1}$, we can assume that $G$ is the identity natural transformation. 
It is then straightforward to check using the formulation from \S\ref{subsubsec:wrapped fuk} that the resulting $HW(T_i,T_i)$ is isomorphic whether we compute it using $N'$ or $N$.
In summary, it follows that $\delta$ determines $HW(T_i,T_i)$ up to isomorphism.

\subsection{Incorporating twistings and bulk deformations}\label{subsec:incorporating twistings}

As before, let $V_1,...,V_k$ be the vanishing cycles of a Liouville Lefschetz fibration $(E^{2n},\Theta,\pi)$. 
Suppose $\Omega$ is a closed two-form on the fiber whose support is disjoint from the vanishing cycles. In this case there is natural extension of $\Omega$ to a closed two-form $\wt{\Omega}$ on $E$. Namely, if we view $E$ as the result of attaching $k$ critical handles to $E_* \times \D^2$, then we take $\wt{\Omega}$ to be the pullback of $\Omega$ under the projection $E_* \times \D^2 \rightarrow E_*$, extended trivially over the critical handles.  We can further arrange, at least after a suitable deformation of $(E,\Theta,\pi)$, that the support of $\wt{\Omega}$ is disjoint from the Lefschetz thimbles $T_1,...,T_k$.
In this situation we have the following twisted analogue of Theorem \ref{thm:meta-principle}.
It can be proved following the same outline, {\em mutadis mutandis}, twisting fiber invariants by $\Omega$ and total space invariants by $\wt{\Omega}$:
\begin{thm}\label{thm:meta-principle twisted}
The wrapped Floer cohomology $\K$-modules $\twhwblank_{\wt{\Omega}}(T_i,T_i)$ depend only on $\twfuk(V_1,...,V_k)$ up to order-preserving quasi-isomorphism.
\end{thm}

Similarly, if $i_\mho: (\mho,\bdy \mho) \rightarrow (E_*\bdy E_*)$ is a smooth cycle of codimension $l > 2$ which is disjoint from the vanishing cycles,
there is a natural extension $i_{\wt{\mho}}: \wt{\mho} \rightarrow E$ to a codimension $l$
cycle in $E$, where $\wt{\mho} := \mho \times \D^2$ (modulo smoothing corners).
We can also assume that ${\wt{\mho}}$ is disjoint from the thimbles $T_1,...,T_k$.
The bulk deformed analogue of Theorem \ref{thm:meta-principle} in this situation is:
\begin{thm}\label{thm:meta-principle bulk deformed}
The wrapped Floer cohomology $\KL$-modules $\hw_{{\wt{\mho}}}(T_i,T_i)$ depend only on $\fuk_\mho(V_1,...,V_k)$ up to order-preserving quasi-isomorphism.
\end{thm}

Using the above two theorems,
we can now complete the proof of Theorem \ref{thm:compute twsh and bdfuk weak}
by applying either Theorem \ref{thm:main thm twisted} or Theorem \ref{thm:main thm bulk deformed} iteratively $k$ times, which results in an order preserving quasi-isomorphism
\begin{align*}
\fuk_\Omega(\tau_{S_1}^2V_1,...,\tau_{S_k}^2V_k) \simeq \fuk(V_1,...,V_k)
\end{align*}
in the case $\dim X = 4$, or
\begin{align*}
\fuk_\mho(\tau_{S_1}^2V_1,...,\tau_{S_k}^2V_k) \simeq \fuk(V_1,...,V_k)
\end{align*}
in the case $\dim X = 4l \geq 8$.
One then appeals to \S\ref{subsubsec:wrapped Floer as a module} to bootstrap from wrapped Floer cohomology to symplectic cohomology.
Similarly, Theorem* \ref{thm:compute twsh and bdfuk strong} follows from the following
 stronger versions of Theorem \ref{thm:meta-principle twisted} and Theorem
\ref{thm:meta-principle bulk deformed} (see the discussion in \S\ref{subsec:the general setup}).
\begin{thm*}
The twisted symplectic cohomology $\sh_{\wt{\Omega}}(E,\Theta)$ depends only on $\fuk_\Omega(V_1,...,V_k)$ up to order-preserving quasi-isomorphism.
\end{thm*}
\begin{thm*}
The bulk deformed symplectic cohomology $\sh_{{\wt{\mho}}}(E,\Theta)$ depends only on $\fuk_\mho(V_1,...,V_k)$ up to order-preserving quasi-isomorphism.
\end{thm*}


\bibliographystyle{plain}
\bibliography{biblio}

\end{document}